\documentclass[11pt]{amsart}

\usepackage{amsmath,amssymb}
\usepackage{verbatim}
\usepackage[colorlinks]{hyperref}
\usepackage[colorinlistoftodos]{todonotes}
\usepackage{tabmac}
\usepackage[margin=1in]{geometry}
\usepackage{tikz-cd}
\usepackage{pgf}
\usepackage{ytableau}

\usepackage{xcolor}
\newtheorem{thm}{Theorem}
\newtheorem{conj}[thm]{Conjecture}
\newtheorem{prop}[thm]{Proposition}
\newtheorem{proposition}[thm]{Proposition}
\newtheorem{lemma}[thm]{Lemma}

\newtheorem{cor}[thm]{Corollary}
\newtheorem{lem}[thm]{Lemma}
\newtheorem{problem}[thm]{Problem}
\theoremstyle{remark}

\newtheorem{rem}[thm]{Remark}
\newtheorem{remark}[thm]{Remark}
\newtheorem{example}[thm]{Example}
\newtheorem{definition}[thm]{Definition}

\numberwithin{equation}{section}
\numberwithin{thm}{section}

\newcommand{\bR}{\overleftarrow{R}}
\newcommand{\defn}[1]{{\it #1}}

\def\S{{\mathfrak S}}
\def\tS{{\tilde S}}
\def\Z{{\mathbb Z}}
\def\a{{\mathbf {a}}}
\def\b{{\mathbf{b}}}
\def\deg{{\rm deg}}
\def\af{{\rm af}}

\def\Q{{\mathbb Q}}
\def\C{{\mathbb C}}
\def\CP{{\mathbb{CP}}}
\def\A{{\mathbb A}}

\def\sp{{\rm span}}
\def\cS{{\bar S}}
\def\vdim{{\rm vdim}}
\newcommand{\hFl}{\widetilde{\Fl}}
\def\Fun{{\rm Fun}}

\def\Fl{{\rm Fl}}

\def\Gr{{\mathrm{Gr}}}
\def\hGr{{\widetilde {\Gr}}}
\def\GL{{\mathrm{GL}}}
\def\SL{{\mathrm{SL}}}

\newcommand{\bQ}{\mathbb{Q}}

\newcommand{\ex}{\mathrm{ex}}
\newcommand{\For}{\mathrm{For}}

\newcommand{\hLa}{\hat{\La}}
\newcommand{\id}{\mathrm{id}}
\newcommand{\la}{\lambda}
\newcommand{\La}{\Lambda}

\newcommand{\Schub}{\mathfrak{S}}
\newcommand{\bS}{\overleftarrow{\mathfrak{S}}}
\newcommand{\rot}{\mathrm{rot}}
\newcommand{\sh}{\gamma}
\newcommand{\shift}{\gamma}
\newcommand{\pt}{\mathrm{pt}}
\newcommand{\TZ}{T_{\mathbb{Z}}}

\def\P{{\mathbb P}}
\def\tP{{\tilde \P}}
\def\Par{{\mathbb Y}}
\def\cS{{\mathcal S}}

\def\u{{\mathbf u}}
\def\ta{{\tilde \a}}
\def\tj{\tilde j}
\def\ev{{\rm ev}}
\def\tA{{\tilde A}}
\def \tB{{\tilde B}}
\def\ttA{{\tilde \A}}
\def\Red{{\rm Red}}

\def\sh{{\rm sh}}
\def\SYT{{\rm SYT}}
\def\Des{{\rm Des}}
\def\Red{{\rm Red}}

\newcommand{\xm}{ x_-}
\newcommand{\xp} {x_+}
\newcommand{\am} {a_-}
\newcommand{\ap} {a_+}
\newcommand{\bc}{\overleftarrow{c}}

\newcommand{\Hom}{\mathrm{Hom}}
\newcommand{\leleft}[1]{\,{}^{#1}\!\!\le}
\newcommand{\leright}[1]{\le^{#1}}
\newcommand{\pair}[2]{\langle #1\,,\,#2\rangle}

\newcommand{\bA}{\mathbb{A}}

\newcommand{\raction}{\bullet}
\newcommand{\laction}{*}

\newcommand{\bb}{\overline{\beta}}

\def \i{{\mathbf{i}}}
\pgfmathsetmacro{\boxsize}{1}
\pgfmathsetmacro{\halfboxsize}{0.5}
\newcommand{\bbox}[2]{
\draw[thin] (#1,#2)-- (#1,#2+\boxsize)-- (#1+\boxsize,#2+\boxsize)-- (#1+\boxsize,#2)-- (#1,#2);
}
\newcommand{\elbow}[4]{
\draw[#3] (#1+\boxsize,#2) ++ (90:\halfboxsize) arc (90:180:\halfboxsize);
\draw[#4] (#1,#2+\boxsize) ++ (0:\halfboxsize) arc (0:-90:\halfboxsize);
}
\newcommand{\rightelbow}[3]{
\draw[#3] (#1+\boxsize,#2) ++ (90:\halfboxsize) arc (90:180:\halfboxsize);
}
\newcommand{\leftelbow}[3]{
\draw[#3] (#1,#2+\boxsize) ++ (0:\halfboxsize) arc (0:-90:\halfboxsize);
}
\newcommand{\cross}[4]{
\draw[#3] (#1+\halfboxsize,#2) --  (#1+\halfboxsize,#2+\boxsize);
\draw[#4] (#1,#2+\halfboxsize) --  (#1+\boxsize,#2+\halfboxsize);
}
\newcommand{\horline}[3]{
\draw[#3] (#1,#2+\halfboxsize) --  (#1+\boxsize,#2+\halfboxsize);
}
\newcommand{\vertline}[3]{
\draw[#3] (#1+\halfboxsize,#2) --  (#1+\halfboxsize,#2+\boxsize);

}

\def\half{blue}
\def\inactive{black}

\def\wt{{\rm wt}}
\def\tf{{\tilde f}}

\def\hs{{\hat s}}

\newcommand{\seungjin}[1]{\todo[size=\tiny,color=green!30]{#1 \\ \hfill --- Seung Jin}}

\newcommand{\thomas}[1]{\todo[size=\tiny,color=red!30]{#1
      \\ \hfill --- Thomas}}
\newcommand{\Thomas}[1]{\todo[size=\tiny,inline,color=red!30]{#1
      \\ \hfill --- Thomas}}

\begin{document}
\title{Back stable Schubert calculus}
\author{Thomas Lam}
\address{Department of Mathematics \\
University of Michigan \\
530 Church St. \\
Ann Arbor 48109 USA}
\email{tfylam@umich.edu}
\thanks{T.L. was supported by NSF DMS-1464693}
\author{Seung Jin Lee}
\address{Department of Mathematical Sciences \\ Research institute of Mathematics \\ Seoul National University \\ Gwanak-ro 1, Gwanak-gu \\ Seoul 151-747 Republic of Korea}
\email{lsjin@snu.ac.kr}
\thanks{S. J. Lee was supported by the National Research Foundation of Korea(NRF) grant funded by the Korea government(MSIT) (No. 2019R1C1C1003473).}
\author{Mark Shimozono}
\address{Department of Mathematics \\
460 McBryde Hall, Virginia Tech\\
 255 Stanger St. \\
Blacksburg, VA, 24601, USA }
\email{mshimo@math.vt.edu}
\thanks{M.S. was supported by NSF DMS-1600653}

\setcounter{tocdepth}{1} 

\begin{abstract}
We study the back stable Schubert calculus of the infinite flag variety.  Our main results are:
\begin{itemize}
\item
a formula for back stable (double) Schubert classes expressing them in terms of a symmetric function part and a finite part;
\item
a novel definition of double and triple Stanley symmetric functions;
\item
a proof of the positivity of double Edelman-Greene coefficients generalizing the results of Edelman-Greene and Lascoux-Sch\"utzenberger;
\item
the definition of a new class of {\it bumpless} pipedreams, giving new formulae for double Schubert polynomials, back stable double Schubert polynomials, and a new form of the Edelman-Greene insertion algorithm;
\item
the construction of the Peterson subalgebra of the infinite nilHecke algebra, extending work of Peterson in the affine case;
\item
equivariant Pieri rules for the homology of the infinite Grassmannian;
\item
homology divided difference operators that create the equivariant homology Schubert classes of the infinite Grassmannian.
\end{itemize}
\end{abstract}
\maketitle
\tableofcontents
\section{Introduction}
\subsection{Flag varieties and Schubert polynomials}
The flag variety $\Fl_n$ is the smooth projective algebraic variety classifying full flags inside an $n$-dimensional complex vector space $\C^n$.  The cohomology ring $H^*(\Fl_n)$ was determined by Borel \cite{Bor}: it is the quotient of the polynomial ring $\Q[x_1,\ldots,x_n]$ by the ideal generated by symmetric functions in $x_1,\ldots,x_n$ of positive degree.

The flag variety has a distinguished stratification by Schubert varieties, and the cohomology classes of Schubert varieties form a basis of $H^*(\Fl_n)$, called the {\it Schubert basis}.  Bernstein, Gelfand, and Gelfand \cite{BGG} and Demazure \cite{Dem} found formulae for the Schubert basis in terms of divided difference operators.  Lascoux and Sch\"utzenberger \cite{LS} defined and studied polynomial representatives for the Schubert classes, called the \defn{Schubert polynomials} $\S_w \in \Q[x_1,\ldots,x_n]$.  Lascoux and Sch\"utzenberger  furthermore defined the {\it double Schubert polynomials} $\S_w(x;a)$ that represent Schubert classes in the torus-equivariant cohomology ring $H^*_T(\Fl_n)$.  

There is a rich combinatorial theory for Schubert polynomials.  Among the fundamental results crucial to us is the formula of Billey-Jockusch-Stanley \cite{BJS} for the monomial expansion of $\S_w$.

%

\subsection{Back stable Schubert polynomials}
In this work, we consider limits of Schubert polynomials called \defn{back stable Schubert polynomials} $$\bS_w := \lim_{\substack{p\to-\infty \\q\to\infty}} \S_w(x_p,x_{p+1},\ldots,x_q),$$ for $w\in S_\Z$, the group of permutations of $\Z$ moving finitely many elements. Two of us (T. L. and M. S.) first learnt of this construction from Allen Knutson \cite{K:backstable}. Anders Buch \cite{Bu} was also aware of how to back stabilize (double) Schubert polynomials. Finally, one of us (S.-J. Lee) found them on his own independently.  

Define the ring of back symmetric formal power series $$\bR := \Lambda \otimes \Q[\ldots,x_{-1},x_0,x_1,\ldots]$$ where $\Lambda$ denotes the symmetric functions in $\ldots,x_{-1},x_0$.  In Theorem~\ref{T:backstable basis}, we show that the back stable Schubert polynomials $\bS_w$ form a basis of the ring $\bR$.  As far as we are aware, the ring $\bR$ has not previously been explicitly studied.

\subsection{Coproduct formula}
Stanley \cite{Sta} defined the Stanley symmetric functions $F_w \in \Lambda$, for $w \in S_\Z$ to study the enumeration of reduced words of permutations.  It is well-known that the symmetric functions $F_w$ can be obtained as ``forward limits" of the Schubert polynomials $\S_w$.  We give a new construction of $F_w$ from back stable Schubert polynomials.  Namely, we define a natural algebra homomorphism $\eta_0: \bR \to \Lambda$ and show that Stanley's definition of $F_w$ agrees with $\eta_0(\S_w)$. This is closely related to, and explains, a formula of Li \cite{Li}.  In contrast, the map sending $\S_w$ to $F_w$ is not multiplicative.

We prove that back stable Schubert polynomials satisfy the ``coproduct formula" (Theorem \ref{thm:coprod}) 
\begin{equation}\label{eq:coprodintro}\bS_w = \sum_{w \doteq uv} F_u \otimes \S_v
\end{equation} where 
$w\doteq uv$ denotes a length-additive factorization such that $v$ is a permutation not using the reflection $s_0$.  The coproduct formula decomposes $\bS_w$ into a ``symmetric" part and a ``finite polynomial" part.   We do not know of an analogue of the coproduct formula for finite Schubert polynomials.

\subsection{Double Stanley symmetric functions}
\defn{Back stable double Schubert polynomials} $\bS_w(x;a)$ can also be defined in a similar manner (though the existence of the limit is less clear; see Proposition~\ref{P:backstable double well-defined}), and we show (Theorem~\ref{thm:backstabledouble}) that they form a basis of the \defn{back symmetric double power series ring} $\bR(x;a) :=\Lambda(x||a) \otimes_{\Q[a]} \Q[x,a]$, where $\Q[x,a]:=\Q[x_i,a_i\mid i\in \Z]$ and $\Lambda(x||a)$ is the ring of double symmetric functions.  The ring $\Lambda(x||a)$ is the polynomial $\Q[a]=\Q[\dotsc,a_{-1},a_0,a_1,\dotsc]$-algebra generated by the double power sums $p_k(x||a):=\sum_{i \leq 0} x_i^k- \sum_{i \leq 0} a_i^k$.  The ring $\Lambda(x||a)$ is a $\Q[a]$-Hopf algebra with basis the double Schur functions $s_\lambda(x||a)$, and is studied in detail by Molev \cite{M}. 

Generalizing $\eta_0$, there is an algebra homomorphism $\eta_a: \bR(x;a) \to \Lambda(x||a)$.  We define the \defn{double Stanley symmetric functions} $F_w(x||a) \in \La(x||a)$ by $F_w(x||a):= \eta_a(\bR(x;a))$.  As far as we are aware, the symmetric functions $F_w(x||a)$ are novel. When $w$ is 321-avoiding, the double Stanley symmetric function is equal to the skew double Schur function which was studied by Molev \cite{M}; see Proposition \ref{P:skew double is 321 double Stanley}.

One of our main theorems (Theorem \ref{thm:doubleStanleypositivity}) is a proof that the \defn{double Edelman-Greene coefficients} $j_\la^w(a) \in \Q[a]$ given by the expansion of double Stanley symmetric functions
$$
F_w(x||a) = \sum_{\la} j_\la^w(a) s_\la(x||a)
$$
into double Schur functions $s_\la(x||a)$, are positive polynomials in certain linear forms $a_i - a_j$.  The usual Edelman-Greene coefficients $j_\la^w(0):=j_\la^w(a)|_{a_i \to 0}$ are known to be positive by the influential works of Edelman and Greene \cite{EG} and Lascoux and Sch\"utzenberger \cite{LS2}.  Molev 
\cite{M} has given a combinatorial rule for the expansion coefficients of skew double Schurs into double Schurs (that is, for $j_\la^w(a)$ where $w$ is $321$-avoiding) but it does not exhibit the above positivity.

Back stable double Schubert polynomials satisfy (Theorem~\ref{thm:doublecoprod}) the same kind of coproduct formula \eqref{eq:coprodintro} as the non-doubled version, with the \defn{double Stanley symmetric functions} $F_w(x||a)$ replacing $F_w$ and double Schubert polynomials $\S_w(x;a)$ replacing the usual finite Schubert polynomials $\S_w$.

\subsection{Bumpless pipedreams}
We introduce a combinatorial object called \defn{bumpless pipedreams}, to study the monomial expansion of back stable double Schubert polynomials.  These are pipedreams where pipes are not allowed to bump against each other, or equivalently, the ``bumping" or ``double elbow tile" is forbidden:  
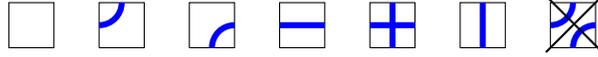
\begin{center}
\begin{tikzpicture}[scale=0.6,line width=0.8mm]
\bbox{-3}{0}
\bbox{-1}{0}
\leftelbow{-1}{0}{blue}
\bbox{1}{0}
\rightelbow{1}{0}{blue}
\bbox{3}{0}
\horline{3}{0}{blue}
\bbox{5}{0}
\cross{5}{0}{blue}{blue}
\bbox{7}{0}
\vertline{7}{0}{blue}
\bbox{9}{0}
\leftelbow{9}{0}{blue}
\rightelbow{9}{0}{blue}
\draw[thick,red] (8.9,-0.1)--(10.15,1.15);
\draw[thick,red] (10.15,-0.1)--(8.9,1.15);
\end{tikzpicture}
 \end{center}
Using bumpless pipedreams, we obtain:
\begin{itemize}
\item An expansion for double Schubert polynomials $\S_w(x;a)$ in terms of products of binomials $\prod (x_i - a_j)$. Our formula is different from the classical pipe-dream formula of Fomin and Kirillov \cite{FK} for double Schubert polynomials: unlike theirs, our formula is obviously back stable. Hence we also obtain such an expansion for back stable double Schubert polynomials.
\item A positive expression for the coefficient of $s_\la(x||a)$ in $\bS(x;a)$ (Theorem \ref{thm:decomp})
\item A new combinatorial interpretation of Edelman-Greene coefficients $j_\la^w(0)$ as the number of certain EG pipedreams (Theorem \ref{thm:EGStanley}).
\end{itemize}
Our bumpless pipedreams are a streamlined version of the \defn{interval positroid pipedreams} defined by Knutson \cite{Knu}.  Heuristically, our formula for $\bS_w(x;a)$ is obtained by ``pulling back" a Schubert variety in $\Fl$ to various Grassmannians where it can be identified (after equivariant shifts) with \defn{graph Schubert varieties}, a special class of positroid varieties. This connects our work with that of Knutson, Lam, and Speyer \cite{KLS}, who identified the equivariant cohomology classes of positroid varieties with affine double Stanley symmetric functions.

When presenting our findings we were informed by Anna Weigandt\footnote{See the recent preprint \cite{Wei}.} that Lascoux's use \cite{L2} of alternating sign matrices (ASMs) in a formula for Grothendieck polynomials, is very close to our pipedreams; ours correspond to the subset of reduced ASMs. Our construction has the advantage that the underlying permutation is evident; in the ASM one must go through an algorithm to extract this information.  Lascoux's ASMs naturally compute in $K$-theory rather than in cohomology.

\subsection{Infinite flag variety}
Whereas Schubert polynomials represent Schubert classes in the cohomology of the flag variety, back stable Schubert polynomials represent Schubert classes in the cohomology of an appropriate {\it infinite flag variety}.

The infinite Grassmannian $\Gr$ is an ind-finite variety over $\C$, the points of which are identified with (infinite-dimensional over $\C$) \defn{admissible subspaces} $\Lambda \subset F$, where $F = \C((t))$ (see Section \ref{S:geometry}).  The infinite Grassmannian can be presented as an infinite union of finite-dimensional Grassmannians.
The infinite flag variety $\Fl$ is an ind-finite variety over $\C$, the points of which are identified with \defn{admissible flags}
$$
\Lambda_\bullet = \{\cdots \subset \Lambda_{-1} \subset \Lambda_0 \subset \Lambda_1 \subset \cdots\}.
$$
Under an isomorphism between $\bR$ and the cohomology of $\Fl$, we show in Theorem \ref{thm:HFl} that back stable Schubert polynomials represent Schubert classes of $\Fl$.  For the infinite Grassmannian it is well-known that Schur functions represent Schubert classes.  Our $\Fl$ differs somewhat from other infinite-dimensional flag varieties we have seen in the literature (see for example \cite{PS}), and thus we give a reasonably independent development in Section~\ref{S:geometry}.  

The infinite flag variety $\Fl$ is the union of finite-dimensional flag varieties, and any product $\xi^x \xi^y$ of two Schubert classes $\xi^x,\xi^y \in H^*(\Fl)$ can be computed within some finite-dimensional flag variety.  Naively, as some subset of the authors had mistakenly assumed, no interesting and new phenomena would arise in the infinite case.  To the contrary, in this article we present our findings of entirely new phenomena that have no classical counterpart.

\subsection{Localization and infinite nilHecke algebra}
The torus-equivariant cohomology $H^*_T(\Fl_n)$ of the flag variety can be studied by localizing to the torus fixed points, giving an injection $H^*_T(\Fl_n) \hookrightarrow \bigoplus_{v \in S_n} H^*_T(\pt) \simeq \Q[a_1,\ldots,a_n]$.  It is known \cite[Remark 1]{Bil} that the localization $\xi^v|_w$ of a Schubert class indexed by $v \in S_n$ at the torus fixed-point indexed $w \in S_n$ is given by the evaluation $\S_v(wa;a) \in \Q[a]$.  We prove in Proposition~\ref{prop:bAbR} an analogous result for the equivariant cohomology ring $H^*_{T_\Z}(\Fl)$: the localization of a Schubert class $\xi^v$ at a $T_\Z$-fixed point $w \in S_\Z$ is equal to a specialization $\bS_v(wa;a)$ of the back stable double Schubert polynomial. 

Kostant and Kumar \cite{KK} studied the torus-equivariant cohomology of Kac-Moody flag varieties (including the usual flag variety) using the action of the {\it nilHecke ring} on these cohomologies.  We construct in Section~\ref{sec:local} an action of the infinite nilHecke ring $\A'$ on $H^*_{T_\Z}(\Fl)$, giving an infinite rank variant of the results of Kostant and Kumar.

\subsection{Homology}
The torus-equivariant cohomology ring $H_{T_\Z}^*(\Gr)$ of the infinite Grassmannian is isomorphic to the ring $\La(x||a)$ of double symmetric functions (see Theorem~\ref{thm:HTFl}).
The (appropriately completed) equivariant homology $H_*^{T_\Z}(\Gr)$ of the infinite Grassmannian is Hopf-dual to the Hopf algebra $\La(x||a)$. Non-equivariantly, this can be explained by the homotopy equivalence $\Gr \cong \Omega SU(\infty)$ with a group.  Restricting to a one-dimensional torus $\C^\times \subset T_\Z$, the multiplication of $H_*^{\C^\times}(\Gr)$ is induced by the direct sum operation on finite Grassmannians, and was studied in some detail by Knutson and Lederer \cite{KL}.  The geometry of the full multiplication on $H_*^{T_\Z}(\Gr)$ is still mysterious to us, and we hope to study it in the context of the {\it affine infinite Grassmannian} in the future.

Molev \cite{M} studied the Hopf algebra $\hLa(y||a)$ Hopf-dual to $\La(x||a)$, and defined the basis $\hs_\la(y||a)$ of {\it dual Schur functions} in $\hLa(y||a)$, dual to the double Schur functions.  We identify (Proposition \ref{prop:dualSchur}) the Schubert basis of $H_*^{T_\Z}(\Gr)$ with Molev's dual Schur functions $\hs_\la(y||a)$ \cite{M}.  We use this to resolve (Theorem \ref{thm:KL}) a question posed in \cite{KL}: to find deformations of Schur functions that have structure constants equal to the Knutson-Lederer direct sum product.

One of our main results (Theorem \ref{T:create dual Schur}) is a recursive formula for the dual Schur functions $\hs_\la(x||a)$ in terms of novel \defn{homology divided difference operators}, which are divided difference operators on equivariant variables, but conjugated by the equivariant Cauchy kernel. 
A similar formula had previously been found independently by Naruse \cite{Na}, who was studying the homology of the infinite Lagrangian Grassmannian.  Our construction is also closely related to the presentation of the equivariant homology of the affine Grassmannian given by Bezrukavnikov, Finkelberg, and Mirkovic \cite{BFM}.  We hope to return to the affine setting in the future.

We compute the ring structure of this equivariant homology ring by giving a positive Pieri rule (Theorem \ref{thm:homologyhook}).  Our computation of the Pieri structure constants relies on some earlier work of Lam and Shimozono \cite{LaSh} in the affine case, and on \defn{triple Stanley symmetric functions} $F_w(x||a||b)$ that we define in Section~\ref{sec:triple}.  The double Stanley symmetric functions $F_w(x||a)$ are recovered from $F_w(x||a||b)$ by setting $b = a$.  The triple Stanley symmetric functions distinguish ``stable" phenomena from ``unstable" phenomena in the limit from the affine to the infinite setting.  
%

\subsection{Affine Schubert calculus}
Our study of back stable Schubert calculus is to a large extent motivated by our study of the Schubert calculus of the affine flag variety $\hFl$, and in particular Lee's recent definition of \defn{affine Schubert polynomials} \cite{Lee}.  There is a surjection $H^*(\Fl) \to H^*(\hFl_n)$ from the cohomology of the infinite flag variety to that of the affine flag variety of $\SL(n)$.  A complete understanding of this map yields a presentation for the cohomology of the affine flag variety. Thus this project can be considered as a first step towards understanding the geometry and combinatorics of affine Schubert polynomials and their equivariant analogues. 

We shall apply back stable Schubert calculus to affine Schubert calculus in future work \cite{LLS:affine}.  In particular, analogues of our coproduct formulae (Theorems \ref{thm:coprod} and \ref{thm:doublecoprod}) hold for equivariant Schubert classes in the affine flag variety of any semisimple group $G$ \cite{LLS:coprod}.  

\subsection{Peterson subalgebra}
The (finite) torus-equivariant cohomology ring $H^*_T(\hFl_n)$ of the affine flag variety $\hFl_n$ has an action of the level zero affine nilHecke ring $\ttA$.
Peterson \cite{Pet, Lam} constructed a subalgebra $\tP \subset \ttA$ (recalled in Appendix~\ref{A:affinenilHecke}) and showed that the torus-equivariant homology $H_*^T(\hGr_n)$ of the affine Grassmannian $\hGr_n$ is isomorphic to $\tP$.  We refer the reader to \cite{LLMSbook} for an introduction to affine Grassmannian Schubert calculus.

While Kostant and Kumar's definition of the nilHecke algebra applies to any Kac-Moody flag variety, the definition of the Peterson algebra is special to the case of the affine flag variety (of a semisimple group).
Thus it came as a surprise that we are able to construct (Theorem \ref{thm:PetersonHopf}) a subalgebra $\P' \subset \A'$ of the infinite nilHecke ring that is an analogue of the Peterson subalgebra in the affine case.  While the infinite symmetric group $S_\Z$ is not an affine Coxeter group, we are able to define elements in $\A'$ that behave like translation elements in affine Coxeter groups.  

Our infinite Peterson algebra $\P'$ is in a precise sense the limit of Peterson algebras for affine type $A$.  This allows us to apply known positivity results in affine Schubert calculus to deduce the positivity (Theorem \ref{thm:doubleStanleypositivity}) of double Edelman-Greene coefficients.

\subsection{Other directions}
Most of the results of the present work have $K$-theoretic analogues.  We plan to address $K$-theory in a separate work \cite{LLS:K}.

The results in this paper (for example, \textsection \ref{SS:translation}) suggests the study of the affine infinite flag variety $\hFl$, an ind-variety whose torus-fixed points are the affine infinite symmetric group $S_\Z \ltimes Q_{\Z}^\vee$, where $Q_\Z^\vee$ is the $\Z$-span of root vectors $e_i-e_j$ for $i\ne j$ integers and $e_i$ is the standard basis of a lattice with $i\in \Z$.  Curiously, Schubert classes of $\hFl$ can have infinite codimension (elements of $S_\Z \ltimes Q_{\Z}^\vee$ can have infinite length) and should lead to new phenomena in Schubert calculus.

\subsection*{Acknowledgements}
We thank Anna Weigandt, Zach Hamaker, Anders Buch, and especially Allen Knutson for their comments on and inspiration for this work.  We also thank the referee for a number of helpful comments and suggestions.

\section{Schubert polynomials}
We recall known results concerning Lascoux and Sch\"utzenberger's (double) Schubert polynomials.  None of the results in this section are new, but for completeness we provide short proofs for many of them.
\subsection{Notation}\label{SS:notation}
Throughout the paper, we set  $ \chi(\text{True})=1$ and $\chi(\text{False})=0$.
\subsubsection{Permutations}
Let $S_\Z$ denote the subgroup of permutations of $\Z$ generated by $s_i$ for $i\in\Z$ where $s_i$ exchanges $i$ and $i+1$. This is the group of permutations of $\Z$ that move finitely many elements. 
Let $S_+$ (respectively $S_-$, resp. $S_n$) be the subgroup of $S_\Z$ generated by $s_1, s_2,\ldots$ (resp. $s_{-1},s_{-2},\ldots$, resp. $s_1,s_2,\dotsc,s_{n-1}$).
We have $S_+ = \bigcup_{n\ge1} S_n$.  We write $S_{\ne0} = S_-\times S_+$. 
For $w\in S_\Z$ denote by $\ell(w)$ the length of $w$ and $\Red(w)$ for the set of reduced words of $w$ \cite[\textsection 1.6]{Hum}.
For $x,y,z \in S_\Z$, we write $z \doteq xy$ if $z = xy$ and $\ell(z) = \ell(x) + \ell(y)$.  This notation generalizes to longer products $z \doteq x_1 x_2 \cdots x_r$.
Let $w_0^{(n)}\in S_n$ be the longest element \cite[\textsection 1.8]{Hum}.
Let $\shift:S_\Z\to S_\Z$ be the ``shifting" automorphism $\shift(s_i)=s_{i+1}$ for all $i\in\Z$. 

Let $\le$ be the (strong) Bruhat order on $S_\Z$ \cite[\textsection 5.9]{Hum}.
For a fixed $k\in\Z$, say that $w\in S_\Z$ is $k$-Grassmannian 
if $w<ws_i$ (equivalently, $w(i) < w(i+1)$ viewing $w$ as a function $\Z\to\Z$) for all $i\in\Z-\{k\}$.
We write $S_\Z^0$ for the set of $0$-Grassmannian permutations. 

\subsubsection{Partitions}
Let $\Par$ denote the set of partitions or Young diagrams. We consider a partition $\lambda=(\lambda_1,\ldots,\lambda_\ell)$ as an infinite sequence $(\lambda_1,\ldots,\lambda_\ell,0,0,\ldots)$ if necessary. Throughout the paper, Young diagrams are drawn in English notation: the boxes are top left justified in the plane. For a Young diagram $\la$, we let $\la'$ denote the conjugate (or transpose) Young diagram.  The \defn{dominance order} on partitions of the same size is given by $\lambda \leq \mu$ if $\sum_{i=1}^k \lambda_i \leq \sum_{i=1}^k \mu_i$ for all $k$.

There is a bijection between $\Par$ and $S_\Z^0$, given by $\lambda \mapsto w_\lambda$, where
\begin{align}
\label{E:wla}
	w_\lambda(i) &:= i + \begin{cases}
		\lambda_{1-i} & \text{if $i\le0$} \\
		-\lambda'_i & \text{if $i>0$.}
	\end{cases}
\end{align}
A reduced expression for $w_\la$ is obtained by labeling the box $(i,j)$ in the $i$-th row and $j$-th column of the diagram of $\la$ by $s_{j-i}$ and reading the rows from right to left starting with the bottom row.

If $\mu\subset \la$, we define 
\begin{align}\label{E:w skew}
	w_{\la/\mu} := w_\la w_\mu^{-1}.
\end{align}
We note that $w_\la \doteq w_{\la/\mu} w_\mu$. An element $w\in S_\Z$ is 
\defn{321-avoiding} if there is no triple of integers $i<j<k$ such that $w(i)>w(j)>w(k)$.

\begin{lemma}[{\cite[Section 2]{BJS}}]
	An element $w\in S_\Z$ is $321$-avoiding if and only if $w=w_{\la/\mu}$ for some partitions $\mu\subset\la$. 
\end{lemma}

\begin{example} \label{X:grass perm} For $\la=(3,2)$, the values of $w_\la:\Z\to\Z$  are given. For $\mu=(1)$ we have $w_\mu=s_0$. Reduced decompositions for $w_\la$ and $w_{\la/\mu}$ are given.
\begin{align*}
	\begin{array}{|c||c|c|c|c|c||c|c|c|c|c|c|c|} \hline
			i                 & \dotsm & -3 & -2 & -1 &   0 &   1 &   2 &   3 & 4 & 5 & \dotsm \\ \hline
			w_{\la}(i)    & \dotsm & -3 & -2 &   1 &   3 & -1 &   0 &   2 & 4 & 5 & \dotsm  \\ \hline 
			w_{\la}(i)-i & \dotsm &   0 &   0 &   2 &   3 & -2 & -2 & -1 & 0 & 0 & \dotsm \\ \hline
	\end{array}
\end{align*}

\begin{align*}
	&\begin{ytableau} 
		   s_0&s_1&s_2 \\ s_{-1} & s_0 
	   \end{ytableau} \,\,w_{(3,2)} = (s_0s_{-1})( s_2s_1s_0) 
    &\begin{ytableau} *(black) s_0 & s_1 & s_2 \\
		s_0 & s_{-1} \end{ytableau} \,\, w_{(3,2)/(1)} = (s_0s_{-1})(s_2s_1). 
\end{align*}
\end{example}

\subsection{Schubert polynomials} \label{SS:Schub def}
Following \cite{LS}, we define Schubert polynomials using divided difference operators.
Let $\Q[\xp]:= \Q[x_1,x_2,x_3,\ldots]$ be the polynomial ring in infinitely many positively-indexed variables and $\Q[x]:=\Q[\dotsc,x_{-1},x_0,x_1,\dotsc]$ the polynomial ring in variables indexed by integers.  Define the $\Q$-algebra automorphism $\shift: \Q[x] \to \Q[x]$ given by $x_i \mapsto x_{i+1}$.

For $i\in\Z$ the divided difference operator $A_i: \Q[x]\to\Q[x]$ is defined by
\begin{equation}\label{eq:dd}
A_i(f) := \dfrac{f - s_i(f)}{x_i - x_{i+1}}.
\end{equation}
We have the operator identities
\begin{align}\label{E:dd2}
	A_i^2 &= 0 \\
	\label{E:dd commute}
	A_iA_j &= A_j A_i \qquad\text{for $|i-j|>1$} \\
	\label{E:dd braid}
	A_i A_{i+1} A_i &= A_{i+1} A_i A_{i+1}.
\end{align}

For $w\in S_\Z$ this allows the definition of
\begin{align}\label{E:ddiff w}
	A_w &:= A_{i_1} A_{i_2}\dotsm A_{i_\ell}
&\qquad&\text{where $(i_1,i_2,\dotsc,i_\ell)\in\Red(w)$.} 
\end{align}

\begin{lemma} \label{L:ker partial} Both the kernel of $A_i$ and the
	image of $A_i$ are the subalgebra of $s_i$-invariant elements.
\end{lemma}

For $w\in S_n$, the Schubert polynomial $\S_w\in \Q[\xp]$ is defined by
\begin{align}
\label{E:Schub long}
  \S_{w_0^{(n)}}(\xp) &:= x_1^{n-1}x_2^{n-2}\dotsm x_{n-1}^1  \\
\label{E:Schub nonlong}
  \S_w(\xp) &:= A_i \S_{ws_i}(\xp)\qquad\text{for any $i$ with $ws_i>w$.}
\end{align}
The polynomials $\S_w(\xp)$ are well-defined for $w\in S_n$ by \eqref{E:dd commute} and \eqref{E:dd braid}.

\begin{lem} \label{L:S_+ well-defined}
	$\S_w(\xp)$ is well-defined for $w\in S_+$.
\end{lem}
\begin{proof} It suffices to show that the definitions of $\S_{w^{(n)}_0}$ and $\S_{w^{(n+1)}_0}$ are consistent.
Using $w_0^{(n+1)}\doteq w_0^{(n)} s_n \dotsm s_2 s_1$ we have
$A_n \dotsm A_2 A_1(x_1^nx_2^{n-1}\dotsm x_{n}^1) = x_1^{n-1}x_2^{n-2}\dotsm x_{n-1}^1$.
\end{proof}

We recall the monomial expansion of $\S_w$ due to Billey, Jockusch, and Stanley.

\begin{thm}[\cite{BJS}]\label{thm:BJS}
For $w \in S_+$, we have
\begin{align}\label{E:BJS}
	\S_w (\xp)= \sum_{a_1a_2 \cdots a_\ell \in \Red(w)} \sum_{\substack{1 \leq b_1 \leq b_2 \leq \cdots \leq b_\ell\\ a_i<a_{i+1} \implies b_i < b_{i+1} \\ b_i \leq a_i}} x_{b_1} x_{b_2} \cdots x_{b_\ell}.
\end{align}
\end{thm}

Define the \defn{code} $c(w) = (\dotsc,c_{-1},c_0,c_1,\dotsc) $ of $w\in S_\Z$ by 
\begin{align}\label{E:code}
	c_i &:= |\{j>i \mid w(j)<w(i)\}|.
\end{align}
The support of an indexed collection of integers $(c_i\mid i\in J)$ is the set of $i\in J$ such that $c_i\ne0$. The code gives a bijection from $S_\Z$ to finitely-supported sequences of nonnegative integers $(\dotsc,c_{-1},c_0,c_1,\dotsc)$.
It restricts to a bijection from $S_+$ to finitely-supported sequences of nonnegative integers $(c_1,c_2,\dotsc)$.

For a sequence $b=(b_1,b_2,b_3,\cdots)$ of integers, let $x^b$ denote $x_1^{b_1}x_2^{b_2}\cdots$. For two monomials $x^b$ and $x^c$ in $\Q[x]$,
we say that $x^c>x^b$ in reverse-lex order if $b\ne c$ and for the maximum $i\in \Z$ such that $b_i\ne c_i$ we have $b_i<c_i$.  The following triangularity of Schubert polynomials with monomials can be seen from Bergeron and Billey's rc-graph formula for Schubert polynomials \cite{BB}, and is also proven in \cite{BH}.

\begin{prop} \label{P:Schub triangular}
The transition matrix between Schubert polynomials and monomials is unitriangular:
\begin{align}\label{E:Schub monomial triangularity}
	\S_w (\xp)= x^{c(w)} + \text{reverse-lex lower terms.}
\end{align}
\end{prop}

\begin{thm}\label{T:single Schub} The Schubert polynomials are the unique family of polynomials $\{\S_w (\xp)\in \Q[\xp] \mid w \in S_+\}$ satisfying the following conditions:
\begin{align}
\label{E:Schub identity}
\S_\id (\xp)&= 1 \\
\label{E:Schub degree}
\S_w(\xp)\, &\text{is homogeneous of degree $\ell(w)$} \\
\label{E:S recurrence} A_i \S_w(\xp) &= 
\begin{cases}
\S_{w s_i}(\xp) & \text{if $ws_i < w$} \\
0& \text{otherwise.}
\end{cases}
\end{align}
The elements $\{\S_w (\xp)\mid w \in S_+\}$ form a basis of $\Q[\xp]$ over $\Q$.
\end{thm}
\begin{proof} 
For uniqueness, by induction we may assume that $\S_{ws_i}(\xp)$ is uniquely determined for all $i$ such that $ws_i<w$. 
Since the applications of all the $A_i$ are specified on $\S_w$,
the difference of any two solutions of \eqref{E:S recurrence}, being in the kernel of all $A_i$, is $S_+$-invariant by Lemma \ref{L:ker partial}. But $\Q[\xp]^{S_+} = \Q$ so the homogeneity assumption implies that the two solutions must be equal.

For existence, we note that the Schubert polynomials satisfy \eqref{E:Schub identity}, \eqref{E:Schub degree}, and \eqref{E:S recurrence} when $ws_i<w$. When $ws_i>w$, we have $\S_w=A_i \S_{ws_i}$ by \eqref{E:S recurrence} applied for $ws_i$. The element $\S_w$, being in the image of $A_i$, is $s_i$-invariant and therefore is in $\ker A_i$ by Lemma \ref{L:ker partial}. That is, $A_i \S_w=0$, establishing \eqref{E:S recurrence}.

The basis property holds by Proposition \ref{P:Schub triangular}. 
\end{proof}

\begin{rem} All the basis theorems for Schubert polynomials and their relatives, such as Theorem \ref{T:single Schub}, hold over $\Z$.
\end{rem}

\subsection{Double Schubert polynomials}
Let $\Q[\xp,\ap] := \Q[x_1,x_2,\ldots,a_1,a_2,\ldots]$.  The divided difference operators $A_i$, $i>0$ act on $\Q[\xp,\ap]$ by acting on the $x$-variables only.  Double Schubert polynomials \cite{LS} are defined by the action of divided difference operators on the expression in \eqref{E:doubleSchublong}.  We summarize the fundamental statements concerning double Schubert polynomials in the following theorem.

\begin{thm}\label{thm:double}
There exists a unique family $\{\S_w(\xp;\ap) \in \Q[\xp,\ap] \mid w \in S_+\}$ of polynomials satisfying the following conditions:
\begin{align}
\label{E:dSchub identity}
\S_\id(\xp;\ap) &= 1 \\
\label{E:dSchub loc at id}
\S_w(\ap;\ap) &= 0\qquad\text{if $w \neq \id$} \\
\label{E:dSchub ddiff}
A_i \S_w(\xp;\ap) &= 
\begin{cases}\S_{w s_i}(\xp;\ap) & \text{if $ws_i < w$} \\
0 & \text{otherwise.}
\end{cases}
\end{align}
The elements $\{\S_w(\xp;\ap) \mid w \in S_+\}$ form a basis of $\Q[\xp,\ap]$ over $\Q[\ap]$.
\end{thm}
\begin{proof} Uniqueness is proved as in Theorem \ref{T:single Schub}.
For existence, let
\begin{align}\label{E:doubleSchublong}
	\S_{w_0^{(n)}}(\xp;\ap) = \prod_{\substack{1\le i,j\le n \\ i+j\le n}}
	(x_i-a_j).
\end{align}
This agrees with \eqref{E:dSchub identity}. It is straightforward to verify the double analogue of Lemma \ref{L:S_+ well-defined}. 

For \eqref{E:dSchub loc at id} it suffices to prove the stronger vanishing property
\begin{align}\label{E:Schub vanishing}
	\S_v(w \ap;\ap) = 0 \qquad\text{unless $v\le w$.}
\end{align}
Here $\S_v(w\ap;\ap):=\S_v(a_{w(1)},a_{w(2)},\dotsc;\ap)$.
Let $v,w \in S_n$ with $v\not\le w$. If $v=w_0^{(n)}$ then by inspection $\S_v(w\ap;\ap)=0$. So suppose $v<w_0^{(n)}$. Let $1\le i\le n-1$ be such that $vs_i>v$. Then $vs_i \not\le w$ and also $vs_i \not\le ws_i$.
Substituting $x_k\mapsto a_{w(k)}$ into $A_i \S_{vs_i}(\xp;\ap) =\S_v(\xp;\ap)$ and using induction we have $\S_v(w\ap;\ap) = (a_i-a_{i+1})^{-1}(\S_{vs_i}(w\ap;\ap) - \S_{vs_i}(ws_i\ap;\ap)) = 0$, proving \eqref{E:Schub vanishing}.

The basis property follows from the fact that $\S_w(\xp;0)=\S_w(\xp)$ are a $\Q$-basis of $\Q[\xp]$.
\end{proof}

\subsection{Double Schubert polynomials into single}

The following identity is proved in Appendix \ref{A:Schub inversion}.

\begin{lem}\label{L:Schub cancellation} For $w \in S_+$, we have
\begin{align}\label{E:Schub cancellation}
	\sum_{w \doteq uv} (-1)^{\ell(u)} \S_{u^{-1}} (\ap)\S_v(\ap)= \delta_{w,\id}.
\end{align}
\end{lem}

\begin{proposition}[{\cite[(6.1)]{Mac} \cite[Lemma 4.5]{FS}}] \label{prop:dSchub} Let $w \in S_+$.  Then 
\begin{align}\label{E:double to single}
\S_w(\xp;\ap) = \sum_{w \doteq uv} (-1)^{\ell(u)} \S_{u^{-1}}(\ap)\S_v(\xp).
\end{align}
\end{proposition}

\begin{proof} It suffices to verify the conditions of 
Theorem \ref{thm:double}.  \eqref{E:dSchub identity} is clear. \eqref{E:dSchub loc at id} holds by Lemma \ref{L:Schub cancellation}. We prove \eqref{E:dSchub ddiff} by induction on $\ell(w)$.  The case $\ell(w) = 0$ is trivial.  We have
\begin{align*}
A_i \sum_{w \doteq uv} (-1)^{\ell(u)} \S_{u^{-1}}(\ap)\S_v(\xp)  & = \sum_{\substack{w \doteq uv \\ vs_i < v}} (-1)^{\ell(u)} \S_{u^{-1}}(\ap)\S_{vs_i}(\xp) \\
&= \begin{cases} \sum_{\substack{ws_i \doteq uv' }} (-1)^{\ell(u)} \S_{u^{-1}}(\ap)\S_{v'}(\xp) & \text{if $ws_i < w$,} \\
0 & \text{otherwise.}
\end{cases}
\end{align*}
This establishes \eqref{E:dSchub ddiff} by induction.
\end{proof}

\subsection{Left divided differences}
Let $A_i^a$ be the divided difference operator acting on the $a$-variables.

\begin{lemma}\label{L:left ddiff} For $i>0$ and $w\in S_+$,
\begin{align}
	A_i^a \S_w(\xp;\ap) &= \begin{cases}
		- \S_{s_iw}(\xp;\ap) & \text{if $s_iw<w$} \\
		0 & \text{otherwise.}
		\end{cases}
\end{align}
\end{lemma}
\begin{proof} This is easily verified using Proposition \ref{prop:dSchub}.
\end{proof}

\section{Back stable Schubert polynomials}
We define the ring of back symmetric formal power series, and study the basis of back stable Schubert polynomials.
\subsection{Symmetric functions in nonpositive variables}
For $b\in\Z$, let $\Lambda(x_{\le b})$ be the $\Q$-algebra of symmetric functions in the variables $x_i$ for $i\in \Z$ with $i\le b$.
We write $\Lambda = \Lambda(x_{\le0})= \Lambda(\xm)$, emphasizing that our symmetric functions are in variables with \textit{nonpositive} indices. See Appendix \ref{S:pos to non} for the comparison with symmetric functions in variables with positive indices.

The tensor product $\La\otimes\La$ is isomorphic to the $\Q$-algebra of formal power series of bounded total degree in $\xm$ and $\am$ which are separately symmetric in $\xm$ and $\am$. Under this isomorphism, we have $g\otimes h\mapsto  g(\xm)h(\am)$. We use this alternate notation without further mention.

The $\Q$-algebra $\Lambda$ is a Hopf algebra over $\Q$, generated as a polynomial $\Q$-algebra by primitive elements 
$$
p_k = \sum_{i\le0} x_i^k.
$$
That is, $\Delta(p_k) = 1 \otimes p_k + p_k \otimes 1$ (or $\Delta(p_k)=p_k(\xm)+p_k(\am)$). Equivalently, for $f\in \La$, $\Delta(f)$ is given by plugging both $\xm$ and $\am$ variable sets into $f$.
The counit takes the coefficient of the constant term, or equivalently, is the $\Q$-algebra map sending $p_k\mapsto0$ for all $k\ge1$. The antipode is the $\Q$-algebra automorphism sending $p_k\mapsto -p_k$ for all $k\ge1$. For a symmetric function $f(x)$ we write $f(/x)$ for its image under the antipode. 

The \defn{superization} map 
\begin{align}
	\label{E:super map}
	\Lambda&\to \Lambda\otimes\Lambda, \qquad
	f \mapsto f(x/a)
\end{align}
is the $\Q$-algebra homomorphism defined by applying the coproduct $\Delta$ 
followed by applying the antipode in the second factor. Equivalently, it is the $\Q$-algebra homomorphism sending $p_k\mapsto p_k(\xm)-p_k(\am)$.
In particular, $f(x/a)$ is symmetric in $\xm$ and symmetric in $\am$.  We use the notation $f(x/a)$ instead of $f(\xm/\am)$ for the sake of simplicity.

\subsection{Back symmetric formal power series}
Let $R$ be the $\Q$-algebra of formal power series $f$ in the variables
$x_i$ for $i\in\Z$ such that $f$ has bounded total degree (there is an $M$ such that all monomials in $f$ have total degree at most $M$) and the support of $f$ is bounded above (there is an $N$ such that the variables $x_i$ do not appear in $f$ for $i>N$). The group $S_\Z$ acts on $R$ by permuting variables. Say that $f\in R$ is \defn{back symmetric} if there is a $b\in \Z$ such that $s_i(f)=f$ for all $i< b$.
Let $\bR$ be the subset of  back symmetric elements of $R$.

\begin{prop}\label{P:backsymmetric} We have the equality
\begin{align}\label{E:backsymmetric}
  \bR = \Lambda \otimes \Q[x].
\end{align}
\end{prop}
\begin{proof} It is straightforward to verify that $\bR$ is a $\Q$-subalgebra of $R$ containing $\Lambda$ and $\Q[x]$.
Suppose $f\in R$ is back symmetric. Let $b\in\Z$ be such that $s_i(f)=f$ for all $i< b$.
Then $f \in\Lambda(x_{\le b}) \otimes \Q[x_{b+1},x_{b+2},\dotsc]$ is a polynomial in the power sums $p_k(x_{\le b})$ and the variables $x_{b+1},x_{b+2},\dotsc$. But $p_k(x_{\le b}) - p_k(x_{\le0}) \in \Q[x]$. It follows that $f\in \Lambda \otimes \Q[x]$.
\end{proof}

We emphasize that $\bR$ is a polynomial $\Q$-algebra with algebraically independent generators $p_k$ for $k\ge 1$ and $x_i$ for $i\in\Z$.
The restriction of the action of $S_\Z$ from $R$ to $\bR$ is given
on algebra generators by
\begin{align*}
	w(x_i) &= x_{w(i)} \\
	s_i(p_k) &= 
	\begin{cases}
		p_k  &\text {if $i\ne0$} \\
		p_k - x_0^k + x_1^k & \text {if $i=0$.}
	\end{cases}
\end{align*}
For $s_0(p_k)$ we use the computation
\begin{align*}
	s_0 \sum_{i\le 0} (x_i^k-a_i^k) 
	= \sum_{i\le -1} (x_i^k-a_i^k) + s_0(x_0^k-a_0^k) 
	= \sum_{i\le -1} (x_i^k-a_i^k) + x_1^k-a_0^k 
	= p_k -x_0^k + x_1^k.
\end{align*}
The divided difference operators $A_i$ for $i \in \Z$ act on $\bR$ using the same formula as \eqref{eq:dd}.

\subsection{Back stable limit}
Let $\shift: \bR \to \bR$ be the $\Q$-algebra automorphism shifting all $x$ variables, that is,
\begin{align}\label{E:shift}
	\shift(x_i) &= x_{i+1} & \shift^{-1}(x_i) &= x_{i-1} \\
	\shift(p_k) &= p_k + x_1^k & \shift^{-1}(p_k) &= p_k - x_0^k.
\end{align}

Given $w\in S_{\Z}$, let $[p,q]\subset \Z$ be an interval that
contains all non-fixed points of $w$. Let $\S_{w}^{[p,q]}$ be the usual Schubert polynomial but computed using the variables $x_p,x_{p+1},\dotsc,x_q$ instead of starting with $x_1$. This is the same as shifting $w$ to start at $1$ instead of $p$, constructing the Schubert polynomial, and then shifting variables to start at $x_p$ instead of $x_1$. That is,
\begin{align*}
	\S_w^{[p,q]}(x_p,\dotsc,x_q) = \shift^{p-1}(\S_{\shift^{1-p}(w)}(\xp)).
\end{align*}

We say that the limit of a sequence $f_1,f_2,\ldots$ of formal power series is equal to a formal power series $f$ if, for each monomial $M$, the coefficient of $M$ in $f_1,f_2,\ldots$ eventually stabilizes and equals the coefficient in $f$.

\begin{thm}\label{T:backstable} For $w\in S_\Z$,
there is a well-defined power formal series $\bS_w\in \bR$ given by $$\bS_w := \lim_{\substack{p \to -\infty \\ q\to \infty}} \S_{w}^{[p,q]}$$ called the \defn{back stable Schubert polynomial}. It has the monomial expansion
\begin{equation}\label{eq:backstable}
	\bS_w = \sum_{a_1a_2 \cdots a_\ell \in \Red(w)} \sum_{\substack{b_1 \leq b_2 \leq \cdots \leq b_\ell\\ a_i<a_{i+1} \implies b_i < b_{i+1} \\ b_i \leq a_i}} x_{b_1} x_{b_2} \cdots x_{b_\ell}
\end{equation}
in which $b_i\in\Z$. Moreover, the back stable Schubert polynomials are the unique family
$\{\bS_w \in \bR \mid w \in S_\Z\}$ of elements satisfying the following conditions:
\begin{align}
	\label{E:backstable identity}
	\bS_\id &= 1 \\
	\label{E:backstable degree}
	\bS_w &\text{ is homogeneous of degree $\ell(w)$} \\
	\label{E:partial backstable}
	A_i \bS_w &= \begin{cases} \bS_{w s_i} & \text{if $ws_i < w$,} \\
		0 & \text{otherwise.}
	\end{cases}
\end{align}

\end{thm}
\begin{proof} The well-definedness of the series and its monomial expansion follows by taking the limit of \eqref{E:BJS}. Let $w\in S_\Z$. For $i\ll0$ we have $ws_i>w$. By \eqref{E:S recurrence} and Lemma \ref{L:ker partial} $\bS_w$ is $s_i$-symmetric. Thus $\bS_w$ is back symmetric.

Properties \eqref{E:backstable identity}, \eqref{E:backstable degree} and \eqref{E:partial backstable} hold for $\bS_w$ by the corresponding parts of Theorem \ref{T:single Schub} for usual Schubert polynomials.
\end{proof}

\begin{prop} \label{P:shift backstable}  For $w \in S_\Z$, we have $\shift(\bS_w) = \bS_{\shift(w)}$.
\end{prop}

\begin{prop} \label{P:backstable Grassmannian} For $\la\in\Par$, we have $\bS_{w_\la} = s_\la \in\La(\xm)$, the Schur function.
\end{prop}
\begin{proof} Let $0<k<n$ be large enough such that $\la$ is contained in the $k \times (n-k)$ rectangular partition. For such partitions the map $\la \mapsto \shift^k(w_\la)$ defines a bijection to the $k$-Grassmannian elements of $S_n$. It is well-known that $\Schub_{\shift^k(w_\la)} = s_\la(x_1,\dotsc,x_k)$ \cite[Chapter 10, Proposition 8]{Ful}. Applying $\shift^{-k}$ we have
$\Schub^{[1-k,n-k]}_{w_\la} = s_\la(x_{1-k}\dotsc,x_{-1},x_0)$.
The result follows by letting $k,n\to\infty$.
\end{proof}

By Propositions \ref{P:shift backstable} and 
\ref{P:Schub triangular} we have
\begin{align}\label{E:backstable monomial triangularity}
	\bS_w = x^{c(w)} + \text{reverse-lex lower terms}.
\end{align}

\begin{thm} \label{T:backstable basis} The back stable Schubert polynomials form a $\Q$-basis of $\bR$.
\end{thm}
\begin{proof} By \eqref{E:backstable monomial triangularity}
the back stable Schubert polynomials are linearly independent. 
For spanning, using Proposition \ref{P:shift backstable} and applying $\shift^n$ for $n$ sufficiently large, it suffices to show that any element of $\La(\xm) \otimes \Q[\xp]$ is a $\Q$-linear combination of finitely many back stable Schubert polynomials. This holds due to the unitriangularity \eqref{E:backstable monomial triangularity} of back stable Schubert polynomials with monomials and the following facts:
 (1) the reverse-lex leading monomial $x^\beta$ in any nonzero element of $\La(\xm)\otimes \Q[\xp]$ satisfies 
$\dotsm \le \beta_{-2} \le \beta_{-1} \le \beta_0$. (2) If $w\in S_\Z$ is such that $c(w)=\beta$ for such a $\beta$ then $\dotsm <w(-2)<w(-1)<w(0)$. (3) For such $w$, $\bS_w$ is symmetric in $\xm$ so that $\bS_w\in \La(\xm)\otimes \Q[\xp]$. (4) There are finitely many $\gamma$ below $\beta$ in reverse-lex order such that $x^\gamma$ and $x^\beta$ have the same degree, and satisfying $\dotsm\le \gamma_{-2}\le \gamma_{-1}\le \gamma_0$.
\end{proof}

\subsection{Stanley symmetric functions}\label{SS:Stanley}
Stanley \cite{Sta} defined Stanley symmetric functions $F_w(\xp)$ to enumerate reduced decompositions of permutations.  These symmetric functions are also called stable Schubert polynomials, and are usually defined by $F_w(\xp):= \lim_{n\to\infty} \S_{\shift^n(w)}(\xp)$.  Our definition $F_w$ of Stanley symmetric function agrees (by Theorem \ref{T:Stanley monomial expansion}) with the standard definition up to using $\xm$ instead of $\xp$.

There is a $\Q$-algebra map $\eta_0: \Q[x] \to \Q$ given by evaluation at zero: $x_i \mapsto 0$ for all $i\in\Z$.  This induces a $\Q$-algebra map $1 \otimes \eta_0: \bR \to \Lambda \otimes_\Q \Q \cong \Lambda$, which we simply denote by $\eta_0$ as well.

 \begin{remark} \label{R:eta knows} The map $\eta_0$ ``knows" the difference between $x_i\in \Q[x]$ and the $x_i$ that appear in $\Lambda=\Lambda(\xm)$.
 \end{remark}

For $w \in S_\Z$, we define the \defn{Stanley symmetric function} by
\begin{align}\label{E:Stanley def}
F_w := \eta_0(\bS_w)\in\Lambda.
\end{align}

Recall the shifting automorphism $\shift:S_\Z\to S_\Z$ from \textsection \ref{SS:notation}.

\begin{lem}\label{L:eta shift}   For $f \in \bR$, we have $\eta_0(\shift(f))=\eta_0(f)$.
\end{lem}
\begin{proof} This holds since $\eta_0$ is a $\bQ$-algebra homomorphism and the claim is easily verified for the algebra generators of $\bR$.
\end{proof}

\begin{cor}\label{C:shift stanley} For $w\in S_\Z$, we have $F_{\shift(w)} = F_w$.
\end{cor}
\begin{proof}
Using Lemma \ref{L:eta shift} and Proposition \ref{P:shift backstable}, we have
$F_{\shift(w)} = \eta_0(\bS_{\shift(w)}) = \eta_0(\shift(\bS_w)) = \eta_0(\bS_w) = F_w$.
\end{proof}

\begin{thm}[cf. \cite{Sta}]\label{T:Stanley monomial expansion} For $w\in S_\Z$, we have
\begin{align}\label{E:Stanley monomial expansion}
F_w = \sum_{a_1a_2 \cdots a_\ell \in \Red(w)} \sum_{\substack{b_1 \leq b_2 \leq \cdots \leq b_\ell \leq 0 \\ a_i<a_{i+1} \implies b_i < b_{i+1} }} x_{b_1} x_{b_2} \cdots x_{b_\ell} .
\end{align}
\end{thm}
\begin{proof} By Corollary \ref{C:shift stanley} we may assume that $w\in S_+$. Since $ws_i>w$ for $i<0$, $\bS_w$ is $s_i$-symmetric for $i<0$, that is, $\bS_w \in \Lambda \otimes \Q[\xp]$.
Therefore $F_w$ is obtained from $\bS_w$ by setting $x_i=0$ for $i\ge 1$. Making this substitution in \eqref{eq:backstable} yields \eqref{E:Stanley monomial expansion}.
\end{proof}

The \defn{Edelman-Greene coefficients} $j_\lambda^w \in \Z$ are defined by 
\begin{align}\label{E:stanley coefficients} 
	F_w = \sum_\lambda j_\lambda^w s_\lambda.
\end{align}

These coefficients are known to be nonnegative and have a number of combinatorial interpretations: leaves of the transition tree \cite{LS2}, promotion tableaux \cite{H}, and peelable tableaux \cite{RS}.
In particular, by \cite{EG} $j^w_\la$ is equal to the number of {\it reduced word tableaux} for $w$: that is, row strict and column strict tableaux of shape $\la$ whose row-reading words are reduced words for $w$.

Let $\omega$ be the involutive $\Q$-algebra automorphism of $\La$
defined by $\omega(p_r) = (-1)^{r-1} p_r$ for $r\ge1$. We have $\omega(s_\la)=s_{\la'}$ for $\la\in\Par$. The action of $\omega$ on a homogeneous element of degree $d$ is equal to that of the antipode times $(-1)^d$.
Let $\omega$ also denote the automorphism of $S_{\Z}$ given by
$s_i\mapsto s_{-i}$ for all $i\in \Z$.

\begin{prop} \label{P:Stanley symmetries} For $w\in S_\Z$, we have $F_{w^{-1}} = \omega(F_w) = F_{\omega(w)}$.
\end{prop}
\begin{proof} Reversal of a reduced word gives a bijection $\Red(w) \to \Red(w^{-1})$ that sends a Coxeter-Knuth class of shape $\la$ (see \textsection \ref{SS:insertion}) to a Coxeter-Knuth class of shape $\la'$. The first equality follows.  
	
Negating each entry of a reduced word gives a bijection
$\Red(w) \to \Red(\omega(w))$ which sends a Coxeter-Knuth class of shape $\la$ to a Coxeter-Knuth class of shape $\la'$. The second equality follows.  
\end{proof}

\begin{prop} \label{P:Hopf Stanley} For $w\in S_\Z$, we have
\begin{align}
\label{E:coprod stanley}
	\Delta(F_w) &= \sum_{w\doteq uv} F_u \otimes F_v \\
\label{E:antipode stanley}
	F_w(/x) &= (-1)^{\ell(w)} F_{w^{-1}}(x) \\
\label{E:plethystic difference stanley}
	F_w(x/a) &= \sum_{w\doteq uv} (-1)^{\ell(u)} F_{u^{-1}}(a) F_v(x) \\
\label{E:counit}
	F_w(a/a) &= \delta_{w,\id} \\
\label{E:add subtract stanley}
    F_w(x) &= \sum_{w\doteq uvz} (-1)^{\ell(u)} F_{u^{-1}}(a) F_v(x) F_z(a).
\end{align}
\end{prop}
\begin{proof} Equation \eqref{E:coprod stanley} follows by plugging in two set of variables into \eqref{E:Stanley monomial expansion}.
Equation \eqref{E:antipode stanley} follows from Proposition \ref{P:Stanley symmetries}. Equation \eqref{E:plethystic difference stanley} is obtained by combining \eqref{E:coprod stanley} and \eqref{E:antipode stanley}. 
Equation \eqref{E:counit} follows from the Hopf algebra axiom which asserts that superization followed by multiplication is the counit.
For \eqref{E:add subtract stanley} we have
\begin{align*}
	\sum_{w \doteq uvz} (-1)^{\ell(u)} F_{u^{-1}}(a) F_v(x) F_z(a) 
	&= \sum_{w \doteq uvz} (-1)^{\ell(u)} F_{u^{-1}}(a) F_v(a) F_z(x) = \sum_{w \doteq yz} F_y(a/a) F_z(x) = F_w(x)
\end{align*}
using cocommutativity, \eqref{E:plethystic difference stanley}, and \eqref{E:counit}.
\end{proof}

\subsection{Negative Schubert polynomials}\label{SS:negative Schubert} The following Schubert polynomials are indexed by permutations in $S_-$, contain variables indexed by nonpositive integers, and may contain signs. Recall $S_-$ and $S_{\ne0}$ from \textsection \ref{SS:notation} and the automorphism $\omega$ of $S_{\Z}$. It restricts to an isomorphism $S_-\to S_+$. 
Let $\omega:\Q[x]\to\Q[x]$ be the $\Q$-algebra automorphism defined by $\omega(x_i)=-x_{1-i}$ for $i\in\Z$. For $u\in S_-$, define $\S_u(\xm)\in \Q[\xm]$ by
\begin{align}
	\S_u(\xm)  &:= \omega(\S_{\omega(u)}(\xp)).
\end{align}
That is, in $w$ replace the negatively-indexed reflections with positively-indexed ones, take the usual Schubert polynomial in positively-indexed variables, and then use $\omega$ to substitute nonpositively-indexed $x$ variables for the positively-indexed ones (with signs). 
 
\begin{example} For $u=s_{-3}s_{-2}s_{-1}$, we have $\omega(u)=s_3s_2s_1$, $\Schub_{s_3s_2s_1}(\xp) = x_1^3$, and $\Schub_u(\xm)=-x_0^3$. For $i > 0$ we have
	$\Schub_{s_{-i}} = \omega \Schub_{s_i} =\omega (x_1+\dotsm+x_i) = 
	-(x_0+x_{-1}+\dotsm+x_{1-i})$.
\end{example}
 
By Theorem \ref{thm:BJS}, we have
$$
\S_u(\xm) =  \sum_{a_1a_2 \cdots a_\ell \in R(u)} \sum_{\substack{   0 \geq b_1 \geq b_2 \geq \cdots \geq b_\ell  \\ a_i > a_{i+1} \implies b_i > b_{i+1} \\ b_i \geq a_i+1}} x_{b_1} x_{b_2} \cdots x_{b_\ell}\qquad\text{for $u\in S_-$.}
$$
Note the $+1$ in $b_i \geq a_i+1$.

For $w\in S_{\ne0}$, define $\S_w(x)\in \Q[x]$ by
\begin{align}
\S_w(x) &:= \S_u(\xm)\S_v(\xp)\qquad\text{where $w=uv$ with $u\in S_-$ and $v\in S_+$.}
\end{align}

\begin{prop}  \label{P:negate Schubert} For $w \in S_{\ne 0}$, we have $\omega(\Schub_w(x)) = \Schub_{\omega(w)}(x) $.
\end{prop}

\subsection{Coproduct formula} 
\label{SS:coproduct backstable}
 There is a coaction $\Delta: \bR \to \Lambda \otimes \bR$ of $\Lambda$ on $\bR$, defined by the comultiplication on the first factor of the tensor product $\Lambda \otimes_\Q \Q[x]$.

\begin{thm}\label{thm:coprod}
Let $w \in S_\Z$.  We have the coproduct formulae
\begin{align}\label{E:big coprod}
\Delta(\bS_w) &= \sum_{w \doteq xy} F_x \otimes \bS_y\\
\label{E:little coprod}
\bS_w &= \sum_{\substack{w \doteq xy  \\ y \in S_{\neq 0}}} F_x \; \S_y
\end{align}
\end{thm}
\begin{proof} Equation \eqref{E:big coprod} can be deduced from \eqref{E:little coprod} and Proposition  \ref{P:Hopf Stanley}:
\begin{align*}
	\Delta(\bS_w) 
	= \sum_{\substack{w\doteq xy \\ y\in S_{\ne0}}} \Delta(F_x) \S_y 
	= \sum_{\substack{w\doteq uvy \\ y\in S_{\ne0}}} F_u \otimes F_v \S_y 
	= \sum_{w\doteq uz} F_u \otimes \bS_z.
\end{align*} 


We prove \eqref{E:little coprod} by a cancellation argument.
We say that a pair of integer sequences $(\a,\b)$ of the same length is a \defn{compatible pair}, if $\b$ is weakly increasing and $a_i < a_{i+1} \implies b_i < b_{i+1}$.

Let $(x,y,\a,\b)$ index a monomial $x_\b = x_{b_1} \cdots x_{b_\ell}$ on the right hand side, corresponding to the term $F_x \; \S_y$ and reduced word $\a = a_1 a_2 \cdots a_\ell$.  By convention, to obtain $\a$, we always factorize $y \in S_{\ne0}$ as $y = y'y''$ with $y' \in S_-$ and $y'' \in S_+$.  We will provide a partial sign-reversing involution $\iota$ on the quadruples $(x,y,\a,\b)$; the left-over monomials will give the left hand side.

Suppose $\ell(x) = r$, $\ell(y') = s$, $\ell(y'') = t$, and set $\ell = r+s+t$.
Call an index $i \in [1,\ell]$ {\it bad} if $b_i > a_i$, and good if $b_i \leq a_i$.  It follows from the definitions that all indices $i \in [r+s+1,\ell]$ are good, while all indices $i \in [r+1,r+s]$ are bad.  Furthermore, if $i \in [1,r]$ is bad, then $a_i < 0$. 

Let $k$ be the largest bad index in $[1,r]$, which we assume exists.   We claim that $s_{a_k}$ commutes with $s_{a_{k+1}} \cdots s_{a_r}$.  To see this, observe that if $a_{k'} \in \{a_k-1,a_k,a_k+1\}$ where $k < k' \leq r$ then we must have $b_{k'} > a_{k'}$, contradicting our choice of $k$.  If $s = 0$, we set
\begin{equation}\label{eq:case1}
\iota(x,y,\a,\b) = (a_1 \cdots \hat a_k \cdots a_{r} |a_k a_{r+1} \cdots a_t, b_1 \cdots \hat b_k \cdots b_r| b_k b_{r+1} \cdots b_t) = (\tilde x,\tilde y,\tilde \a, \tilde \b)
\end{equation}
where the vertical bar separates $x$ from $y$.  Thus $\tilde y' = s_{a_k}$.  If $s > 0$, we compare $b_k$ with $b_{r+1}$.  If $b_k > b_{r+1}$ or ($b_k = b_{r+1}$ and $a_k < a_{r+1}$) then we again make the definition \eqref{eq:case1} where now $\tilde y' = s_{a_k} y'$.  We call this CASE A.

Suppose still that $s > 0$.  If ($b_k < b_{r+1}$) or ($b_k = b_{r+1}$ and $a_k \geq a_{r+1}$) or ($k$ does not exist) then there is a unique index $j \in [k,r]$ so that
\begin{align*}
\iota(x,y,\a,\b) &= (a_1 \cdots a_j a_{r+1} a_{j+1} \cdots a_{r} |a_{r+2} \cdots a_t, b_1 \cdots b_j b_{r+1} b_{j+1} \cdots b_r | b_{r+2} \cdots b_t) = (\tilde x,\tilde y,\tilde \a, \tilde \b)
\end{align*}
has the property that $(a_1 \cdots a_j a_{r+1} a_{j+1} \cdots a_{r} , b_1 \cdots b_j b_{r+1} b_{j+1} \cdots b_r)$ is a compatible sequence.  In this case, $s_{a_{r+1}}$ commutes with $s_{a_{j+1}} \cdots s_{a_{r}}$.  We call this CASE B.

Finally, if $s = 0$ and $k$ does not exist, then $\iota$ is not defined. 

It remains to observe that CASE A and CASE B are sent to each other via $\iota$, which keeps $x_\b$ constant and changes $\ell(y')$ by 1.
\end{proof}

Let $\omega$ be the involutive $\Q$-algebra automorphism of $\bR$
given by combining the maps $\omega$ on $\La$ from \textsection \ref{SS:Stanley} and on $\Q[x]$ from \textsection \ref{SS:negative Schubert}. 

\begin{prop}\label{P:negate back stable} For all $w\in S_\Z$, we have $\omega(\bS_w) = \bS_{\omega(w)}$.
\end{prop}
\begin{proof} This follows immediately from Theorem \ref{thm:coprod} and Propositions \ref{P:Stanley symmetries} and \ref{P:negate Schubert}.
\end{proof}

\begin{remark}
The elements $\{s_\lambda \otimes \S_v \mid \lambda \in \Par \text{ and } v \in S_{\neq 0}\}$ form a $\Q$-basis of $\bR$.  It follows from Theorem \ref{thm:coprod} that the coefficient of $s_\lambda \otimes \S_v$ in $\bS_w$ is equal to $j_\lambda^{wv^{-1}}$ if $\ell(wv^{-1}) = \ell(w) - \ell(v)$, and 0 otherwise.
\end{remark}

\begin{remark}
Let $\nu_\lambda: \bR \to \Q[x]$ denote the linear map given by ``taking the coefficient of $s_\lambda$".  Then 
\begin{align}
\nu_\lambda(\bS_w) = \sum_{\substack{v \in S_{\neq 0} \\ \ell(wv^{-1}) = \ell(w) - \ell(v)}} j_\lambda^{wv^{-1}} \S_v.
\end{align}
We will give an explicit description of the polynomial $\nu_\lambda(\bS_w)$ in Theorem \ref{thm:decomp}.
\end{remark}
\subsection{Back stable Schubert structure constants}
For $u,v,w\in S_+$, define the usual Schubert structure constants $c^w_{uv}$ by 
\begin{align}\label{E:structure constants}
	\Schub_u \Schub_v = \sum_{w\in S_+} c^w_{uv} \Schub_w.
\end{align}
For $u,v,w\in S_\Z$, define the back stable Schubert structure constants $\bc^w_{uv}\in\Q$ by
\begin{align}\label{E:backstable structure constants}
	\bS_u \bS_v &= \sum_{w\in S_\Z} \bc^w_{uv} \bS_w.
\end{align}
By Proposition \ref{P:shift backstable}, we have
\begin{align}\label{E:Schubert constant shift}
	\bc^{\shift^n(w)}_{\shift^n(u),\shift^n(v)} = \bc^w_{uv}\qquad\text{for all $u,v,w\in S_\Z$ and $n\in\Z$.}
\end{align}

\begin{prop} \label{P:backstable structure constants}  \
	\begin{enumerate}
		\item For $u,v,w\in S_+$, we have $c^w_{uv} = \bc^w_{uv}$.
		\item Every back stable Schubert structure constant is a usual Schubert structure constant.
	\end{enumerate}
\end{prop}
\begin{proof} Consider the $\Q$-algebra homomorphism $\pi_+:\bR\to \Q[\xp]$ sending $p_r\mapsto0$ for $r\ge 1$, $x_i\mapsto0$ for $i\le 0$ and $x_i\mapsto x_i$ for $i>0$. Applying $\pi_+$ to Theorem \ref{thm:coprod}  for $y\in S_\Z$ we have
\begin{align}\label{E:pi+}
	\pi_+(\bS_y) &= \begin{cases}
		\S_y & \text{if $y\in S_+$} \\
		0 & \text{otherwise.}
	\end{cases}
\end{align}
because $\pi_+$  kills all symmetric functions with no constant term and all negative Schubert polynomials of positive degree.
Now let $u,v\in S_+$. Applying $\pi_+$ to \eqref{E:backstable structure constants} and using \eqref{E:pi+}, (1) follows.
	
For (2), let $u,v\in S_\Z$. By \eqref{E:Schubert constant shift}, we may assume that $u,v\in S_+$ and that the finitely many $w$ appearing in \eqref{E:backstable structure constants} are also in $S_+$. The proof is completed by applying part (1).
\end{proof}

\begin{example} We have $\bS_{s_1}^2 = \bS_{s_2s_1} + \bS_{s_0s_1}$
	and $\S_{s_1}^2 = \S_{s_2s_1}$. Shifting forward by one we obtain
	$\bS_{s_2}^2 = \bS_{s_3s_2}+\bS_{s_1s_2}$ and $\S_{s_2}^2 = \S_{s_3s_2}+\S_{s_1s_2}$.
\end{example}

We derive a relation involving back stable Schubert structure constants and Edelman-Greene coefficients.

\begin{prop} \label{P:tease}
Let $u\in S_m$, $v\in S_n$ and $\la\in \Par$. Let $u\times v:=u \shift^m(v)\in   S_{m+n}\subset S_+$. Then
\begin{align}
\label{E:constants}
	j^{u \times v}_\la &= \sum_{w\in S_\Z} \bc^w_{uv} j^w_\la.
\end{align}
\end{prop}
\begin{proof} Since $u\times v\in S_m \times S_n \subset S_{m+n}$ it follows that $\S_{u\times v} = \S_u \S_{\shift^m(v)}$.
We deduce that $\bS_{u \times v} = \bS_u \shift^m(\bS_v)$.
Using the algebra map $\eta_0$ several times we obtain
\begin{align*}
	F_{u\times v} = F_u F_v 
	= \eta_0(\bS_u \bS_v) 
	= \eta_0(\sum_{w\in S_\Z} \bc^w_{uv} \bS_w) 
	= \sum_{w\in S_\Z} \bc^w_{uv} F_w.
\end{align*}
Taking the coefficient of $s_\la$ we obtain \eqref{E:constants}.
\end{proof}

%
%

\section{Back stable double Schubert polynomials}
%

We define the back symmetric double power series ring, and study the basis of double back stable Schubert polynomials.

\subsection{Double symmetric functions}\label{SS:double symmetric}
Let $p_k(x||a) := p_k(x/a) =\sum_{i \le 0} x_i^k - a_i^k$, a formal power series in variables $x_i$ and $a_i$; it is the image of $p_k$ under superization. Let $\Lambda(x||a)$ be the $\Q[a]$-algebra generated by the elements $p_1(x||a), p_2(x||a),\ldots$, which are algebraically independent over $\Q[a]$.  We call $\Lambda(x||a)$ the \defn{ring of double symmetric functions} (see \cite{M} for more details). For $\la = (\la_1,\la_2,\ldots,\la_{\ell}) \in \Par$, we denote $p_\la(x||a):= p_{\la_1}(x||a) \cdots p_{\la_\ell}(x||a)$.

The algebra $\La(x||a)$ is a Hopf algebra over $\Q[a]$ with primitive generators
$p_k(x||a)$ for $k\ge1$. The counit is the $\Q[a]$-algebra homomorphism $\epsilon:\La(x||a)\to \Q[a]$ given by $p_k(x||a)\mapsto0$ for $k\ge1$.
The antipode is the $\Q[a]$-algebra homomorphism defined by $p_k(x||a)\mapsto -p_k(x||a)$ for $k\ge1$.

\subsection{Back symmetric double power series} 
Define the \defn{back symmetric double power series ring} $\bR(x;a):= \Lambda(x||a) \otimes_{\Q[a]} \Q[x,a]$, where $\Q[x,a]:=\Q[x_i,a_i\mid i\in \Z]$.  The ring $\bR(x;a)$ has two actions of $S_\Z$: one that acts on all the $x$ variables and one that acts on all the $a$ variables, including those in $\La(x||a)$. More precisely 
for $i\in\Z$ let $s_i^x$ (resp. $s_i^a)$ act on $\bR(x;a)$ by exchanging $x_i$ and $x_{i+1}$ (resp. $a_i$ and $a_{i+1}$) while leaving the other polynomial generators of $\Q[x,a]$ alone and
\begin{align}
s_i^x(p_k(x||a)) &= 
\begin{cases}
p_k(x||a) & \text{if $i\ne0$} \\
p_k(x||a) - x_0^k + x_1^k & \text{if $i=0$} 
\end{cases} \\
s_i^a(p_k(x||a)) &= 
\begin{cases}
p_k(x||a) & \text{if $i\ne0$} \\
p_k(x||a) + a_0^k - a_1^k & \text{if $i=0$}.
\end{cases}
\end{align}
For $w\in S_\Z$, we write $w^x$ (resp. $w^a$) for this action of $w$ on the $x$-variables (resp. $a$-variables).


\subsection{Localization of back symmetric formal power series}\label{SS:loc back symmetric}
Let $\epsilon: \bR(x;a)\to\Q[a]$ be the $\Q[a]$-algebra homomorphism which extends the counit $\epsilon$ of $\La(x||a)$ via
\begin{align}
\epsilon(p_k(x||a)) &= 0\qquad\text{for all $k\ge1$} \\
\epsilon(x_i) &= a_i\qquad\text{for all $i\in\Z$.}
\end{align}
In other words $\epsilon$ ``sets all $x_i$ to $a_i$'' including those in $p_k(x||a)$. Define
\begin{align}\label{E:localize f}
f|_w = \epsilon(w^x(f)))=f(wa;a)\qquad\text{for $f(x,a)\in \bR(x;a)$ and $w\in S_\Z$.}
\end{align}


For any $w\in S_\Z$, let
\begin{align}
	\label{E:I+}
	I_{w,+} &:= \Z_{>0} \cap w(\Z_{\le0})\\
	\label{E:I-}
	I_{w,-} &:= \Z_{\le0} \cap w(\Z_{>0})
\end{align}

The map $w\mapsto (I_{w,+},I_{w,-})$ is a bijection
from $S_\Z^0$  to pairs of finite sets $(I_+,I_-)$ such that $I_+\subset \Z_{>0}$, $I_-\subset \Z_{\le0}$, and $|I_+|=|I_-|$. Then the following holds:

\begin{lemma} \label{lem:locp} We have $p_k(x||a)|_w = \sum_{i\in I_{w,+}} a_i^k - \sum_{i\in I_{w,-}} a_i^k$.
\end{lemma}

\begin{example} Using $w=w_\la$ of Example \ref{X:grass perm} we have $I_{w,+}=\{1,3\}$ and $I_{w,-}=\{-1,0\}$. Therefore $p_k(x||a)|_w = a_1^k+a_3^k-a_{-1}^k-a_0^k$.
\end{example}

\subsection{Back stable double Schubert polynomials}
Let $\shift$ be the $\Q$-algebra automorphism of $\bR(x;a)$ which shifts all variables forward by $1$ in $\bR(x;a)$. That is, $\shift(x_i)=x_{i+1}$, $\shift(a_i)=a_{i+1}$, and $\shift(p_k(x||a)) = p_k(x||a) + x_1^k - a_1^k$.
As before, let $[p,q]$ be an interval of integers containing all integers moved by $w\in S_\Z$. Define
\begin{align}
	\S_w^{[p,q]}(x;a) &:= \shift^{p-1}(\S_{\shift^{1-p}(w)}(\xp;\ap)).
\end{align}

For $w\in S_\Z$, define the \defn{back stable double Schubert polynomial} $\bS_w(x;a)$ by
\begin{align}
	\bS_w(x;a) := \lim_{\substack{p\to-\infty\\ q\to\infty}} \S_w^{[p,q]}(x;a).
\end{align}

There is a double version of the monomial expansion (Theorem \ref{thm:BJS}) of Schubert polynomials; see for example \cite{FK}.  However, the well-definedness of $\bS_w(x;a)$ is not apparent from that expansion.  In Theorem \ref{T:square bumpless} we give a new combinatorial formula for $\S_w(\xp;\ap)$ using bumpless pipedreams as a sum of products of binomials $x_i-a_j$.  Theorem \ref{T:square bumpless} is compatible with the back stable limit and yields a monomial formula (Theorem \ref{T:bumpless}) for the back stable double Schubert polynomials.

\begin{prop} \label{P:backstable double well-defined}
For $w \in S_\Z$, $\S_w(x;a)$ is a well-defined series such that
\begin{align}
\label{E:double double}
	\bS_w(x;a) &= \sum_{w\doteq uv} (-1)^{\ell(u)} \bS_{u^{-1}}(a) \bS_v(x)
\end{align}
\end{prop}
\begin{proof} Since length-additive factorizations are well-behaved under shifting it follows that
\begin{align*}
	\S_w^{[p,q]}(x;a) &= \shift^{p-1}(\S_{\shift^{1-p}(w)}(\xp;\ap)) \\
	&= \shift^{p-1}( \sum_{w \doteq uv} (-1)^{\ell(u)} \S_{\shift^{1-p}(u^{-1})}(\ap) \S_{\shift^{1-p}(v)}(\xp) )\\
	&= \sum_{w\doteq uv} (-1)^{\ell(u)} \S_{u^{-1}}^{[p,q]}(a) 
	\S_v^{[p,q]}(x)
\end{align*}
using Proposition \ref{prop:dSchub}.
Taking the limit as $p\to-\infty$ and $q\to\infty$ we obtain \eqref{E:double double}.
\end{proof}

\begin{cor}\label{C:backstable double shift} For $w \in S_\Z$, we have $\shift(\bS_w(x;a)) = \bS_{\shift(w)}(x;a)$.
\end{cor}

\begin{cor}\label{C:backstable triple sum} For $w \in S_\Z$, we have 
\begin{align}\label{E:backstable triple def}
	\bS_w(x;a) &= \sum_{\substack{w\doteq uvz \\ u,z\in S_{\ne0}}}
	(-1)^{\ell(u)} \S_{u^{-1}}(a) F_v(x/a) \S_z(x).
\end{align}	
In particular, $\bS_w(x;a)\in \bR(x;a)$.
\end{cor}
\begin{proof} Using \eqref{E:little coprod} and Propositions \ref{P:backstable double well-defined} and \ref{P:Hopf Stanley} we have
	\begin{align*}
	\bS_w(x;a) &= \sum_{w\doteq uv} (-1)^{\ell(u)} \bS_{u^{-1}}(a) \bS_v(x) \\
	&= \sum_{w\doteq uv} 
	\sum_{\substack{u^{-1}\doteq u_1v_1 \\ v_1\in S_{\ne0}}}
	\sum_{\substack{v\doteq u_2v_2 \\ v_2 \in S_{\ne0}}} 
	(-1)^{\ell(u)}
	F_{u_1}(a) \S_{v_1}(a) F_{u_2}(x) \S_{v_2}(x) \\
	&=\sum_{\substack{w\doteq v_1^{-1} u_1^{-1} u_2 v_2 \\ v_1,v_2\in S_{\ne0}}} (-1)^{\ell(u_1)+\ell(v_1)} \S_{v_1}(a) F_{u_1}(a) F_{u_2}(x) \S_{v_2}(x) \\
	&= \sum_{\substack{w\doteq v_1^{-1} u v_2 \\ v_1,v_2\in S_{\ne0}}} (-1)^{\ell(v_1)}\S_{v_1}(a) F_u(x/a) \S_{v_2}(x). \qedhere
	\end{align*}
\end{proof}

\begin{example} \label{X:backstable double Schub triple expansion}
We have 
\begin{align*}
	\bS_{s_i}(x;a) &= - \S_{s_i}(a) + F_{s_i}(x/a) = -\S_{s_i}(a) + s_1(x/a) = s_1[x_{\le0} - a_{\le i}] \\
	\bS_{s_1s_0}(x;a) &= - \S_{s_1}(a) F_{s_0}(x/a) + F_{s_1s_0}(x/a) = -a_1 s_1(x/a) + s_2(x/a) \\
	\bS_{s_{-1}s_0} &= -\S_{s_{-1}}(a) F_{s_0}(x/a) + F_{s_{-1}s_0}(x/a) = a_0 s_1(x/a) + s_{11}(x/a) \\
	\bS_{s_0s_{-1}} &= F_{s_0s_{-1}}(x/a) + F_{s_0} \S_{s_{-1}}(x) = s_2(x/a) + s_1(x/a) (-x_0).
\end{align*}
\end{example}

\begin{thm}\label{thm:backstabledouble}
	$\{\bS_w(x;a) \in \bR(x;a) \mid w \in S_\Z\}$ is the unique family of power series satisfying the following conditions:
	\begin{align}
	\label{E:backstable double identity}
	\bS_\id &= 1 \\
	\label{E:backstable double specialize}
	\bS_w(a;a) &= 0 \qquad\text{if $w\ne\id$} \\
	\label{E:backstable double partial}
	A_i \bS_w(x;a) &= \begin{cases} \bS_{w s_i}(x;a) & \text{if $ws_i < w$} \\
	0 & \text{otherwise.}
	\end{cases}
	\end{align}
	The elements $\{\bS_w(x;a)  \mid w \in S_\Z\}$ form a basis of $\bR(x;a)$ over $\Q[a]$.
\end{thm}
\begin{proof} Uniqueness follows as in the proof of Theorem \ref{T:backstable}. Since the double Schubert polynomials are related by divided differences, the corresponding fact \eqref{E:backstable double partial} also holds. For \eqref{E:backstable double specialize}, applying the map $\epsilon$ of \textsection \ref{SS:loc back symmetric} to \eqref{E:backstable triple def} and using \eqref{E:counit}, we have
$$\bS_w(a;a) = \sum_{\substack{w\doteq uz \\ u,z\in S_{\ne0}}} (-1)^{\ell(u)} \S_u(a) \S_z(a).$$
This is $0$ automatically if $w\not\in S_{\ne0}$. If $w\in S_{\ne0}\setminus\{\id\}$ then the vanishing follows from the straightforward generalization of Lemma \ref{L:Schub cancellation} 
to $\Schub_w$ for $w\in S_{\ne0}$.
	
The basis property follows from the fact that setting $a_i=0$ for all $i\in \Z$ gives $\bS_w(x,0)=\bS_w(x)$ and the latter are a basis of $\Lambda\otimes \Q[x]$.
\end{proof}

The back stable double Schubert polynomials localize the same way that ordinary double Schubert polynomials do in the following sense.

\begin{prop} \label{P:double backstable finite loc}
	Let $v,w\in S_\Z$ and let $[p,q]\subset\Z$ be an interval that contains all elements moved by $v$ and by $w$. Then $\bS_v(wa;a) = \Schub^{[p,q]}_v(wa;a)$.
\end{prop}
\begin{proof} By Corollary \ref{C:backstable double shift} we may assume that $[p,q]=[1,n]$ for some $n$ so that $v,w\in S_n$. We are specializing $x_i\mapsto a_{w(i)}$ for all $i$, and in particular $x_i\mapsto a_i$ for all $i\le 0$. Under this substitution all the super Stanley functions in \eqref{E:backstable triple def} vanish except those indexed by the identity.  Using Proposition \ref{prop:dSchub}, we have
	\begin{equation*}
	\bS_v(wa;a) = \sum_{\substack{v \doteq uz \\ u,z\in S_{\ne0}}} (-1)^{\ell(u)} \S_{u^{-1}}(a) \Schub_z(wa) 
	= \sum_{\substack{v \doteq uz \\ u,z\in S_n}} (-1)^{\ell(u)} \Schub_{u^{-1}}(\ap) \S_z(w\ap) 
	= \S_v(w\ap;\ap).\qedhere
	\end{equation*}
\end{proof}

Let $s_i^a$ and $A_i^a$ be the reflection and divided difference operators acting on the $a$-variables in both $\Q[a]$ and in $p_r(x||a)$.

\begin{prop}\label{P:left ddiff} For all $i\in \Z$ and $w \in S_\Z$, we have
	\begin{align*}
		A_i^a \bS_w(x;a) &= 
		\begin{cases}
			-\bS_{s_iw}(x;a) & \text{if $s_iw<w$} \\
			0 & \text{otherwise.}
		\end{cases}
	\end{align*}
\end{prop}
\begin{proof} This follows from \eqref{E:double double} and  \eqref{E:partial backstable}.
\end{proof}

\subsection{Double Schur functions}
\label{SS:double Schur}
We realize the double Schur functions (see \cite{M}) as the Grassmannian back stable double Schubert polynomials. As such our double Schur functions are symmetric in $\xm$. In Appendix \ref{S:pos to non} a precise dictionary is given which connects our conventions with the literature, which uses symmetric functions in $\xp$.

Let $\shift_a$ be the shift of all of the $a$-variables, that is, the $\Q$-algebra automorphism of $\La(x||a)$ given by 
\begin{align}\label{E:phi_a_def}
	\shift_a(a_i) &= a_{i+1} & \shift_a^{-1}(a_i) &= a_{i- 1} \\
	\shift_a(p_k(x||a)) &= p_k(x||a) - a_1^k &
	\shift_a^{-1}(p_k(x||a)) &= p_k(x||a) + a_0^k.
\end{align}
By definition $\shift_a$ leaves the $x$ variables alone.
For $\la\in\Par$ define the double Schur function $s_\la(x||a)\in\La(x||a)$ by 
\begin{align}
	h_r(x||a) &:= \shift_a^{r-1}(h_r(x/a)) &
	\label{E:double Schur det} 
	s_\la(x||a) &:= \det \shift_a^{1-j}(h_{\la_i-i+j}(x||a)).
\end{align}

\begin{example}\label{X:double schur}
\begin{align*}
	h_r(x||a) &= h_r(x_{\le0}/a_{\le r-1})\qquad\text{for $r\ge1$} \\
	s_{11}(x||a) &= \det \begin{pmatrix}
		h_1(x||a) & \gamma_a^{-1}(h_2(x||a)) \\
		h_0(x||a) & \gamma_a^{-1}(h_1(x||a)) 
	\end{pmatrix}\\
&= h_1(x_{\le0}/a_{\le0}) h_1(x_{\le0}/a_{\le-1}) - h_2(x_{\le0}/a_{\le 0}) \\
&= h_1(x/a) (h_1(x/a)+a_0) - h_2(x/a) \\
&= s_{11}(x/a) + a_0 s_1(x/a) = \bS_{s_{-1}s_0}(x;a)
\end{align*}
by Example \ref{X:backstable double Schub triple expansion} for
$w_{(1,1)}=s_{-1}s_0$. 
\end{example}


\begin{prop}\label{P:backstable double Grassmannian}  For $\la \in \Par$, we have $\bS_{w_\la}(x;a) = s_\la(x||a)$.
\end{prop}
\begin{proof} 
Using \cite[(2.21)]{M}, one may compute $s_\la(x||a)|_{w_\la}$ and show that $s_\la(x||a)|_{w_\mu}$ vanishes when $\la \not \subseteq \mu$ (see also \cite[Theorem 7]{LaSh2}).

The result then follows from the characterization of $\bS_{w_\la}(x;a)$ obtained by combining Proposition~\ref{P:GKM Schubert basis} and Theorem~\ref{thm:HTFl} below.
\end{proof}

\subsection{Double Stanley symmetric functions}
We introduce the double Stanley symmetric functions $F_w(x||a)$ for $w\in S_\Z$. If $w$ is a $321$-avoiding permutation, we recover Molev's skew double Schur function; see Appendix \ref{SS:Molev skew double}.

Let $\eta_a$ be the $\Q[a]$-algebra homomorphism $\Q[x,a] \to \Q[a]$ given by $x_i \mapsto a_i$.  This induces a $\Q[a]$-algebra map $1 \otimes \eta_a: \bR(x;a) \to \Lambda(x||a) \otimes_{\Q[a]} \Q[a] \cong \Lambda(x||a)$, which we simply denote by $\eta_a$ as well. 
 
\begin{remark} Analogously to $\eta_0$ in Remark \ref{R:eta knows}, the map $\eta_a$ substitutes $x_i\mapsto a_i$ for the $x_i$ generators of $\Q[x]$ but leaves the ``$x_i$ in $\Lambda(x||a)$" alone.
 \end{remark}
 
For $w \in S_\Z$, define the \defn{double Stanley symmetric function} $F_w(x||a) \in \Lambda(x||a)$ by
\begin{align}\label{E:double Stanley}
F_w(x||a) := \eta_a(\bS_w(x;a)).
\end{align}

\begin{proposition}\label{prop:double Stanley triple sum}
For $w \in S_\Z$, we have
\begin{align}\label{E:double Stanley triple sum}
F_w(x||a) = \sum_{\substack{w \doteq uvz\\ u,z\in S_{\ne0}}} (-1)^u \S_{u^{-1}}(a) F_v(x/a) \S_z(a)
\end{align}
\end{proposition}
\begin{proof} This follows from the definition \eqref{E:double Stanley} and Corollary \ref{C:backstable triple sum}.
\end{proof}

\begin{prop} \label{P:Grassmannian Stanley} For $\la\in\Par$, we have $F_{w_\la}(x||a) = s_\la(x||a)$.
\end{prop}
\begin{proof} $F_{w_\la}(x||a) = \eta_a(\bS_{w_\la})=\bS_{w_\la}=s_\la(x||a)$ 
by Proposition \ref{P:backstable double Grassmannian},
since $\eta_a$ is the identity when restricted to $\La(x||a)$.
\end{proof}

\subsection{Negative double Schubert polynomials}
Let $\omega$ be the involutive $\Q$-algebra automorphism of 
$\Q[x;a]$ given by $\omega(x_i)=-x_{1-i}$ and $\omega(a_i)=-a_{1-i}$ for $i\in\Z$. For $w\in S_-$, define the negative double Schubert polynomial $\S_w(\xm;\am)\in \Q[\xm,\am]$ by 
\begin{align}\label{E:negative double Schub}
  \S_w(\xm;\am) := \omega(\S_{\omega(w)}(\xp;\ap)) \qquad\text{for $w\in S_-$.}
\end{align}
Define $\S_w(x;a)\in \Q[x;a]$ for $w\in S_{\ne0}$ by
\begin{align}\label{E:nonzero double Schub}
  \S_w(x;a) := \S_u(\xp;\ap) \S_v(\xm;\am)\qquad\text{where $w=uv$ with $u\in S_+$ and $v\in S_-$.}
\end{align}

\begin{prop}\label{P:double to single}
For $w\in S_{\ne0}$, we have
\begin{align}
\label{E:little double to single} 
\S_w(x;a) &= \sum_{w\doteq uv} (-1)^u \S_{u^{-1}}(a) \S_v(x) \\
\label{E:little single to double}
\S_w(x) &= \sum_{w\doteq uv} \S_u(a) \S_v(x;a)
\end{align}
\end{prop}
\begin{proof} Equation \eqref{E:little double to single} is straightforwardly reduced to the case that $w\in S_+$, which is Proposition \ref{prop:dSchub}. Equation \ref{E:little single to double} follows from \eqref{E:little double to single} by
Corollary \ref{C:SchubInversion}.
\end{proof}

\subsection{Coproduct formula}
There is a coaction $\Delta: \bR(x;a) \to \Lambda(x||a) \otimes_{\Q[a]} \bR(x||a)$ of $\Lambda(x||a)$ on $\bR(x;a)$, defined by the comultiplication on the first factor of the tensor product $\Lambda(x||a) \otimes_{\Q[a]} \Q[x,a]$.  

\begin{thm}\label{thm:doublecoprod}
Let $w \in S_\Z$.  We have the coproduct formulae
\begin{align}
\label{E:backstable double big coprod}
\Delta(\bS_w(x;a)) &= \sum_{\substack{w \doteq uv }} F_u(x||a) \otimes \bS_v(x;a) \\
\label{E:backstable double small coprod}
\bS_w(x;a) &= \sum_{\substack{w \doteq uv \\ v \in S_{\neq 0}}} F_u(x||a) \; \S_v(x;a).
\end{align}
\end{thm}
\begin{proof} We first deduce \eqref{E:backstable double big coprod} from
\eqref{E:backstable double small coprod}. Using Corollary \ref{C:backstable triple sum}, Proposition \ref{prop:double Stanley triple sum} and Lemma \ref{L:Schub cancellation} we have
\begin{align*}
&\qquad\;\;\;\sum_{w \doteq uv }  F_u(x||a) \otimes \bS_v(x;a) \\
&= \sum_{\substack{w \doteq u_1v_1z_1u_2v_2z_2  \\ u_i,z_j \in S_{\ne 0}}} (-1)^{\ell(u_1)+\ell(u_2)} \S_{u_1^{-1}}(a) F_{v_1}(x/a) \S_{z_1}(a) \otimes \S_{u_2^{-1}}(a) F_{v_2}(x/a) \S_{z_2}(x) \\
&= \sum_{\substack{w \doteq u_1v_1v_2z_2  \\ u_1,z_2 \in S_{\ne 0}}} (-1)^{\ell(u_1)} \S_{u_1^{-1}}(a) F_{v_1}(x/a)  \otimes  F_{v_2}(x/a) \S_{z_2}(x) \\
&= \Delta(\sum_{\substack{w \doteq u_1vz_2  \\ u_1,z_2 \in S_{\ne 0}}} (-1)^{\ell(u_1)} \S_{u_1^{-1}}(a) F_v(x/a) \S_{z_2}(x)) \\
&= \Delta(\bS_w(x;a)).
\end{align*}
For \eqref{E:backstable double small coprod}, using Propositions \ref{prop:double Stanley triple sum} and \ref{P:double to single} we have
\begin{align*}
\sum_{\substack{w \doteq uv \\ v \in S_{\neq 0}}} F_u(x||a) \; \S_v(x;a)
&= \sum_{\substack{w \doteq u_1v_1z_1 v \\ u_1,z_1,v \in S_{\ne 0}}} (-1)^{\ell(u_1)} \S_{u_1{^{-1}}}(a) F_{v_1}(x/a) \S_{z_1}(a) \S_v(x;a) \\
&= \sum_{\substack{w \doteq u_1v_1z \\ u_1,z \in S_{\ne 0}}} (-1)^{\ell(u_1)} \S_{u_1{^{-1}}}(a) F_{v_1}(x/a) \S_z(x) \\
&= \bS_w(x;a). \qedhere
\end{align*}
\end{proof}

\begin{cor}\label{cor:coproddoubleStanley}
Let $w \in S_\Z$.  Then
$$
\Delta(F_w(x||a)) = \sum_{w \doteq uv} F_u(x||a) \otimes F_v(x||a).
$$
\end{cor}
\begin{proof}
We have $\Delta \circ \eta_a = (1 \otimes \eta_a) \circ \Delta$ acting on $\bR(x;a)$, where $(1 \otimes \eta_a)$ acts on $\Lambda(x||a) \otimes_{\Q[a]} \bR(x;a)$ by acting as $\eta_a$ on the second factor.  The result follows from \eqref{E:backstable double big coprod}.
\end{proof}

Recall the definition of $w_{\la/\mu}$ from \eqref{E:w skew}.

\begin{cor}\label{cor:coproddoubleSchur}
For $\la \in \Par$, we have
\begin{align}
\label{E:coprod double Schur}
\Delta(s_\la(x||a)) = \sum_{\mu \subset \la} F_{w_{\la/\mu}}(x||a) \otimes s_\mu(x||a).
\end{align}
\end{cor}
\begin{proof} Consider \eqref{E:backstable double big coprod} for $\bS_{w_\la}(x;a)=s_\la(x||a)$. Let $w_\la\doteq uv$. Since
$w_\la\in S_\Z^0$ it follows that $v\in S_\Z^0$. Let $\mu\in\Par$ be such that $\mu\subset\la$ and $w_\mu=v$. Then $u=w_{\la/\mu}$ and \eqref{E:coprod double Schur} follows.
\end{proof}

\subsection{Dynkin reversal}
Extend the $\Q$-algebra automorphism $\omega$ of $\Q[x;a]$ to $\bR(x;a)$ by $\omega(p_k(x||a)) = (-1)^{k-1}p_k(x||a)$.

\begin{prop}\label{P:omega double}We have
\begin{align} \label{E:omega loc}
	\omega(f)|_{\omega(v)} &= \omega(f|_v) &\text{for $f\in \bR(x;a)$ and $v\in S_\Z$.} \\
	\label{E:shift and negate}
	\omega(\shift_a(f)) &= \shift_a^{-1}(\omega(f))&\text{for $f\in \bR(x;a)$.}
\end{align}
Moreover,
\begin{align}
	\label{E:omega super Stanley}
	\omega(F_v(x/a)) &= F_{\omega(v)}(x/a)&\text{for $v \in S_\Z$.} \\
	\label{E:omega back stable}
	\omega(\bS_v(x;a))&= \bS_{\omega(v)}(x;a)&\text{for $v \in S_\Z$.} \\
	\label{E:omega double Stanley}
	\omega(F_v(x||a)) &= F_{\omega(v)}(x||a) &\text{for $v \in S_\Z$.} \\
	\label{E:omega double Schur}
	\omega(s_\la(x||a)) &= s_{\la'}(x||a)& \text{for $\la\in\Par$.}
\end{align}
\end{prop}
\begin{proof} It is straightforward to verify \eqref{E:omega loc} on $\Q$-algebra generators with the help of Lemma \ref{lem:locp}. 
Equation \eqref{E:shift and negate} is also easily verified on algebra generators.

Equation \eqref{E:omega super Stanley} follows from Proposition \ref{P:Stanley symmetries} 
by superization. Equations \eqref{E:omega back stable} and \eqref{E:omega double Stanley} follow by applying $\omega$ to the coproduct formulae \eqref{E:backstable double small coprod}
and Proposition \ref{prop:double Stanley triple sum}. Equation \eqref{E:omega double Schur} follows from \eqref{E:omega back stable} and Proposition \ref{P:backstable double Grassmannian} using $\omega(w_\la)=w_{\la'}$.

Alternatively, \eqref{E:omega back stable} follows from the uniqueness of the Schubert basis as defined by localizations. 
\end{proof}

\subsection{Double Edelman-Greene coefficients}
Define the \defn{double Edelman-Greene coefficients} $j_\lambda^w(a) \in \Q[a]$ by the equality
\begin{equation}\label{eq:doubleStanleydefn}
F_w(x||a) = \sum_\lambda j_\lambda^w(a) s_\lambda(x||a).
\end{equation}

\begin{lem}\label{lem:doubleid}
We have $j_\emptyset^w(a) = 0$ unless $w =\id$, and $j_\emptyset^\id = 1$.
\end{lem}
\begin{proof}
By Theorem \ref{thm:backstabledouble}, we have $\bS_w(a;a) = 0$ if $w \ne \id$ and $\bS_\id(a;a) = 1$.  The result follows by localizing both sides of \eqref{eq:doubleStanleydefn} at $\id$.
\end{proof}

\begin{example} We have $F_{s_{k+1}s_k}(x||a)=s_2(x||a)+(a_1-a_{k+1})s_1(x||a)$
	and $F_{s_{k-1}s_k}(x||a) = s_{11}(x||a) +(a_k-a_0)s_1(x||a)$ for all $k\in \Z$.  Thus $j_1^{s_{k+1}s_k}(a) = a_1 - a_{k+1}$ and $j_1^{s_{k-1}s_k}(a) = a_k - a_{0}$.
\end{example}

\begin{thm}\label{thm:doubleStanleypositivity}
Let $x \in S_\Z$ and $v \in S_\Z^0$.  Then $j^x_v(a) \in \Q[a]$ is a positive integer polynomial in the linear forms $a_i - a_j$ where $i \prec j$ under the total ordering of $\Z$ given by
$$
1 \prec 2 \prec 3  \prec \cdots\prec -2 \prec -1 \prec 0.
$$
\end{thm}
Theorem \ref{thm:doubleStanleypositivity} will be proven in \textsection \ref{ssec:positivityproof}.

Define the coproduct structure constants $\hat{c}^\la_{\mu\nu}(a)\in \Q[a]$ for $\la,\mu,\nu\in\Par$ by
\begin{align}\label{E:coproduct structure constants}
	\Delta(s_\la(x||a)) = \sum_{\mu,\nu\in\Par} \hat{c}^\la_{\mu\nu}(a)
	s_\mu(x||a) \otimes s_\nu(x||a).
\end{align}

\begin{proposition} \label{P:coproduct constants are Stanley coefficients}
For $\la,\mu,\nu\in \Par$, we have $\hat{c}^\la_{\mu\nu}(a) = j^{w_{\la/\mu}}_\nu(a)$.
\end{proposition}
\begin{proof} This holds by taking the coefficient of $s_\nu(x||a) \otimes s_\mu(x||a)$ in \eqref{E:coprod double Schur}.
\end{proof}

\begin{remark} In Theorem \ref{thm:homologyhook} we give a formula for $\hat{c}^\la_{\mu\nu}(a)$ which is positive in the sense of Theorem \ref{thm:doubleStanleypositivity} in the special case that $\mu$ or $\nu$ is a hook.
\end{remark}

Let $d(\la)$ be the Durfee square of $\la\in\Par$, the maximum index $d$ such that $\la_d \ge d$. The following change of basis coefficients between the double and super Schur bases were previously computed in \cite{ORV} expressed as a determinantal formula and in \cite{M} by a tableau formula. We give them as (signed) Schubert polynomials.

\begin{prop}\label{prop:doublesuper} For $\la \in \Par$, we have
	\begin{align}
		\label{E:double to super}
		s_\la(x||a) &= \sum_{\substack{\mu\subset\la\\ d(\mu)=d(\la)}} (-1)^{|\la/\mu|} \S_{w_{\la/\mu}^{-1}} (a) s_\mu(x/a) \\
		\label{E:super to double}
		s_\la(x/a) &= \sum_{\substack{\mu\subset\la\\ d(\mu)=d(\la)}} \S_{w_{\la/\mu}} (a) s_\mu(x||a) 
	\end{align}
\end{prop}
\begin{proof} Consider \eqref{E:backstable triple def} for $\bS_{w_\la}=s_\la(x||a)$. For $w_\la \doteq uvz$, arguing as in the proof of Corollary \ref{cor:coproddoubleSchur}, we first have $z\in S_\Z^0 \cap S_{\ne0} = \{\id\}$. Next we deduce that $v = w_\mu$ for some $\mu\in\Par$ such that $\mu\subset\la$. Thus $u=w_{\la/\mu}$. The condition $w_{\la/\mu}\in S_{\ne0}$ holds if and only if the skew shape $\la/\mu$ contains no boxes on the main diagonal, that is, $d(\la)=d(\mu)$. This proves \eqref{E:double to super}.
	
Equation \eqref{E:super to double} follows from \eqref{E:double to super} by Corollary \ref{C:SchubInversion}.
\end{proof}

\begin{example} \label{X:dual_Schur_box} Let $\mu=(1)$ so that $d_\mu=1$ and $w_\mu=s_0$. Consider the $\la$ such that $s_1(x||a)$ occurs in $s_\la(x/a)$. We must have $d(\la)=1$, that is, $\la$ is a hook $(p+1,1^q)$ for $p,q\ge0$. Then $\S_{w_{\la/\mu}}(a) = (-a_0)^q a_1^p$.
\end{example}

\section{Bumpless pipedreams}
We shall consider various versions of \defn{bumpless pipedreams}.
These are tilings of some region in the plane by the tiles: empty, NW elbow, SE elbow, horizontal line, crossing, and vertical line.
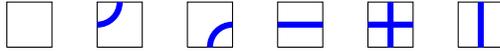
\begin{center}
\begin{tikzpicture}[scale=0.6,line width=0.8mm]
\bbox{-3}{0}
\bbox{-1}{0}
\leftelbow{-1}{0}{blue}
\bbox{1}{0}
\rightelbow{1}{0}{blue}
\bbox{3}{0}
\horline{3}{0}{blue}
\bbox{5}{0}
\cross{5}{0}{blue}{blue}
\bbox{7}{0}
\vertline{7}{0}{blue}
\end{tikzpicture}
 \end{center}
 We shall use {\bf matrix coordinates} for unit squares in the plane.  Thus row coordinates increase from top to bottom, column coordinates increase from left to right, and $(i,j)$ indicates the square in row $i$ and column $j$.
 
\subsection{$S_\Z$-bumpless pipedreams}
Let $w \in S_\Z$.  A \defn{$w$-bumpless pipedream} is a bumpless pipedream covering the whole plane, satisfying the following conditions:
\begin{enumerate}
\item there is a bijective labeling of pipes by integers;
\item
the pipe labeled $i$ eventually heads south in column $i$ and heads east in row $w^{-1}(i)$;
\item
two pipes cannot cross more than once;
\item
for all $N \gg 0$ and all $N \ll 0$, the pipe labeled $N$ travels north from $(\infty,N)$ to the square $(N,N)$ where it turns east and travels towards $(N,\infty)$.
\end{enumerate}
Because of condition (2), every pipe has to make at least one turn.  We call pipe $i$ \defn{standard} if it makes exactly one turn and this turn is at the diagonal square $(i,i)$.  By (4), all but finitely many pipes are standard.  We often omit standard pipes from our drawings of pipedreams.  The weight $\wt(P) := \prod (x_i - a_j)$ of a pipedream $P$ is the product of $x_i - a_j$ over all empty tiles $(i,j)$.

\begin{example} Let $w = s_3s_0s_1$.  In one line notation, $w(-2,-1,0,1,2,3,4) = (-2,-1,1,2,0,4,3)$ and the rest are fixed points.  Figure~\ref{fig:bumplessexample} shows a $w$-bumpless pipedream, where we have only drawn the region $\{(i,j) \mid i,j \in [-2,4]\}$.  In the left picture, the empty tiles have been indicated, as have the row and column numbers.  The label of a pipe is the {\it column number} to which its south end is attached.  In the right picture, we have indicated the labels of the pipes instead of the row numbers.  The one-line notation of $w$ can then be read off the east border.
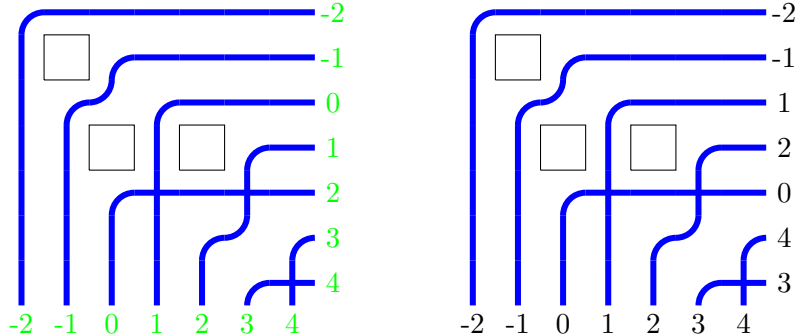
\begin{figure}
\begin{center}
\begin{tikzpicture}[scale=0.6,line width=0.8mm]
\rightelbow{-2}{2}{blue}
\horline{-1}{2}{blue}
\horline{0}{2}{blue}
\horline{1}{2}{blue}
\horline{2}{2}{blue}
\horline{3}{2}{blue}
\horline{4}{2}{blue}

\vertline{-2}{1}{blue}
\bbox{-1}{1}{blue}
\rightelbow{0}{1}{blue}
\horline{1}{1}{blue}
\horline{2}{1}{blue}
\horline{3}{1}{blue}
\horline{4}{1}{blue}

\vertline{-2}{0}{blue}
\rightelbow{-1}{0}{blue}
\leftelbow{0}{0}{blue}
\rightelbow{1}{0}{blue}
\horline{2}{0}{blue}
\horline{3}{0}{blue}
\horline{4}{0}{blue}

\vertline{-2}{-1}{blue}
\vertline{-1}{-1}{blue}
\bbox{0}{-1}{blue}
\vertline{1}{-1}{blue}
\bbox{2}{-1}{blue}
\rightelbow{3}{-1}{blue}
\horline{4}{-1}{blue}

\vertline{-2}{-2}{blue}
\vertline{-1}{-2}{blue}
\rightelbow{0}{-2}{blue}
\cross{1}{-2}{blue}{blue}
\horline{2}{-2}{blue}
\cross{3}{-2}{blue}{blue}
\horline{4}{-2}{blue}

\vertline{-2}{-3}{blue}
\vertline{-1}{-3}{blue}
\vertline{0}{-3}{blue}
\vertline{1}{-3}{blue}
\rightelbow{2}{-3}{blue}
\leftelbow{3}{-3}{blue}
\rightelbow{4}{-3}{blue}

\vertline{-2}{-4}{blue}
\vertline{-1}{-4}{blue}
\vertline{0}{-4}{blue}
\vertline{1}{-4}{blue}
\vertline{2}{-4}{blue}
\rightelbow{3}{-4}{blue}
\cross{4}{-4}{blue}{blue}

\draw[green] (-1.5,-4.4) node {-2};
\draw[green] (-0.5,-4.4) node {-1};
\draw[green] (0.5,-4.4) node {0};
\draw[green] (1.5,-4.4) node  {$1$};
\draw[green] (2.5,-4.4) node  {$2$};
\draw[green] (3.5,-4.4) node  {3};
\draw[green] (4.5,-4.4) node  {4};

\draw[green] (5.4,2.5) node {-2};
\draw[green] (5.4,1.5) node {-1};
\draw[green] (5.4,.5) node {0};
\draw[green] (5.4,-.5) node {1};
\draw[green] (5.4,-1.5) node {2};
\draw[green] (5.4,-2.5) node {3};
\draw[green] (5.4,-3.5) node {4};


\begin{scope}[shift={(10,0)}]
\rightelbow{-2}{2}{blue}
\horline{-1}{2}{blue}
\horline{0}{2}{blue}
\horline{1}{2}{blue}
\horline{2}{2}{blue}
\horline{3}{2}{blue}
\horline{4}{2}{blue}

\vertline{-2}{1}{blue}
\bbox{-1}{1}{blue}
\rightelbow{0}{1}{blue}
\horline{1}{1}{blue}
\horline{2}{1}{blue}
\horline{3}{1}{blue}
\horline{4}{1}{blue}

\vertline{-2}{0}{blue}
\rightelbow{-1}{0}{blue}
\leftelbow{0}{0}{blue}
\rightelbow{1}{0}{blue}
\horline{2}{0}{blue}
\horline{3}{0}{blue}
\horline{4}{0}{blue}

\vertline{-2}{-1}{blue}
\vertline{-1}{-1}{blue}
\bbox{0}{-1}{blue}
\vertline{1}{-1}{blue}
\bbox{2}{-1}{blue}
\rightelbow{3}{-1}{blue}
\horline{4}{-1}{blue}

\vertline{-2}{-2}{blue}
\vertline{-1}{-2}{blue}
\rightelbow{0}{-2}{blue}
\cross{1}{-2}{blue}{blue}
\horline{2}{-2}{blue}
\cross{3}{-2}{blue}{blue}
\horline{4}{-2}{blue}

\vertline{-2}{-3}{blue}
\vertline{-1}{-3}{blue}
\vertline{0}{-3}{blue}
\vertline{1}{-3}{blue}
\rightelbow{2}{-3}{blue}
\leftelbow{3}{-3}{blue}
\rightelbow{4}{-3}{blue}

\vertline{-2}{-4}{blue}
\vertline{-1}{-4}{blue}
\vertline{0}{-4}{blue}
\vertline{1}{-4}{blue}
\vertline{2}{-4}{blue}
\rightelbow{3}{-4}{blue}
\cross{4}{-4}{blue}{blue}

\draw (-1.5,-4.4) node {-2};
\draw (-0.5,-4.4) node {-1};
\draw (0.5,-4.4) node {0};
\draw (1.5,-4.4) node  {1};
\draw (2.5,-4.4) node  {2};
\draw (3.5,-4.4) node  {3};
\draw (4.5,-4.4) node  {4};

\draw (5.4,2.5) node {-2};
\draw (5.4,1.5) node {-1};
\draw (5.4,.5) node {1};
\draw (5.4,-.5) node {2};
\draw (5.4,-1.5) node {0};
\draw (5.4,-2.5) node {4};
\draw (5.4,-3.5) node {3};
\end{scope}
\end{tikzpicture}
\end{center}
\caption{A bumpless pipedream with weight $\wt = (x_{-1}-a_{-1})(x_1-a_0)(x_1-a_2)$.}
\label{fig:bumplessexample}
\end{figure}
\end{example}

\begin{thm}\label{T:bumpless}
Let $w \in S_\Z$.  Then $\bS_w(x;a) = \sum_P \wt(P)$ where the sum is over all $w$-bumpless pipedreams.
\end{thm}

The proof of Theorem~\ref{T:bumpless} is delayed to after Theorem~\ref{T:square bumpless}.

\subsection{Drooping and the Rothe pipedream}\label{ssec:droop}
A $w$-bumpless pipedream is uniquely determined by the location of the two kinds of elbow tiles.  Each pipe has to turn at least once.  There is a unique $w$-bumpless pipedream such that for all $i$, pipe $i$ turns right from south to east in the square $(w^{-1}(i),i)$.  We call this the {\it Rothe pipedream} $D(w)$ of $w$.  The empty tiles of the Rothe pipedream form what is commonly known as the {\it Rothe diagram} of $w$.

Let $P$ be a $w$-bumpless pipedream.  A {\it droop} is a local move that swaps a SE elbow $e$ with an empty tile $t$, when the SE elbow lies strictly to the northwest of the empty tile.  Let $R$ be the rectangle with northwest corner $e$ and southeast corner $t$ and let $p$ be the pipe passing through $e$.  After the droop, the pipe $p$ travels along the southmost row and eastmost column of $R$; a NW elbow occupies the square that used to be empty while the square that contained a SE elbow becomes empty.  The droop is allowed only if: 
\begin{enumerate}
\item
the westmost column and northmost row of $R$ contains $p$, 
\item
the rectangle $R$ contains only one elbow which is at $e$, and
\item
after the droop we obtain a bumpless pipedream $P'$.
\end{enumerate}
Pipes $p' \neq p$ do not move in a droop.  We denote a droop by $P \searrow P'$.  Figure~\ref{fig:Rothe} shows a Rothe bumpless pipedream followed by a sequence of two droops.
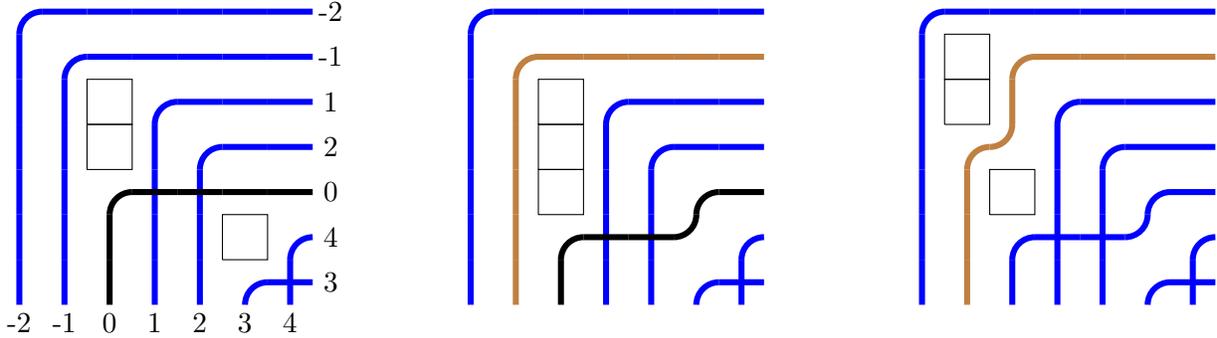
\begin{figure}
\begin{center}
\begin{tikzpicture}[scale=0.6,line width=0.8mm]
\rightelbow{-2}{2}{blue}
\horline{-1}{2}{blue}
\horline{0}{2}{blue}
\horline{1}{2}{blue}
\horline{2}{2}{blue}
\horline{3}{2}{blue}
\horline{4}{2}{blue}

\vertline{-2}{1}{blue}
\rightelbow{-1}{1}{blue}
\horline{0}{1}{blue}
\horline{1}{1}{blue}
\horline{2}{1}{blue}
\horline{3}{1}{blue}
\horline{4}{1}{blue}

\vertline{-2}{0}{blue}
\vertline{-1}{0}{blue}
\bbox{0}{0}{blue}
\rightelbow{1}{0}{blue}
\horline{2}{0}{blue}
\horline{3}{0}{blue}
\horline{4}{0}{blue}

\vertline{-2}{-1}{blue}
\vertline{-1}{-1}{blue}
\bbox{0}{-1}{blue}
\vertline{1}{-1}{blue}
\rightelbow{2}{-1}{blue}
\horline{3}{-1}{blue}
\horline{4}{-1}{blue}

\vertline{-2}{-2}{blue}
\vertline{-1}{-2}{blue}
\rightelbow{0}{-2}{black}
\cross{1}{-2}{blue}{black}
\cross{2}{-2}{blue}{black}
\horline{3}{-2}{black}
\horline{4}{-2}{black}

\vertline{-2}{-3}{blue}
\vertline{-1}{-3}{blue}
\vertline{0}{-3}{black}
\vertline{1}{-3}{blue}
\vertline{2}{-3}{blue}
\bbox{3}{-3}{blue}
\rightelbow{4}{-3}{blue}

\vertline{-2}{-4}{blue}
\vertline{-1}{-4}{blue}
\vertline{0}{-4}{black}
\vertline{1}{-4}{blue}
\vertline{2}{-4}{blue}
\rightelbow{3}{-4}{blue}
\cross{4}{-4}{blue}{blue}

\draw (-1.5,-4.4) node {-2};
\draw (-0.5,-4.4) node {-1};
\draw (0.5,-4.4) node {0};
\draw (1.5,-4.4) node  {1};
\draw (2.5,-4.4) node  {2};
\draw (3.5,-4.4) node  {3};
\draw (4.5,-4.4) node  {4};

\draw (5.4,2.5) node {-2};
\draw (5.4,1.5) node {-1};
\draw (5.4,.5) node {1};
\draw (5.4,-.5) node {2};
\draw (5.4,-1.5) node {0};
\draw (5.4,-2.5) node {4};
\draw (5.4,-3.5) node {3};
\begin{scope}[shift={(10,0)}]
\rightelbow{-2}{2}{blue}
\horline{-1}{2}{blue}
\horline{0}{2}{blue}
\horline{1}{2}{blue}
\horline{2}{2}{blue}
\horline{3}{2}{blue}
\horline{4}{2}{blue}

\vertline{-2}{1}{blue}
\rightelbow{-1}{1}{brown}
\horline{0}{1}{brown}
\horline{1}{1}{brown}
\horline{2}{1}{brown}
\horline{3}{1}{brown}
\horline{4}{1}{brown}

\vertline{-2}{0}{blue}
\vertline{-1}{0}{brown}
\bbox{0}{0}{blue}
\rightelbow{1}{0}{blue}
\horline{2}{0}{blue}
\horline{3}{0}{blue}
\horline{4}{0}{blue}

\vertline{-2}{-1}{blue}
\vertline{-1}{-1}{brown}
\bbox{0}{-1}{blue}
\vertline{1}{-1}{blue}
\rightelbow{2}{-1}{blue}
\horline{3}{-1}{blue}
\horline{4}{-1}{blue}

\vertline{-2}{-2}{blue}
\vertline{-1}{-2}{brown}
\bbox{0}{-2}{blue}
\vertline{1}{-2}{blue}
\vertline{2}{-2}{blue}
\rightelbow{3}{-2}{black}
\horline{4}{-2}{black}

\vertline{-2}{-3}{blue}
\vertline{-1}{-3}{brown}
\rightelbow{0}{-3}{black}
\cross{1}{-3}{blue}{black}
\cross{2}{-3}{blue}{black}
\leftelbow{3}{-3}{black}
\rightelbow{4}{-3}{blue}

\vertline{-2}{-4}{blue}
\vertline{-1}{-4}{brown}
\vertline{0}{-4}{black}
\vertline{1}{-4}{blue}
\vertline{2}{-4}{blue}
\rightelbow{3}{-4}{blue}
\cross{4}{-4}{blue}{blue}

\end{scope}

\begin{scope}[shift={(20,0)}]
\rightelbow{-2}{2}{blue}
\horline{-1}{2}{blue}
\horline{0}{2}{blue}
\horline{1}{2}{blue}
\horline{2}{2}{blue}
\horline{3}{2}{blue}
\horline{4}{2}{blue}

\vertline{-2}{1}{blue}
\bbox{-1}{1}{blue}
\rightelbow{0}{1}{brown}
\horline{1}{1}{brown}
\horline{2}{1}{brown}
\horline{3}{1}{brown}
\horline{4}{1}{brown}

\vertline{-2}{0}{blue}
\bbox{-1}{0}{blue}
\vertline{0}{0}{brown}
\rightelbow{1}{0}{blue}
\horline{2}{0}{blue}
\horline{3}{0}{blue}
\horline{4}{0}{blue}

\vertline{-2}{-1}{blue}
\rightelbow{-1}{-1}{brown}
\leftelbow{0}{-1}{brown}
\vertline{1}{-1}{blue}
\rightelbow{2}{-1}{blue}
\horline{3}{-1}{blue}
\horline{4}{-1}{blue}

\vertline{-2}{-2}{blue}
\vertline{-1}{-2}{brown}
\bbox{0}{-2}{blue}
\vertline{1}{-2}{blue}
\vertline{2}{-2}{blue}
\rightelbow{3}{-2}{blue}
\horline{4}{-2}{blue}

\vertline{-2}{-3}{blue}
\vertline{-1}{-3}{brown}
\rightelbow{0}{-3}{blue}
\cross{1}{-3}{blue}{blue}
\cross{2}{-3}{blue}{blue}
\leftelbow{3}{-3}{blue}
\rightelbow{4}{-3}{blue}

\vertline{-2}{-4}{blue}
\vertline{-1}{-4}{brown}
\vertline{0}{-4}{blue}
\vertline{1}{-4}{blue}
\vertline{2}{-4}{blue}
\rightelbow{3}{-4}{blue}
\cross{4}{-4}{blue}{blue}

\end{scope}
\end{tikzpicture}
\end{center}
\caption{A Rothe bumpless pipedream $P$, and a sequence of two droops.}
\label{fig:Rothe}
\end{figure}

\begin{proposition}
Every $w$-bumpless pipedream can be obtained from the Rothe pipedream $D(w)$ of $w$ by a sequence of droops.
\end{proposition}
\begin{proof}
Let $P'$ be a $w$-bumpless pipedream which is not the Rothe pipedream and $e$ be a NW elbow that is northwestmost among NW elbows in $P'$.  Let $p$ be the pipe passing through $e$.  Then $p$ passes through SE elbows $f$ (resp. $f'$) in the same row (resp. column) as $e$.  Let $R$ be the rectangle bounded by $e,f,f'$ with northwest corner $t$.  It is easy to see that $t$ must be an empty tile and $R$ does not contain any other elbows.  Thus there is a droop $P \searrow P'$ which occurs in the rectangle $R$, and $P$ has strictly fewer NW elbows than $P'$.  Repeating, we eventually arrive at the Rothe pipedream.
\end{proof}

\begin{cor}
The number of empty tiles in a $w$-bumpless pipedream is equal to $\ell(w)$.
\end{cor}


\subsection{Halfplane crossless pipedreams}
Let $P$ be an $S_\Z$-bumpless pipedream.  By (4) of the definition, only finitely many crossings appear.  If we cut off, using a horizontal line, the bottom part of $P$ containing all crossing tiles, we will obtain a picture that we call a halfplane crossless pipedream.  It turns out that the double Schur function is a generating function of such pipedreams.

For $\la\in\Par$, a \defn{$\lambda$-halfplane pipedream} is a bumpless pipedream in the upper halfplane $H=\Z_{\le0}\times \Z$ such that the crossing tile is not used, and
\begin{enumerate}
\item
there are (unlabeled) pipes entering from the southern boundary in the columns indexed by $I \subset \Z$;
\item
setting $(I_+,I_-) = (I \cap \Z_{>0}, \Z_{>0} \setminus I)$, we have
$I_{\pm} = I_{w_\la,\pm}$ (see \eqref{E:I+}, \eqref{E:I-}, and \eqref{E:wla});
\item
the $i$-th eastmost pipe entering from the south heads off to the east in row $1-i$.  (Equivalently, for every row $i \in \Z_{\leq 0}$, there is some pipe heading towards $(i,\infty)$.)
\end{enumerate}
As before, the weight of a $\lambda$-halfplane pipedream is $\wt(P) = \prod (x_i - a_j)$, where the product is over all empty tiles $(i,j)$ in the halfplane $H$. 

For example, taking $\lambda = (2,1,1)$ we have $w_\lambda = s_{-2}s_{-1}s_1 s_0$ and $(I_+,I_-) = (\{2\},\{-2\})$.  Figure~\ref{fig:halfplane} shows a $\lambda$-halfplane pipedream.
\begin{figure}
 \begin{center}
\begin{tikzpicture}[scale=0.6,line width=0.8mm]
\rightelbow{-4}{4}{\half}
\horline{-3}{4}{\half}
\horline{-2}{4}{\half}
\horline{-1}{4}{\half}
\horline{0}{4}{\half}
\horline{1}{4}{\half}
\horline{2}{4}{\half}
\horline{3}{4}{\half}

\vertline{-4}{3}{\half}
\bbox{-3}{3}
\rightelbow{-2}{3}{\half}
\horline{-1}{3}{\half}
\horline{0}{3}{\half}
\horline{1}{3}{\half}
\horline{2}{3}{\half}
\horline{3}{3}{\half}

\vertline{-4}{2}{\half}
\bbox{-3}{2}
\vertline{-2}{2}{\half}
\rightelbow{-1}{2}{\half}
\horline{0}{2}{\half}
\horline{1}{2}{\half}
\horline{2}{2}{\half}
\horline{3}{2}{\half}

\vertline{-4}{1}{\half}
\rightelbow{-3}{1}{\half}
\leftelbow{-2}{1}{\half}
\vertline{-1}{1}{\half}
\bbox{0}{1}
\rightelbow{1}{1}{\half}
\horline{2}{1}{\half}
\horline{3}{1}{\half}

\vertline{-4}{0}{\half}
\vertline{-3}{0}{\half}
\bbox{-2}{0}
\vertline{-1}{0}{\half}
\rightelbow{0}{0}{\half}
\leftelbow{1}{0}{\half}
\rightelbow{2}{0}{\half}
\horline{3}{0}{\half}

\draw[green] (-3.5,-.4) node {-4};
\draw[green] (-2.5,-.4) node {-3};
\draw[green] (-1.5,-.4) node {-2};
\draw[green] (-.5,-.4) node  {-1};
\draw[green] (0.5,-.4) node  {0};
\draw[green] (1.5,-.4) node  {1};
\draw[green] (2.5,-.4) node  {2};
\draw[green] (3.5,-.4) node  {3};

\draw[green] (4.4,4.5) node {-4};
\draw[green] (4.4,3.5) node {-3};
\draw[green] (4.4,2.5) node {-2};
\draw[green] (4.4,1.5) node {-1};
\draw[green] (4.4,.5) node {0};

\draw [line width=0.4mm] (-4,0)--(5,0);
\end{tikzpicture}
\end{center}
\caption{A $(2,1,1)$-halfplane pipedream with weight $\wt = (x_{-3}-a_{-3})(x_{-2}-a_{-3})(x_{-1}-a_0)(x_0-a_{-2})$.}
\label{fig:halfplane}
\end{figure}

\begin{lemma}
The number of empty tiles in a $\lambda$-halfplane pipedream is equal to $|\lambda|$.
\end{lemma}

\begin{thm}\label{T:lambdabumpless}
Let $\lambda$ be a partition.  Then $s_\lambda(x||a) = \sum_P \wt(P)$ where the sum is over all $\lambda$-halfplane pipedreams.
\end{thm}
The proof of Theorem~\ref{T:lambdabumpless} is delayed to after Theorem~\ref{T:square bumpless}.

A \defn{$\Z_{\leq 0}$-semistandard Young tableau} of shape $\lambda$ is a filling of the Young diagram $\lambda$ (in English notation) with the integers $0,-1,-2,\ldots$ such that rows are weakly increasing and columns are strictly increasing.  The weight $\wt(T)$ of a $\Z_{\leq 0}$-SSYT is the product $\wt(T) = \prod_{(i,j) \in T} (x_{T(i,j)} - a_{T(i,j)+c(i,j)})$ where $c(i,j) = j-i$ is the content of the square $(i,j)$ in row $i$ and column $j$.

\begin{cor}
Let $\lambda$ be a partition.  Then $s_\lambda(x||a) = \sum_T \wt(T)$ where the sum is over all $\Z_{\leq 0}$-SSYT of shape $\lambda$.
\end{cor}

\subsection{Rectangular $S_n$-bumpless pipedreams}
Let $w \in S_n$.  A \defn{$w$-rectangular bumpless pipedream} is a bumpless pipedream in the $n \times 2n$ rectangular region 
$$R_n := \{(i,j) \mid i \in [1,n] \text{ and } j \in [1-n,n]\}.
$$  
 The pipes are labeled $1-n,2-n,\ldots,0,1,\ldots,n$, entering the south boundary from left to right.  The positively labeled pipes exit the east boundary: pipe $i$ exits  in row $i$.  The nonpositively labeled pipes exit the north boundary.  Two pipes intersect at most once, and a nonpositively labeled pipe cannot intersect any other pipe.   As before, the weight of a rectangular $S_n$-bumpless pipedream $P$ is given by $\wt(P) = \prod (x_i - a_j)$, with the product  over all empty tiles $(i,j)$.
 
\begin{lemma}
Let $w \in S_n$.  Suppose $P$ is an $S_\Z$-bumpless pipedream for $w$ (considered an element of $S_\Z$.  Then the region inside the rectangle $R_n$ is an $S_n$-rectangular bumpless pipedream for $w$.
\end{lemma}

We also associate a partition $\lambda(P)$ to an $S_n$-rectangular bumpless pipedream: it is obtained by reading the North boundary edges from right to left, to then obtain the boundary of a partition inside a $n \times n$ box, where empty edges correspond to steps to the left, and edges with a pipe exiting correspond to downward steps. See Figure \ref{F:lambda(P)}, where empty edges are marked $e$ and edges with a pipe exiting are marked $x$.

\begin{lemma}
Let $w \in S_n$ and $P$ be a $w$-bumpless pipedream.
We have $\ell(w) = |\lambda(P)| + \deg(\wt(P))$. 
\end{lemma}

\begin{example}\label{ex:2143}
Let $w = 2143$. In Figure \ref{F:pipedreams} is a complete list of all $w$-bumpless pipedreams.  Nonpositively labeled pipes are drawn in red.

\begin{figure} 
\begin{center}
\begin{tikzpicture}[scale=0.6,line width=0.8mm]
\draw (-5,-1) node {$P=$};
\rightelbow{1}{0}{blue}
\horline{2}{0}{blue}
\horline{3}{0}{blue}

\vertline{1}{-1}{blue}
\rightelbow{2}{-1}{blue}
\horline{3}{-1}{blue}

\rightelbow{0}{-2}{blue}
\cross{1}{-2}{blue}{blue}
\leftelbow{2}{-2}{blue}
\rightelbow{3}{-2}{blue}

\vertline{0}{-3}{blue}
\vertline{1}{-3}{blue}
\rightelbow{2}{-3}{blue}
\cross{3}{-3}{blue}{blue}

\vertline{-4}{0}{red}

\vertline{-3}{0}{red}
\vertline{-3}{-1}{red}

\rightelbow{-2}{0}{red}
\vertline{-2}{-1}{red}
\vertline{-2}{-2}{red}

\leftelbow{-1}{0}{red}
\rightelbow{-1}{-1}{red}
\vertline{-1}{-2}{red}
\vertline{-1}{-3}{red}

\vertline{0}{0}{red}
\leftelbow{0}{-1}{red}

\draw[line width=0.4mm] (0,-3)--(4,-3)--(4,1)--(0,1);
\draw[dashed,line width=0.3mm] (0,1)--(0,-3);
\draw[line width=0.4mm] (0,1)--(-4,1)--(-4,-3)--(0,-3);
\vertline{-4}{-1}{red}
\vertline{-4}{-2}{red}
\vertline{-4}{-3}{red}
\vertline{-3}{-2}{red}
\vertline{-3}{-3}{red}
\vertline{-2}{-3}{red}

\draw (3.5,1.4) node {$e_1$};
\draw (2.5,1.4) node {$e_2$};
\draw (1.5,1.4) node {$e_3$};
\draw (0.5,1.4) node {$x_1$};
\draw (-0.5,1.4) node {$x_2$};
\draw (-1.5,1.4) node {$e_4$};
\draw (-2.5,1.4) node {$x_3$};
\draw (-3.5,1.4) node {$x_4$};
\end{tikzpicture}
\qquad
\begin{tikzpicture}[scale=0.6,line width=0.8mm]
\draw[line width=0.4mm] (0,1)--(-4,1)--(-4,-3); 
\draw[line width=0.4mm] (-3,1)--(-3,-1)--(-4,-1);
\draw (-0.5,1.4) node {$e_1$};
\draw (-1.5,1.4) node {$e_2$};
\draw (-2.5,1.4) node {$e_3$};
\draw (-3.5,-0.6) node {$e_4$};
\draw (-2.5,0.5) node {$x_1$};
\draw (-2.5,-0.5) node {$x_2$};
\draw (-3.5,-1.5) node {$x_3$};
\draw (-3.5,-2.5) node {$x_4$};
\draw (2.5,-1) node {$\lambda(P)=(1,1)$};
\end{tikzpicture}
\end{center}
\caption{The partition of a rectangular $S_n$-bumpless pipedream}
\label{F:lambda(P)}
\end{figure}
\begin{figure}

\begin{center}
\begin{tikzpicture}[scale=0.6,line width=0.8mm]
\rightelbow{1}{0}{blue}
\horline{2}{0}{blue}
\horline{3}{0}{blue}

\vertline{1}{-1}{blue}
\rightelbow{2}{-1}{blue}
\horline{3}{-1}{blue}

\rightelbow{0}{-2}{blue}
\cross{1}{-2}{blue}{blue}
\leftelbow{2}{-2}{blue}
\rightelbow{3}{-2}{blue}

\vertline{0}{-3}{blue}
\vertline{1}{-3}{blue}
\rightelbow{2}{-3}{blue}
\cross{3}{-3}{blue}{blue}
\vertline{-4}{0}{red}

\vertline{-3}{0}{red}
\vertline{-3}{-1}{red}

\rightelbow{-2}{0}{red}
\vertline{-2}{-1}{red}
\vertline{-2}{-2}{red}

\leftelbow{-1}{0}{red}
\rightelbow{-1}{-1}{red}
\vertline{-1}{-2}{red}
\vertline{-1}{-3}{red}

\vertline{0}{0}{red}
\leftelbow{0}{-1}{red}

\draw[line width=0.4mm] (0,-3)--(4,-3)--(4,1)--(0,1);
\draw[dashed,line width=0.3mm] (0,1)--(0,-3);
\draw[line width=0.4mm] (0,1)--(-4,1)--(-4,-3)--(0,-3);
\vertline{-4}{-1}{red}
\vertline{-4}{-2}{red}
\vertline{-4}{-3}{red}
\vertline{-3}{-2}{red}
\vertline{-3}{-3}{red}
\vertline{-2}{-3}{red}

\draw (4.4,0.5) node {$2$};
\draw (4.4,-0.5) node {$1$};
\draw (4.4,-1.5) node {$4$};
\draw (4.4,-2.5) node {$3$};
\draw (0.5,-3.4) node  {$1$};
\draw (1.5,-3.4) node  {$2$};
\draw (2.5,-3.4) node  {3};
\draw (3.5,-3.4) node  {4};
\draw (-5.5,-1) node {$\lambda = (11)$};
\draw (-5.5,-2) node {$1$};
\begin{scope}[shift={(13,0)}]
\rightelbow{1}{0}{blue}
\horline{2}{0}{blue}
\horline{3}{0}{blue}

\vertline{1}{-1}{blue}
\rightelbow{2}{-1}{blue}
\horline{3}{-1}{blue}

\rightelbow{0}{-2}{blue}
\cross{1}{-2}{blue}{blue}
\leftelbow{2}{-2}{blue}
\rightelbow{3}{-2}{blue}

\vertline{0}{-3}{blue}
\vertline{1}{-3}{blue}
\rightelbow{2}{-3}{blue}
\cross{3}{-3}{blue}{blue}
\vertline{-4}{0}{red}

\vertline{-3}{0}{red}
\vertline{-3}{-1}{red}

\vertline{-2}{0}{red}
\vertline{-2}{-1}{red}
\vertline{-2}{-2}{red}

\rightelbow{-1}{-1}{red}
\vertline{-1}{-2}{red}
\vertline{-1}{-3}{red}

\vertline{0}{0}{red}
\leftelbow{0}{-1}{red}

\draw[line width=0.4mm] (0,-3)--(4,-3)--(4,1)--(0,1);
\draw[dashed,line width=0.3mm] (0,1)--(0,-3);
\draw[line width=0.4mm] (0,1)--(-4,1)--(-4,-3)--(0,-3);
\vertline{-4}{-1}{red}
\vertline{-4}{-2}{red}
\vertline{-4}{-3}{red}
\vertline{-3}{-2}{red}
\vertline{-3}{-3}{red}
\vertline{-2}{-3}{red}

\draw (4.4,0.5) node {$2$};
\draw (4.4,-0.5) node {$1$};
\draw (4.4,-1.5) node {$4$};
\draw (4.4,-2.5) node {$3$};
\draw (0.5,-3.4) node  {$1$};
\draw (1.5,-3.4) node  {$2$};
\draw (2.5,-3.4) node  {3};
\draw (3.5,-3.4) node  {4};
\draw (-5.5,-1) node {$\lambda = (1)$};
\draw (-5.5,-2) node {$ x_1-a_0$};
\end{scope}

%
\begin{scope}[shift={(0,-6)}]
\rightelbow{1}{0}{blue}
\horline{2}{0}{blue}
\horline{3}{0}{blue}

\vertline{1}{-1}{blue}
\rightelbow{2}{-1}{blue}
\horline{3}{-1}{blue}

\rightelbow{0}{-2}{blue}
\cross{1}{-2}{blue}{blue}
\leftelbow{2}{-2}{blue}
\rightelbow{3}{-2}{blue}

\vertline{0}{-3}{blue}
\vertline{1}{-3}{blue}
\rightelbow{2}{-3}{blue}
\cross{3}{-3}{blue}{blue}
\vertline{-4}{0}{red}

\vertline{-3}{0}{red}
\vertline{-3}{-1}{red}

\vertline{-2}{0}{red}
\vertline{-2}{-1}{red}
\vertline{-2}{-2}{red}

\rightelbow{-1}{0}{red}
\vertline{-1}{-1}{red}
\vertline{-1}{-2}{red}
\vertline{-1}{-3}{red}

\leftelbow{0}{0}{red}

\draw[line width=0.4mm] (0,-3)--(4,-3)--(4,1)--(0,1);
\draw[dashed,line width=0.3mm] (0,1)--(0,-3);
\draw[line width=0.4mm] (0,1)--(-4,1)--(-4,-3)--(0,-3);
\vertline{-4}{-1}{red}
\vertline{-4}{-2}{red}
\vertline{-4}{-3}{red}
\vertline{-3}{-2}{red}
\vertline{-3}{-3}{red}
\vertline{-2}{-3}{red}

\draw (4.4,0.5) node {$2$};
\draw (4.4,-0.5) node {$1$};
\draw (4.4,-1.5) node {$4$};
\draw (4.4,-2.5) node {$3$};
\draw (0.5,-3.4) node  {$1$};
\draw (1.5,-3.4) node  {$2$};
\draw (2.5,-3.4) node  {3};
\draw (3.5,-3.4) node  {4};
\draw (-5.5,-1) node {$\lambda = (1)$};
\draw (-5.5,-2) node {$x_2-a_1$};
\end{scope}
\begin{scope}[shift={(13,-6)}]
\rightelbow{1}{0}{blue}
\horline{2}{0}{blue}
\horline{3}{0}{blue}

\vertline{1}{-1}{blue}
\rightelbow{2}{-1}{blue}
\horline{3}{-1}{blue}

\rightelbow{0}{-2}{blue}
\cross{1}{-2}{blue}{blue}
\leftelbow{2}{-2}{blue}
\rightelbow{3}{-2}{blue}

\vertline{0}{-3}{blue}
\vertline{1}{-3}{blue}
\rightelbow{2}{-3}{blue}
\cross{3}{-3}{blue}{blue}
\vertline{-4}{0}{red}

\vertline{-3}{0}{red}
\vertline{-3}{-1}{red}

\vertline{-2}{0}{red}
\vertline{-2}{-1}{red}
\vertline{-2}{-2}{red}

\vertline{-1}{0}{red}
\vertline{-1}{-1}{red}
\vertline{-1}{-2}{red}
\vertline{-1}{-3}{red}


\draw[line width=0.4mm] (0,-3)--(4,-3)--(4,1)--(0,1);
\draw[dashed,line width=0.3mm] (0,1)--(0,-3);
\draw[line width=0.4mm] (0,1)--(-4,1)--(-4,-3)--(0,-3);
\vertline{-4}{-1}{red}
\vertline{-4}{-2}{red}
\vertline{-4}{-3}{red}
\vertline{-3}{-2}{red}
\vertline{-3}{-3}{red}
\vertline{-2}{-3}{red}

\draw (4.4,0.5) node {$2$};
\draw (4.4,-0.5) node {$1$};
\draw (4.4,-1.5) node {$4$};
\draw (4.4,-2.5) node {$3$};
\draw (0.5,-3.4) node  {$1$};
\draw (1.5,-3.4) node  {$2$};
\draw (2.5,-3.4) node  {3};
\draw (3.5,-3.4) node  {4};
\draw (-5.5,0) node {$\lambda = \emptyset$};
\draw (-5.5,-1) node {$(x_1-a_1)$};
\draw (-5.5,-2) node {$(x_2-a_1)$};
\end{scope}
%
%
\begin{scope}[shift={(0,-12)}]
\rightelbow{2}{0}{blue}
\horline{3}{0}{blue}

\rightelbow{0}{-1}{blue}
\horline{1}{-1}{blue}
\cross{2}{-1}{blue}{blue}
\horline{3}{-1}{blue}

\vertline{0}{-2}{blue}
\rightelbow{1}{-2}{blue}
\leftelbow{2}{-2}{blue}
\rightelbow{3}{-2}{blue}

\vertline{0}{-3}{blue}
\vertline{1}{-3}{blue}
\rightelbow{2}{-3}{blue}
\cross{3}{-3}{blue}{blue}
\vertline{-4}{0}{red}

\vertline{-3}{0}{red}
\vertline{-3}{-1}{red}

\vertline{-2}{0}{red}
\vertline{-2}{-1}{red}
\vertline{-2}{-2}{red}

\rightelbow{-1}{0}{red}
\vertline{-1}{-1}{red}
\vertline{-1}{-2}{red}
\vertline{-1}{-3}{red}

\horline{0}{0}{red}
\leftelbow{1}{0}{red}

\draw[line width=0.4mm] (0,-3)--(4,-3)--(4,1)--(0,1);
\draw[dashed,line width=0.3mm] (0,1)--(0,-3);
\draw[line width=0.4mm] (0,1)--(-4,1)--(-4,-3)--(0,-3);
\vertline{-4}{-1}{red}
\vertline{-4}{-2}{red}
\vertline{-4}{-3}{red}
\vertline{-3}{-2}{red}
\vertline{-3}{-3}{red}
\vertline{-2}{-3}{red}

\draw (4.4,0.5) node {$2$};
\draw (4.4,-0.5) node {$1$};
\draw (4.4,-1.5) node {$4$};
\draw (4.4,-2.5) node {$3$};
\draw (0.5,-3.4) node  {$1$};
\draw (1.5,-3.4) node  {$2$};
\draw (2.5,-3.4) node  {3};
\draw (3.5,-3.4) node  {4};
\draw (-5.5,-1) node {$\lambda = (2)$};
\draw (-5.5,-2) node {$1$};
\end{scope}
\begin{scope}[shift={(13,-12)}]
\rightelbow{2}{0}{blue}
\horline{3}{0}{blue}

\rightelbow{0}{-1}{blue}
\horline{1}{-1}{blue}
\cross{2}{-1}{blue}{blue}
\horline{3}{-1}{blue}

\vertline{0}{-2}{blue}
\rightelbow{1}{-2}{blue}
\leftelbow{2}{-2}{blue}
\rightelbow{3}{-2}{blue}

\vertline{0}{-3}{blue}
\vertline{1}{-3}{blue}
\rightelbow{2}{-3}{blue}
\cross{3}{-3}{blue}{blue}
\vertline{-4}{0}{red}

\vertline{-3}{0}{red}
\vertline{-3}{-1}{red}

\vertline{-2}{0}{red}
\vertline{-2}{-1}{red}
\vertline{-2}{-2}{red}

\rightelbow{-1}{0}{red}
\vertline{-1}{-1}{red}
\vertline{-1}{-2}{red}
\vertline{-1}{-3}{red}

\leftelbow{0}{0}{red}

\draw[line width=0.4mm] (0,-3)--(4,-3)--(4,1)--(0,1);
\draw[dashed,line width=0.3mm] (0,1)--(0,-3);
\draw[line width=0.4mm] (0,1)--(-4,1)--(-4,-3)--(0,-3);
\vertline{-4}{-1}{red}
\vertline{-4}{-2}{red}
\vertline{-4}{-3}{red}
\vertline{-3}{-2}{red}
\vertline{-3}{-3}{red}
\vertline{-2}{-3}{red}

\draw (4.4,0.5) node {$2$};
\draw (4.4,-0.5) node {$1$};
\draw (4.4,-1.5) node {$4$};
\draw (4.4,-2.5) node {$3$};
\draw (0.5,-3.4) node  {$1$};
\draw (1.5,-3.4) node  {$2$};
\draw (2.5,-3.4) node  {3};
\draw (3.5,-3.4) node  {4};
\draw (-5.5,-1) node {$\lambda = (1)$};
\draw (-5.5,-2) node {$x_1-a_2$};
\end{scope}
%
%
%
\begin{scope}[shift={(0,-18)}]
\rightelbow{2}{0}{blue}
\horline{3}{0}{blue}

\rightelbow{0}{-1}{blue}
\horline{1}{-1}{blue}
\cross{2}{-1}{blue}{blue}
\horline{3}{-1}{blue}

\vertline{0}{-2}{blue}
\rightelbow{1}{-2}{blue}
\leftelbow{2}{-2}{blue}
\rightelbow{3}{-2}{blue}

\vertline{0}{-3}{blue}
\vertline{1}{-3}{blue}
\rightelbow{2}{-3}{blue}
\cross{3}{-3}{blue}{blue}
\vertline{-4}{0}{red}

\vertline{-3}{0}{red}
\vertline{-3}{-1}{red}

\vertline{-2}{0}{red}
\vertline{-2}{-1}{red}
\vertline{-2}{-2}{red}

\vertline{-1}{0}{red}
\vertline{-1}{-1}{red}
\vertline{-1}{-2}{red}
\vertline{-1}{-3}{red}

\draw[line width=0.4mm] (0,-3)--(4,-3)--(4,1)--(0,1);
\draw[dashed,line width=0.3mm] (0,1)--(0,-3);
\draw[line width=0.4mm] (0,1)--(-4,1)--(-4,-3)--(0,-3);
\vertline{-4}{-1}{red}
\vertline{-4}{-2}{red}
\vertline{-4}{-3}{red}
\vertline{-3}{-2}{red}
\vertline{-3}{-3}{red}
\vertline{-2}{-3}{red}

\draw (4.4,0.5) node {$2$};
\draw (4.4,-0.5) node {$1$};
\draw (4.4,-1.5) node {$4$};
\draw (4.4,-2.5) node {$3$};
\draw (0.5,-3.4) node  {$1$};
\draw (1.5,-3.4) node  {$2$};
\draw (2.5,-3.4) node  {3};
\draw (3.5,-3.4) node  {4};
\draw (-5.5,0) node {$\lambda = \emptyset$};
\draw (-5.5,-1) node {$(x_1-a_1)$};
\draw (-5.5,-2) node {$(x_1-a_2)$};
\end{scope}
\begin{scope}[shift={(13,-18)}]
\rightelbow{1}{0}{blue}
\horline{2}{0}{blue}
\horline{3}{0}{blue}

\rightelbow{0}{-1}{blue}
\cross{1}{-1}{blue}{blue}
\horline{2}{-1}{blue}{blue}
\horline{3}{-1}{blue}

\vertline{0}{-2}{blue}
\vertline{1}{-2}{blue}
\rightelbow{3}{-2}{blue}

\vertline{0}{-3}{blue}
\vertline{1}{-3}{blue}
\rightelbow{2}{-3}{blue}
\cross{3}{-3}{blue}{blue}
\vertline{-4}{0}{red}

\vertline{-3}{0}{red}
\vertline{-3}{-1}{red}

\vertline{-2}{0}{red}
\vertline{-2}{-1}{red}
\vertline{-2}{-2}{red}

\rightelbow{-1}{0}{red}
\vertline{-1}{-1}{red}
\vertline{-1}{-2}{red}
\vertline{-1}{-3}{red}

\leftelbow{0}{0}{red}

\draw[line width=0.4mm] (0,-3)--(4,-3)--(4,1)--(0,1);
\draw[dashed,line width=0.3mm] (0,1)--(0,-3);
\draw[line width=0.4mm] (0,1)--(-4,1)--(-4,-3)--(0,-3);
\vertline{-4}{-1}{red}
\vertline{-4}{-2}{red}
\vertline{-4}{-3}{red}
\vertline{-3}{-2}{red}
\vertline{-3}{-3}{red}
\vertline{-2}{-3}{red}

\draw (4.4,0.5) node {$2$};
\draw (4.4,-0.5) node {$1$};
\draw (4.4,-1.5) node {$4$};
\draw (4.4,-2.5) node {$3$};
\draw (0.5,-3.4) node  {$1$};
\draw (1.5,-3.4) node  {$2$};
\draw (2.5,-3.4) node  {3};
\draw (3.5,-3.4) node  {4};
\draw (-5.5,-1) node {$\lambda = (1)$};
\draw (-5.5,-2) node {$x_3-a_3$};
\end{scope}
%
%
\begin{scope}[shift={(6.5,-24)}]
\rightelbow{1}{0}{blue}
\horline{2}{0}{blue}
\horline{3}{0}{blue}

\rightelbow{0}{-1}{blue}
\cross{1}{-1}{blue}{blue}
\horline{2}{-1}{blue}{blue}
\horline{3}{-1}{blue}

\vertline{0}{-2}{blue}
\vertline{1}{-2}{blue}
\rightelbow{3}{-2}{blue}

\vertline{0}{-3}{blue}
\vertline{1}{-3}{blue}
\rightelbow{2}{-3}{blue}
\cross{3}{-3}{blue}{blue}
\vertline{-4}{0}{red}

\vertline{-3}{0}{red}
\vertline{-3}{-1}{red}

\vertline{-2}{0}{red}
\vertline{-2}{-1}{red}
\vertline{-2}{-2}{red}

\vertline{-1}{0}{red}
\vertline{-1}{-1}{red}
\vertline{-1}{-2}{red}
\vertline{-1}{-3}{red}

\draw[line width=0.4mm] (0,-3)--(4,-3)--(4,1)--(0,1);
\draw[dashed,line width=0.3mm] (0,1)--(0,-3);
\draw[line width=0.4mm] (0,1)--(-4,1)--(-4,-3)--(0,-3);
\vertline{-4}{-1}{red}
\vertline{-4}{-2}{red}
\vertline{-4}{-3}{red}
\vertline{-3}{-2}{red}
\vertline{-3}{-3}{red}
\vertline{-2}{-3}{red}

\draw (4.4,0.5) node {$2$};
\draw (4.4,-0.5) node {$1$};
\draw (4.4,-1.5) node {$4$};
\draw (4.4,-2.5) node {$3$};
\draw (0.5,-3.4) node  {$1$};
\draw (1.5,-3.4) node  {$2$};
\draw (2.5,-3.4) node  {3};
\draw (3.5,-3.4) node  {4};
\draw (-5.5,0) node {$\lambda = \emptyset$};
\draw (-5.5,-1) node {$(x_1-a_1)$};
\draw (-5.5,-2) node {$(x_3-a_3)$};
\end{scope}
\end{tikzpicture}
\end{center}
\caption{Rectangular $S_n$-bumpless pipedreams for $w=2143$}
\label{F:pipedreams}
\end{figure}
\end{example} 

\begin{thm}\label{thm:decomp}
Let $w \in S_n$.  Then we have $\bS_w(x;a) = \sum_P \wt(P) s_{\lambda(P)}(x||a)$ where the sum is over all $w$-rectangular bumpless pipedreams.
\end{thm}
Theorem \ref{thm:decomp} is proved in \textsection \ref{ssec:decompproof}.

\begin{cor}\label{C:Stanleydecomp}
Let $w \in S_n$.  Then $F_w(x||a) = \sum_P \eta_a(\wt(P)) s_{\lambda(P)}(x||a)$ where the sum is over all $w$-rectangular bumpless pipedreams.
\end{cor}

\subsection{Square $S_n$-bumpless pipedreams}
Let $w \in S_n$.  A \defn{$w$-square bumpless pipedream} is a bumpless pipedream in the $n \times n$ square region 
$$S_n := \{(i,j) \mid i \in [1,n] \text{ and } j \in [1,n]\}.
$$  

The pipes are labeled $1,\ldots,n$, entering the south boundary from left to right.  The pipes exit the east boundary: pipe $i$ exits in row $i$.  Two pipes intersect at most once.   As before, the weight of a square $S_n$-bumpless pipedream $P$ is given by $\wt(P) = \prod (x_i - a_j)$, with the product  over all empty tiles $(i,j)$.  In Example \ref{ex:2143}, if we erase the left half and all red pipes, we obtain a square $2143$-bumpless pipedream.

\begin{thm}\label{T:square bumpless}
For $w \in S_n$ we have $\S_w(\xp;\ap) = \sum_P \wt(P)$ where the sum is over all $w$-square bumpless pipedreams.
\end{thm}
\begin{proof}
By Theorem \ref{thm:doublecoprod} and Lemma \ref{lem:doubleid}, when $\bS_w(\xp;\ap)$ is expanded in terms of $\{s_\la(x||a) \mid \la \in \Par\}$, the coefficient of $s_\emptyset(x||a)$ is equal to $\S_w(\xp;\ap)$.  By Theorem \ref{thm:decomp}, we thus have
$$
\S_w(\xp;\ap) = \sum_P \wt(P)
$$
summed over $w$-rectangular bumpless pipedreams $P$ satisfying $\lambda(P) = \emptyset$.  The condition $\lambda(P) = \emptyset$ is equivalent to all nonpositively labeled pipes in $P$ being completely vertical.  In particular, the nonpositively labeled pipes stay within the left $n \times n$ square of $P$.  Such pipedreams are in weight-preserving bijection with $w$-square bumpless pipedreams.
\end{proof}

\begin{proof}[Proof of Theorem \ref{T:bumpless}]
The special role of the row and column indexed 0 is arbitrary.  In Theorem \ref{T:square bumpless}, we could obtain a formula for the double Schubert polynomial $\S^{[p,n]}_w(x;a)$ (with variables $x_p,x_{p+1},\ldots,x_n$ and $a_p,a_{p+1},\ldots,a_n$) if we worked with square $w$-bumpless pipedreams in rows and columns indexed by $p,p+1,\ldots,n$.  We note that such bumpless pipedreams are back stable: there is a natural weight-preserving injection sending such a pipedream for $\S^{[p,n]}_w(x;a)$ to a pipedream for $\S^{[p-1,n]}_w(x;a)$.  The union of all such square $w$-pipedreams are exactly the $w$-bumpless pipedreams of Theorem \ref{T:bumpless}.  Taking $ p \to -\infty$, Theorem \ref{T:bumpless} thus follows from the definition of back stable double Schubert polynomial.
\end{proof}

\begin{proof}[Proof of Theorem \ref{T:lambdabumpless}]
We apply Theorem \ref{T:bumpless} to $w = w_\lambda$.  We have $w_\lambda(1) < w_\lambda(2) < \cdots$ and $w_\lambda(0) > w_\lambda(-1) > \cdots$.  It follows that in a $w_\lambda$-bumpless pipedream:
\begin{enumerate}
\item there are no crossings in rows indexed by nonpositive integers;
\item there are no empty tiles in rows indexed by positive integers.
\end{enumerate}
Thus the lower half of a $w_\lambda$-bumpless pipedream $P$ is completely determined by $\la$, and the upper half is a $\lambda$-halfplane pipedream.
\end{proof}

\subsection{EG pipedreams}
Let $w \in S_n$.  Let $P$ be a $w$-square bumpless pipedream.  We call $P$ a \defn{$w$-EG pipedream} if all the empty tiles are in the northeast corner, where they form a partition shape $\lambda = \lambda(P)$, called the \defn{shape} of $P$. See Figure \ref{F:EG pipedream}.

\begin{figure} 

\begin{center}
\begin{tikzpicture}[scale=0.6,line width=0.8mm]
\rightelbow{1}{0}{blue}
\horline{2}{0}{blue}
\rightelbow{0}{-1}{blue}
\cross{1}{-1}{blue}{blue}
\horline{2}{-1}{blue}
\vertline{0}{-2}{blue}
\vertline{1}{-2}{blue}
\rightelbow{2}{-2}{blue}
\draw[line width=0.6mm] (0,-2)--(3,-2)--(3,1)--(0,1)--(0,-2);
\draw (3.4,0.5) node {$2$};
\draw (3.4,-0.5) node {$1$};
\draw (3.4,-1.5) node {$3$};
\draw (0.5,-2.4) node  {$1$};
\draw (1.5,-2.4) node  {$2$};
\draw (2.5,-2.4) node  {3};
\end{tikzpicture}
\end{center}
\caption{A 213-EG pipedream with partition $(1)$}
\label{F:EG pipedream}
\end{figure}

\begin{thm}\label{thm:EGStanley}
The Edelman-Greene coefficient $j_\lambda^w = j_\lambda^w(0)$ of \eqref{E:stanley coefficients}  is equal to the number of $w$-EG pipedreams $P$ satisfying $\lambda(P) = \lambda$.
\end{thm}
\begin{proof}
Specializing $a_i = 0$ for all $i$ in Corollary \ref{C:Stanleydecomp}, we obtain $F_w = \sum_P s_{\lambda(P)}$ where the sum is over all $w$-rectangular bumpless pipedreams $P$ {\it with no empty tiles}.  In particular, the positively labeled pipes in the right $n \times n$ square of $P$ forms a $w$-EG pipedream.  The nonpositively labeled pipes in $P$ have to fill up all the remaining tiles, and since they cannot intersect, there is a unique way to do so.  Thus there is a bijection between $w$-rectangular bumpless pipedreams with no empty tiles and $w$-EG pipedreams.  Finally, one verifies from the definitions that $\lambda(P)$ is defined consistently for the two kinds of pipedreams.
\end{proof}

An empty tile $T$ in a bumpless pipedream $D$ is called a \defn{floating tile} if there exists a pipe that is northwest of $T$.  A bumpless pipedream $D$ is called \defn{near EG} if it has a single floating tile.

\subsection{Column moves}
We define \defn{column moves} that modify a bumpless pipedream in two adjacent columns; see Figure~\ref{fig:columnmove}.  Only one of the pipes (the \defn{active pipe}) is drawn in these pictures.  For the move to be allowed, the southeastmost tile must be an empty tile (before the move), and it must be the only empty tile.  Thus the move takes the empty tile from the southeastmost position to the northwestmost position.  There are usually other pipes (indicated in black) in the move, and the ``kinks are shifted left'' if necessary, see Figure~\ref{fig:columnmove}.
 
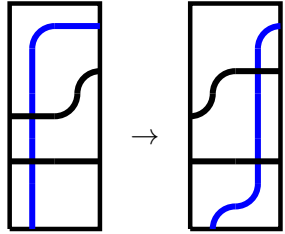
 \begin{figure}
\begin{center}
\begin{tikzpicture}[scale=0.6,line width=0.8mm]
\draw[line width=0.6mm] (0,0)--(2,0)--(2,4)--(0,4)--(0,0);
\vertline{0}{0}{blue}
\vertline{0}{1}{blue}
\vertline{0}{2}{blue}
\rightelbow{0}{3}{blue}
\horline{1}{3}{blue}

\draw (3,2) node {$\rightarrow$};

\begin{scope}[shift={(4,0)}]
\draw[line width=0.6mm] (0,0)--(2,0)--(2,4)--(0,4)--(0,0);
\rightelbow{0}{0}{blue}
\leftelbow{1}{0}{blue}
\vertline{1}{1}{blue}
\vertline{1}{2}{blue}
\rightelbow{1}{3}{blue}
\end{scope}

\begin{scope}[shift={(9,0)}]
\draw[line width=0.6mm] (0,0)--(2,0)--(2,4)--(0,4)--(0,0);
\leftelbow{0}{0}{blue}
\vertline{0}{1}{blue}
\vertline{0}{2}{blue}
\rightelbow{0}{3}{blue}
\horline{1}{3}{blue}

\draw (3,2) node {$\rightarrow$};

\begin{scope}[shift={(4,0)}]
\draw[line width=0.6mm] (0,0)--(2,0)--(2,4)--(0,4)--(0,0);
\horline{0}{0}{blue}
\leftelbow{1}{0}{blue}
\vertline{1}{1}{blue}
\vertline{1}{2}{blue}
\rightelbow{1}{3}{blue}
\end{scope}
\end{scope}

\begin{scope}[shift={(0,-6)}]
\draw[line width=0.6mm] (0,0)--(2,0)--(2,4)--(0,4)--(0,0);
\vertline{0}{0}{blue}
\vertline{0}{1}{blue}
\vertline{0}{2}{blue}
\rightelbow{0}{3}{blue}
\leftelbow{1}{3}{blue}

\draw (3,2) node {$\rightarrow$};

\begin{scope}[shift={(4,0)}]
\draw[line width=0.6mm] (0,0)--(2,0)--(2,4)--(0,4)--(0,0);
\rightelbow{0}{0}{blue}
\leftelbow{1}{0}{blue}
\vertline{1}{1}{blue}
\vertline{1}{2}{blue}
\vertline{1}{3}{blue}
\end{scope}

\begin{scope}[shift={(9,0)}]
\draw[line width=0.6mm] (0,0)--(2,0)--(2,4)--(0,4)--(0,0);
\leftelbow{0}{0}{blue}
\vertline{0}{1}{blue}
\vertline{0}{2}{blue}
\rightelbow{0}{3}{blue}
\leftelbow{1}{3}{blue}

\draw (3,2) node {$\rightarrow$};

\begin{scope}[shift={(4,0)}]
\draw[line width=0.6mm] (0,0)--(2,0)--(2,4)--(0,4)--(0,0);
\horline{0}{0}{blue}
\leftelbow{1}{0}{blue}
\vertline{1}{1}{blue}
\vertline{1}{2}{blue}
\vertline{1}{3}{blue}
\end{scope}
\end{scope}
\end{scope}
\begin{scope}[shift={(19,-3.5)}]
\draw[line width=0.6mm] (0,0)--(2,0)--(2,5)--(0,5)--(0,0);
\vertline{0}{0}{blue}
\cross{0}{1}{blue}{\inactive}
\horline{1}{1}{\inactive}
\cross{0}{2}{blue}{\inactive}
\leftelbow{1}{2}{\inactive}
\vertline{0}{3}{blue}
\rightelbow{1}{3}{\inactive}
\rightelbow{0}{4}{blue}
\horline{1}{4}{blue}

\draw (3,2) node {$\rightarrow$};

\begin{scope}[shift={(4,0)}]
\draw[line width=0.6mm] (0,0)--(2,0)--(2,5)--(0,5)--(0,0);
\rightelbow{0}{0}{blue}
\leftelbow{1}{0}{blue}
\horline{0}{1}{\inactive}
\cross{1}{1}{blue}{\inactive}
\leftelbow{0}{2}{\inactive}
\vertline{1}{2}{blue}
\rightelbow{0}{3}{\inactive}
\cross{1}{3}{blue}{\inactive}
\rightelbow{1}{4}{blue}
\end{scope}
\end{scope}
\end{tikzpicture}
 \end{center}
 \caption{On the left: the different types of column moves.  On the right: kinks are shifted to the left.}
 \label{fig:columnmove}
 \end{figure}

%
%
A column move is a droop if no kinks are present, and in addition, the pipe exits south in the left column and exits east in the right column.  We write $D \to D'$ if two bumpless pipedreams are related by a column move.  We say that $D'$ is obtained from $D$ by a \defn{downwards} column move.  

\begin{lemma}
Let $D$ be a bumpless pipedream that is not an EG pipedream.  Then $D$ admits a downwards column move $D \to D'$.
\end{lemma}
\begin{proof}
Let $E$ be any northwestmost floating tile in $D$.  Then the tile $W$ immediately to west of $E$ is nonempty, and must either be a NW elbow or a vertical line.  Call this pipe $p$.  Then $p$ travels up from $W$ a number of tiles and turns towards the east at a tile $T$.  We may perform a column move in the rectangle with corners $T$ and $E$.
\end{proof}

\begin{lemma}\label{lem:rectify}
Let $D$ be a a near EG pipedream.  Then there is a unique sequence of moves $D \to D' \to D'' \to \cdots \to D^*$ where $D^*$ is a EG pipedream.
\end{lemma}
Write $r(D) = D^*$ for the EG pipedream of Lemma \ref{lem:rectify}.

\begin{remark}
We can define an equivalence relation on bumpless pipedreams using column moves.  We caution the reader that multiple EG pipedreams can belong to a single such equivalence class.
\end{remark}

\subsection{Insertion}\label{SS:insertion}
Let $D$ be an $w$-EG pipedream and $i \in [1,n-1]$ be such that $s_iw > w$, or equivalently, the pipes labeled $i$ and $i+1$ do not cross in $D$.  Let $D'$ be the bumpless diagram obtained from $D$ by swapping the pipes $i$ and $i+1$ in columns $i$ and $i+1$.  Namely, if in $D$ the first turn of pipe $i$ (resp. $i+1$) is in row $a_i$ (resp. $a_{i+1} > a_i$), then in $D'$ the first turn of pipe $i$ (resp. $i+1$) is in row $a_{i+1}$ (resp. $a_i$).  Other pipes that cross pipe $i$ in column $i$ have their ``kinks shifted left'' in $D'$.  See Figure~\ref{fig:swap}.

\begin{figure}
\begin{center}
\begin{tikzpicture}[scale=0.6,line width=0.8mm]
\draw[line width=0.6mm] (0,0)--(2,0)--(2,5)--(0,5)--(0,0);
\vertline{0}{0}{blue}
\vertline{0}{1}{blue}
\cross{0}{2}{blue}{\inactive}
\vertline{0}{3}{blue}
\rightelbow{0}{4}{blue}
\horline{1}{4}{blue}
\vertline{1}{0}{blue}
\rightelbow{1}{1}{blue}
\leftelbow{1}{2}{\inactive}
\rightelbow{1}{3}{\inactive}
\draw (0.5,-.4) node  {$i$};
\draw (1.5,-.4) node  {$i'$};
\draw (-1,2.5) node {$D = $};
\draw (4,2.5) node {$D' = $};

\begin{scope}[shift={(5,0)}]
\draw[line width=0.6mm] (0,0)--(2,0)--(2,5)--(0,5)--(0,0);
\vertline{0}{0}{blue}
\leftelbow{0}{2}{\inactive}
\rightelbow{0}{3}{\inactive}
\cross{1}{1}{blue}{blue}
\vertline{1}{2}{blue}
\cross{1}{3}{blue}{\inactive}
\rightelbow{1}{4}{blue}
\bbox{0}{4}
\vertline{1}{0}{blue}
\rightelbow{0}{1}{blue}

\draw (0.5,-.4) node  {$i$};
\draw (1.5,-.4) node  {$i'$};
\end{scope}
\end{tikzpicture}
\end{center}
\caption{Swapping pipes $i$ and $i'= i+1$.}
\label{fig:swap}
\end{figure}
The northwestmost tile in the shown rectangle is always an empty tile in $D'$.  Thus $D'$ is either a EG pipedream or a near EG pipedream.  Note that there are two possibilities for the northeastmost tile in the shown rectangles.

We define the insertion $D \leftarrow i$ to be the EG pipedream given by
$$
D \leftarrow i := r(D').
$$
(Note that $D \leftarrow i$ is not defined if the pipes $i$ and $i+1$ cross.)
Let $D_0$ be the unique EG pipedream for the identity permutation.  Let $\i = i_1i_2 \cdots i_\ell$ be a reduced word.  Then define $(P,Q) = (P(\i),Q(\i))$ by
\begin{align*}
P(\i) &= (\cdots ((D_0 \leftarrow i_\ell) \leftarrow i_{\ell-1}) \cdots ) \leftarrow i_1 \\
Q(\i) &= \{\lambda(D_0) \subset \lambda(D_0 \leftarrow i_\ell) \subset \cdots \subset \lambda(P(\i))\}.
\end{align*}
Note that $Q(\i)$ is a saturated chain of partitions, and is thus equivalent to a standard Young tableau of shape $\lambda(P(\i))$.

\begin{example}
Let $n = 4$ and $\i = 1231$.  We compute $P(\i), Q(\i)$ in Figure~\ref{fig:EGex}.
\begin{figure}
\begin{center}
\begin{tikzpicture}[scale=0.6,line width=0.8mm]
\rightelbow{0}{0}{blue}
\horline{1}{0}{blue}
\horline{2}{0}{blue}
\horline{3}{0}{blue}

\vertline{0}{-1}{blue}
\rightelbow{1}{-1}{blue}
\horline{2}{-1}{blue}
\horline{3}{-1}{blue}

\vertline{0}{-2}{blue}
\vertline{1}{-2}{blue}
\rightelbow{2}{-2}{blue}
\horline{3}{-2}{blue}

\vertline{0}{-3}{blue}
\vertline{1}{-3}{blue}
\vertline{2}{-3}{blue}
\rightelbow{3}{-3}{blue}
\draw[line width=0.4mm] (0,-3)--(4,-3)--(4,1)--(0,1)--(0,-3);

\draw (4.4,0.5) node {$1$};
\draw (4.4,-0.5) node {$2$};
\draw (4.4,-1.5) node {$3$};
\draw (4.4,-2.5) node {$4$};
\draw (0.5,-3.4) node  {$1$};
\draw (1.5,-3.4) node  {$2$};
\draw (2.5,-3.4) node  {3};
\draw (3.5,-3.4) node  {4};
\draw (5.6,-1) node {$\stackrel{1}{\longrightarrow}$};
\draw (2,-4) node {$D_0$};
\begin{scope}[shift={(7,0)}]
\rightelbow{1}{0}{blue}
\horline{2}{0}{blue}
\horline{3}{0}{blue}

\rightelbow{0}{-1}{blue}
\cross{1}{-1}{blue}{blue}
\horline{2}{-1}{blue}
\horline{3}{-1}{blue}

\vertline{0}{-2}{blue}
\vertline{1}{-2}{blue}
\rightelbow{2}{-2}{blue}
\horline{3}{-2}{blue}

\vertline{0}{-3}{blue}
\vertline{1}{-3}{blue}
\vertline{2}{-3}{blue}
\rightelbow{3}{-3}{blue}
\draw[line width=0.4mm] (0,-3)--(4,-3)--(4,1)--(0,1)--(0,-3);

\draw (4.4,0.5) node {$2$};
\draw (4.4,-0.5) node {$1$};
\draw (4.4,-1.5) node {$3$};
\draw (4.4,-2.5) node {$4$};
\draw (0.5,-3.4) node  {$1$};
\draw (1.5,-3.4) node  {$2$};
\draw (2.5,-3.4) node  {3};
\draw (3.5,-3.4) node  {4};
\draw (5.6,-1) node {$\stackrel{3}{\longrightarrow}$};
\draw (2,-4) node {$P(1)$};
\end{scope}
\begin{scope}[shift={(14,0)}]
\rightelbow{1}{0}{blue}
\horline{2}{0}{blue}
\horline{3}{0}{blue}

\rightelbow{0}{-1}{blue}
\cross{1}{-1}{blue}{blue}
\horline{2}{-1}{blue}
\horline{3}{-1}{blue}

\vertline{0}{-2}{blue}
\vertline{1}{-2}{blue}
\rightelbow{3}{-2}{blue}

\vertline{0}{-3}{blue}
\vertline{1}{-3}{blue}
\rightelbow{2}{-3}{blue}
\cross{3}{-3}{blue}{blue}
\draw[line width=0.4mm] (0,-3)--(4,-3)--(4,1)--(0,1)--(0,-3);

\draw (4.4,0.5) node {$2$};
\draw (4.4,-0.5) node {$1$};
\draw (4.4,-1.5) node {$4$};
\draw (4.4,-2.5) node {$3$};
\draw (0.5,-3.4) node  {$1$};
\draw (1.5,-3.4) node  {$2$};
\draw (2.5,-3.4) node  {3};
\draw (3.5,-3.4) node  {4};
\draw (5.6,-1) node {$\stackrel{r}{\longrightarrow}$};
\end{scope}
\begin{scope}[shift={(21,0)}]
\rightelbow{2}{0}{blue}
\horline{3}{0}{blue}

\rightelbow{0}{-1}{blue}
\horline{1}{-1}{blue}
\cross{2}{-1}{blue}{blue}
\horline{3}{-1}{blue}

\vertline{0}{-2}{blue}
\rightelbow{1}{-2}{blue}
\leftelbow{2}{-2}{blue}
\rightelbow{3}{-2}{blue}

\vertline{0}{-3}{blue}
\vertline{1}{-3}{blue}
\rightelbow{2}{-3}{blue}
\cross{3}{-3}{blue}{blue}
\draw[line width=0.4mm] (0,-3)--(4,-3)--(4,1)--(0,1)--(0,-3);

\draw (4.4,0.5) node {$2$};
\draw (4.4,-0.5) node {$1$};
\draw (4.4,-1.5) node {$4$};
\draw (4.4,-2.5) node {$3$};
\draw (0.5,-3.4) node  {$1$};
\draw (1.5,-3.4) node  {$2$};
\draw (2.5,-3.4) node  {3};
\draw (3.5,-3.4) node  {4};
\draw (5.6,-1) node {$\stackrel{2}{\longrightarrow}$};
\draw (2,-4) node {$P(31)$};
\end{scope}

\begin{scope}[shift={(0,-7)}]
\rightelbow{2}{0}{blue}
\horline{3}{0}{blue}

\rightelbow{0}{-1}{blue}
\horline{1}{-1}{blue}
\cross{2}{-1}{blue}{blue}
\horline{3}{-1}{blue}

\vertline{0}{-2}{blue}
\vertline{2}{-2}{blue}
\rightelbow{3}{-2}{blue}

\vertline{0}{-3}{blue}
\rightelbow{1}{-3}{blue}
\cross{2}{-3}{blue}{blue}
\cross{3}{-3}{blue}{blue}
\draw[line width=0.4mm] (0,-3)--(4,-3)--(4,1)--(0,1)--(0,-3);

\draw (4.4,0.5) node {$3$};
\draw (4.4,-0.5) node {$1$};
\draw (4.4,-1.5) node {$4$};
\draw (4.4,-2.5) node {$2$};
\draw (0.5,-3.4) node  {$1$};
\draw (1.5,-3.4) node  {$2$};
\draw (2.5,-3.4) node  {3};
\draw (3.5,-3.4) node  {4};
\draw (5.6,-1) node {$\stackrel{r}{\longrightarrow}$};
\begin{scope}[shift={(7,0)}]
\rightelbow{2}{0}{blue}
\horline{3}{0}{blue}

\rightelbow{1}{-1}{blue}
\cross{2}{-1}{blue}{blue}
\horline{3}{-1}{blue}

\rightelbow{0}{-2}{blue}
\leftelbow{1}{-2}{blue}
\vertline{2}{-2}{blue}
\rightelbow{3}{-2}{blue}

\vertline{0}{-3}{blue}
\rightelbow{1}{-3}{blue}
\cross{2}{-3}{blue}{blue}
\cross{3}{-3}{blue}{blue}
\draw[line width=0.4mm] (0,-3)--(4,-3)--(4,1)--(0,1)--(0,-3);

\draw (4.4,0.5) node {$3$};
\draw (4.4,-0.5) node {$1$};
\draw (4.4,-1.5) node {$4$};
\draw (4.4,-2.5) node {$2$};
\draw (0.5,-3.4) node  {$1$};
\draw (1.5,-3.4) node  {$2$};
\draw (2.5,-3.4) node  {3};
\draw (3.5,-3.4) node  {4};
\draw (5.6,-1) node {$\stackrel{1}{\longrightarrow}$};
\draw (2,-4) node {$P(231)$};
\end{scope}
\begin{scope}[shift={(14,0)}]
\rightelbow{2}{0}{blue}
\horline{3}{0}{blue}

\rightelbow{1}{-1}{blue}
\cross{2}{-1}{blue}{blue}
\horline{3}{-1}{blue}

\vertline{1}{-2}{blue}
\vertline{2}{-2}{blue}
\rightelbow{3}{-2}{blue}

\rightelbow{0}{-3}{blue}
\cross{1}{-3}{blue}{blue}
\cross{2}{-3}{blue}{blue}
\cross{3}{-3}{blue}{blue}
\draw[line width=0.4mm] (0,-3)--(4,-3)--(4,1)--(0,1)--(0,-3);

\draw (4.4,0.5) node {$3$};
\draw (4.4,-0.5) node {$2$};
\draw (4.4,-1.5) node {$4$};
\draw (4.4,-2.5) node {$1$};
\draw (0.5,-3.4) node  {$1$};
\draw (1.5,-3.4) node  {$2$};
\draw (2.5,-3.4) node  {3};
\draw (3.5,-3.4) node  {4};
\draw (2,-4) node {$P(1231)$};
\end{scope}
\end{scope}

\draw (23,-9) node {$\tableau[sY]{1 &2 \\3 \\4}$};
\draw (23,-11) node {$Q(1231)$};

\end{tikzpicture}

\end{center}
\caption{The computation of the EG pipedream $P(1231)$.  For each insertion step, both $D'$ and $r(D')$ are shown (if they are different).  }
\label{fig:EGex}
\end{figure}
\end{example}

We recall the \defn{Coxeter-Knuth equivalence relation} on $\Red(w)$.  It is generated by the elementary relations
\begin{align*}
\cdots i k j \cdots &\sim \cdots k i j \cdots & \mbox{for $ i < j < k$}\\
\cdots i k j \cdots &\sim \cdots j k i \cdots & \mbox{for $ i < j < k$}\\
\cdots i (i+1) i \cdots &\sim \cdots (i+1) i (i+1) \cdots.
\end{align*}
Edelman-Greene insertion provides a bijection $C \mapsto T(C)$ between Coxeter-Knuth equivalence classes $C \subset \Red(w)$ and reduced word tableaux $T$ for $w$ (see \cite{EG} and the paragraph containing (\ref{E:stanley coefficients})).  Bumpless pipedreams encode a new version of Edelman-Greene insertion.

\begin{thm} \label{thm:EGpipedream}
The map $\i \mapsto (P(\i),Q(\i))$ induces a bijection between reduced words for $S_n$ and pairs consisting of an EG pipedream and a standard Young tableau of the same shape.  For a fixed EG pipedream $D$, the set $C_D:= \{\i \mid P(\i) = D\}$ is a single Coxeter-Knuth equivalence class.  The shape of the reduced word tableau $T(C_D)$ is $\la(D)$.
\end{thm}

\begin{problem}
Find a direct shape-preserving bijection between EG pipedreams for $w$ and reduced word tableaux for $w$. \footnote{Since our preprint was posted, solutions to this problem have appeared in \cite{FGS,Wei}.}
\end{problem}

\begin{remark}
There is a transpose analogue of column moves called {\it row moves}.  We can also define insertion into EG pipedreams using row moves.  Theorem \ref{thm:EGpipedream} holds with (usual) Edelman-Greene insertion replaced by Edelman-Greene column insertion.
\end{remark}

The \defn{insertion path} of the insertion $D \leftarrow i$ is the collection of positions through which the empty tile travels in the calculation of $r(D')$.  An insertion path consists of a number of boxes, one in each of an interval of columns.  Two insertion paths are compared by comparing respective boxes in the same column.

The following key result is immediate from the definition of column moves.
\begin{lemma}\label{lem:reverse}
The pair $(D,i)$ can be recovered uniquely from the pair $$(D \leftarrow i, \mbox{final box in the insertion path}).$$
\end{lemma}

\begin{lemma}\label{lem:insertionpath}
Suppose $i < j$.  
\begin{enumerate} 
\item
Then the insertion path of $D\leftarrow i$ is strictly below the insertion path of $(D\leftarrow i) \leftarrow j$.
\item
Then the insertion path of $D\leftarrow j$ is strictly above the insertion path of $(D\leftarrow j) \leftarrow i$.
\end{enumerate}
\end{lemma}
\begin{proof}
We show claim (1); claim (2) is similar.  Let the insertion path of $D_1 = D \leftarrow i$ be the boxes $b_i, b_{i-1},\ldots,b_s$, where $b_k$ is in column $k$.  Let the insertion path of $(D\leftarrow i) \leftarrow j$ be the boxes $c_j, c_{j-1},\ldots, c_t$ where $c_k$ is in column $k$.  Consider the calculation of $c_i$: the lowest elbow in column $i-1$ of $D_1$ is at the same height as $b_i$.  Thus $c_{i-1}$ must at the same height or above $b_i$.  It follows that $c_i$ is above $b_i$, and indeed it must be at least as high as $b_{i-1}$ because there are no elbows in column $i$ above $b_i$ and below the row of $b_{i-1}$.  The claim (1) follows by repeating this argument.
\end{proof}

Recall that the descent set $\Des(T)$ of a standard Young tableau $T$ is the set of letters $j$ such that $j+1$ is in a lower row than $j$ in $T$.  The descent set $\Des(\i)$ of a word $\i  = i_1 \cdots i_\ell$ is the indices $j$ such that $i_j > i_{j+1}$.
\begin{cor}\label{cor:Des}
For a reduced word $\i$, we have $\Des(\i) = \Des(Q(\i))$.
\end{cor}

\begin{lemma}\label{lem:Knuth}
Let $D$ be a EG pipedream and suppose $i < j < k$.  Then (when the EG pipedreams are defined),
\begin{enumerate} 
\item
$((D \leftarrow j) \leftarrow i) \leftarrow k = ((D \leftarrow j) \leftarrow k \leftarrow i)$;
\item
$((D \leftarrow i) \leftarrow k) \leftarrow j= ((D \leftarrow k) \leftarrow i \leftarrow j)$.
\end{enumerate}
\end{lemma}
\begin{proof}
We prove claim (1); claim (2) is similar.  By Lemma \ref{lem:insertionpath}, the insertion path for $D \leftarrow j$ is above that of $(D \leftarrow j) \leftarrow i$.  In particular, the two EG pipedreams $(D \leftarrow j) \leftarrow i$ and $D\leftarrow i$ differ only in tiles that are below the insertion path of $j$.  On the other hand, the insertion path for $(D \leftarrow j) \leftarrow k$ is above that of $D \leftarrow j$, and thus does not see the part the pipedream below the insertion path of $j$.  The desired equality follows.
\end{proof}

\begin{lemma}\label{lem:Coxeter}
Let $D$ be a EG pipedream and suppose $i,i+1 \in [1,n-1]$.
When the EG pipedreams are defined, we have $((D \leftarrow i) \leftarrow i+1) \leftarrow i =((D \leftarrow i+1) \leftarrow i) \leftarrow i+1$.
\end{lemma}
\begin{proof}
 For the insertions to be defined, the pipes $i$, $i+1$, and $i+2$ in $D$ do not intersect.  Let $h_i,h_{i+1},h_{i+2}$ be the heights of the boxes containing the first right elbow for the pipes $i$, $i+1$, and $i+2$ respectively.  Then $h_i$ is strictly above $h_{i+1}$, which is strictly above $h_{i+2}$.
 
 Let us first consider $D_1 = ((D \leftarrow i) \leftarrow i+1) \leftarrow i$.  To calculate $(D \leftarrow i)$ we will first create an empty tile in the box $(i, h_i)$ in column $i$. Instead of moving this empty tile to the northwest immediately, let us keep it where it is, and consider the insertion of $i+1$.  This creates an empty tile in the box $(i+1,h_i)$.  Finally, the second insertion of $i$ creates an empty tile in $(i,h_{i+1})$.  Call the resulting bumpless diagram $D'_1$.  Checking the definitions, we see that $D_1$ is obtained from $D'_1$ by performing column moves on the three empty tiles, as long as we move the empty tiles in order.
 
Now consider $D_2 = ((D \leftarrow i+1) \leftarrow i) \leftarrow i+1$.  To calculate $(D \leftarrow i+1)$ we will first create an empty tile in the box $(i+1, h_{i+1})$ in column $i$.  Applying a single downward move to this empty tile, we see that it will end up in box $(i,h_i)$.  At this point the first right elbow in column $i$ will be at height $h_{i+1}$.  Now we consider the insertion of $i$, which creates an empty tile at position $(i,h_{i+1})$.  Finally, the second insertion of $i+1$ creates an empty tile in $(i+1,h_i)$.  The resulting bumpless diagram is identical to $D'_1$.  To obtain $D_2$, we perform column moves on the three empty tiles in the correct order.  

The difference between the calculation of $D_1$ and $D_2$ is that the order of applying column moves to the empty tiles in positions $(i,h_{i+1})$ and $(i+1,h_i)$ are swapped.  We claim that the resulting EG-diagrams $D_1$ and $D_2$ are nevertheless identical.  This is because the path of the tile at $(i,h_{i+1})$ (resp. $(i+1,h_i)$) stays below (resp. above) that of the tile at $(i,h_i)$.  Thus the corresponding column moves commute, as in the proof of Lemma \ref{lem:Knuth}.
\end{proof}
 
\begin{proof}[Proof of Theorem \ref{thm:EGpipedream}]
Bijectivity is straightforward from the constructions: by Lemma \ref{lem:reverse}, the map $\i \mapsto (P(\i),Q(\i))$ is injective, and applying this reverse map to pairs $(P(\i),Q(\i))$ shows that the map is surjective.  

By Lemma \ref{lem:Coxeter} and \ref{lem:Knuth}, Coxeter-Knuth equivalent reduced words have the same insertion EG pipedream.  Thus the set $\{\i \mid P(\i) = D\}$ is a union of Coxeter-Knuth equivalence classes.  That it is a single Coxeter-Knuth equivalence class can be deduced from Theorem \ref{thm:EGStanley}.  Alternatively, the same claim can be deduced from the reversed versions of Lemmas \ref{lem:insertionpath}, \ref{lem:Coxeter}, and \ref{lem:Knuth}.

Let $\SYT(\la)$ denote the set of standard Young tableaux of shape $\la$.  Then the collection $\{\Des(S) \mid S \in \SYT(\la)\}$ of descent sets uniquely determines $\la$.  (For example, this collection encodes the expansion of the Schur function $s_\la$ in terms of fundamental quasisymmetric functions, and the assignment $\la \mapsto s_\la$ is injective.)  Let $\sh(T)$ denote the shape of a Young tableau $T$.  Then for a Coxeter-Knuth equivalence class $C$, the equality of multisets $\{\Des(\i) \mid \i \in C\} = \{\Des(S) \mid S \in \SYT(\sh(T(C)))\}$ is known to hold for Edelman-Greene insertion.  The last claim then follows from Corollary \ref{cor:Des}.
\end{proof}

\section{Infinite flag variety}
\label{S:geometry}

Let $A_\Z$ denote the Dynkin diagram with Dynkin node set $\Z$ and simple bonds $(i,i+1)$ for all $i\in \Z$.  In this section, we construct the infinite flag variety $\Fl$ explicitly. It is a  ``type $A_{\Z}$ Kac-Moody flag ind-variety''.\footnote{Strictly speaking, Kac-Moody Dynkin node sets are finite by definition. Kashiwara's thick flag scheme construction \cite{Ka} allows infinite Dynkin node sets.} It affords the action of a torus $\TZ$. We define the Schubert basis for the equivariant cohomology $H_{\TZ}^*(\Fl)$ of $\Fl$ as well as an algebraic construction for it called the GKM (Goresky-Kottwitz-Macpherson) ring $\Psi$. Similar constructions are made for the  infinite Grassmannian $\Gr$. Then we show that the GKM ring of $\Fl$ is isomorphic to the polynomial ring $\bR(x,a)$ with Schubert basis corresponding to backstable Schubert polynomials.  For more on infinite Grassmannians, we refer the reader to \cite{PS}.

\subsection{Infinite Grassmannian}
Let $F = \C((t))$ be the space of formal Laurent series.   For $a \in \Z$, let $E_a = \{\sum_{i=a}^\infty c_i t^i \mid c_i \in \C\} \subset F$. For $N\in\Z_{>0}$ say that a $\C$-subspace $\La\subset F$ is \defn{$N$-admissible} if $E_N \subset \Lambda \subset E_{-N}$ and that $\La$ is admissible if it is admissible for some $N\in\Z_{>0}$. The \defn{virtual dimension} $\vdim(\Lambda)$ of an admissible subspace $\Lambda$ is the difference 
$$
\vdim(\Lambda) := \dim(E_0/(\Lambda \cap E_0)) - \dim(\Lambda/(\Lambda \cap E_0)).
$$
The \defn{Sato Grassmannian} $\Gr^\bullet$ is the set of admissible subspaces in $F$. The Sato Grassmannian is a disjoint union $\Gr^\bullet = \bigsqcup_k \Gr^{(k)}$, where $\Gr^{(k)}$ consists of the admissible subspaces of virtual dimension $k$.  We will mostly focus on the the \defn{infinite Grassmannian} $\Gr:= \Gr^{(0)}$.

There is a bijection between $N$-admissible subspaces of virtual dimension $0$, and
the points of the finite-dimensional Grassmannian $\Gr(N,2N)\cong \Gr(N,E_{-N}/E_N)$ given by $\La\mapsto \La/E_N$.  We have $\Gr = \bigcup_N \Gr(N,2N)$, from which $\Gr$ inherits the structure of an ind-variety over $\C$.
\subsection{Infinite flag variety}
For $N\in\Z_{>0}$, an \defn{$N$-admissible flag} (of virtual dimension $0$) in $F$ is a sequence
$$
\Lambda_\bullet = \{\cdots \subset \Lambda_{-1} \subset \Lambda_0 \subset \Lambda_1 \subset \cdots\}
$$
of admissible subspaces satisfying the conditions (1) $\vdim(\Lambda_i) = i$ and (2) $\Lambda_i = E_{-i}$ for all $i$ with $|i| \ge N$. 
An admissible flag is one that is $N$-admissible for some $N\in\Z_{>0}$. 

The \defn{infinite flag variety} $\Fl$ is the set of all admissible flags. 
There is a bijection from the set of $N$-admissible flags to the points of the variety $\Fl(2N)\cong \Fl(E_{-N}/E_N)$ of complete flags in the $2N$-dimensional vector space $E_{-N}/E_N$ given by
$\La_\bullet\mapsto (\La_{-N}/E_N\subset\La_{1-N}/E_N\subset\dotsm\subset\La_N/E_N)$.
We have $\Fl = \bigcup_N \Fl(2N)$ from which $\Fl$ inherits the structure of an ind-variety over $\C$. For $i\in\Z$ denote by $\pi_i: \Fl \to \Gr^{(i)}$ the projection map sending $\Lambda_\bullet \mapsto \Lambda_i$.

For $k\in\Z$ let $\Fl^{(k)}$ denote the $k$-shifted infinite flag ind-variety, which is defined similarly to $\Fl$ except that the condition $\vdim(\Lambda_i) = i+k$ is imposed. Each $\Fl^{(k)}$ is isomorphic to $\Fl$, and we have the \defn{Sato flag variety} $\Fl^\cdot = \bigsqcup_k \Fl^{(k)}$.

\subsection{Schubert varieties}
Let $T_\Z = (\C^\times)^\Z$ be the \emph{restricted product}, whose elements have only finitely many non-identity factors. The torus $T_\Z$ is the union $\bigcup_{a\leq b} T_{[a,b]}$ of finite-dimensional subtori where $T_{[a,b]}$ consists of the elements equal to the identity in positions outside of $[a,b] \subset \Z$.  The torus $T_\Z$ acts naturally on $F$, with the $i$-th coordinate of $T_\Z$ acting on the coefficient of $t^i$.  This induces an action of $T_\Z$ on $\Fl$ and $\Gr$.  The action of $T_\Z$ on $\Fl(2N)$ (resp. $\Gr(N,2N)$) factors through the action of $T_{[-N,N-1]}$ on $\Fl(2N)$ (resp. $\Gr(N,2N)$).  

For $(i,j)\in \Z^2$ with $i\ne j$ and $a\in \C$ define the $\C$-linear transformation of $F$ given by 
$$
x_{i,j}(a)(t^k) = \begin{cases} t^i + at^j & \mbox{if $k = i$,} \\
t^k &\mbox{otherwise.}
\end{cases}
$$
Let $S_\Z\subset GL(F)$ via permutation matrices.
Let $B$ be the group of linear transformations of $F$ generated by $T_\Z$ and $x_{ij}(a)$ for $a\in \C$ for $i<j$. Let $P$ be the group generated by $B$ and $S_{\ne0}$ and $G$ the group generated by $B$ and $S_\Z$. We call the group $G$ the ``minimal Kac-Moody group of type $A_{\Z}$" in analogy with the situation for Kac-Moody groups \cite[\textsection 7.4]{Kum}. Let $E_\bullet\in\Fl$ be the standard flag (whose $i$-th subspace is the standard subspace $E_i$ for all $i\in\Z$). Then $\Fl\cong G/B$ since $G$ acts transitively on $\Fl$ and $B$ is the stabilizer of $E_\bullet$.
This isomorphism restricts to a bijection of $\TZ$-fixed points $w E_\bullet \mapsto wB/B$. The Schubert cell $BwE_\bullet \mapsto BwB/B$ is isomorphic to the affine space $\C^{\ell(w)}$ and is contained in $\Fl(2N)$ if $w \in S_{[-N,N-1]}$.  We define the \defn{Schubert variety}
$$
X_w:= \overline{B wB/B}.
$$
We have $\Gr\cong G/P$ since $G$ acts transitively on $\Gr$ and $P$ is the stabilizer of $E_0$. This restricts to the bijection of $\TZ$-fixed point sets $w E_0 \mapsto wP/P$ where $w\in S_\Z^0$. Define the Schubert variety $X_w^\Gr := \overline{B w E_0}$
which is isomorphic to $\overline{B wP/P}\subset G/P$.

\subsection{Equivariant cohomology of infinite flag variety} We work with cohomologies with coefficients in $\Q$. The group $T_\Z$ is homotopy equivalent to the restricted product $(S^1)^\Z$, which is a CW-complex of infinite dimension and with infinitely many cells in each dimension.  Then $ET_\Z$ is homotopy equivalent to $(S^\infty)^\Z$, which is again a restricted product where all but finitely many factors must be the basepoint of $S^\infty$.  The classifying space $BT_\Z = ET_\Z/T_\Z$ is the restricted product $(\CP^\infty)^\Z$.  Thus $$H^*(BT_\Z) = H^*_{T_\Z}(\pt) = \Q[\ldots,a_{-1},a_0,a_1,\ldots]=\Q[a].$$

The Schubert cells $\{BwB/B \mid w \in S_\Z\}$ form a $T_\Z$-stable paving of $\Fl$ by finite-dimensional affine spaces.  By standard arguments, $H_*^{T_\Z}(\Fl)$ has a basis given by the fundamental classes $[X_w]$:
\begin{equation}\label{eq:Xw}
H_*^{T_\Z}(\Fl) \cong \bigoplus_{w \in S_\Z} \Q[a] [X_w].
\end{equation}
Similarly,
\begin{equation}\label{eq:XwGr}
H_*^{T_\Z}(\Gr) \cong \bigoplus_{w \in S_\Z^0} \Q[a] [X^\Gr_w].
\end{equation}
The equivariant homology $H_*^{T_\Z}(\Fl) = \varinjlim H_*^{T_\Z}(\Fl(2N))$ is a direct limit of equivariant homologies of finite flag varieties. We define the completed equivariant cohomology $H^*_{T_\Z}(\Fl)'$ of $\Fl$ to be the dual $\Q[a]$-algebra to $H_*^{\TZ}(\Fl)$:
$$
H^*_{T_\Z}(\Fl)' := \Hom_{\Q[a]}(H^*_{\TZ}(\Fl)),\Q[a]) \cong \varprojlim H^*_{T_\Z}(\Fl(2N)) \cong \prod_{w \in S_\Z} \Q[a] \xi^w
$$
where the Schubert classes $\{\xi^w \mid w \in S_\Z\}$ are the cohomology classes dual to the fundamental classes $\{[X_w] \mid w \in S_\Z\}$ under the cap product.  
(Note that we do not invoke Poincare duality: $\Fl$ is infinite-dimensional.) Let $H^*_{T_\Z}(\Fl)$ be the subspace of $H^*_{T_\Z}(\Fl)'$ spanned by the Schubert classes:
$$
H^*_{T_\Z}(\Fl)  \cong \bigoplus_{w \in S_\Z} \Q[a] \xi^w \subset H^*_{T_\Z}(\Fl)'.
$$
This is a kind of restricted dual of $H_*^{\TZ}(\Fl)$.
%

\subsection{Localization and GKM rings for infinite flags and infinite Grassmannian}
Localization \cite{KK, CS, GKM} provides explicit algebraic (GKM) constructions $\Psi$ and $\Psi_\Gr$ of the $\TZ$-equivariant cohomology rings $H_{\TZ}^*(\Fl) \cong \Psi$ and $H_{\TZ}^*(\Gr) \cong \Psi_\Gr$ and their Schubert bases.

Let $\Fun(S_\Z,\Q[a])$ be the $\Q[a]$-algebra of functions $S_\Z\to\Q[a]$ under pointwise product. For $f\in \Fun(S_\Z,\Q[a])$ and $w\in S_\Z$ we write $f|_w$ for $f(w)$. 

Let $R := \{a_i-a_j \mid i \neq j\}$ be the root system of type $\bA_\Z$ and $R^+:=\{a_i-a_j\in R\mid i<j\}$ the positive roots. For a root $\alpha=a_i-a_j$, let $s_\alpha\in S_\Z$ be the transposition swapping $i$ and $j$. 

Let $\Psi'$ be the $\Q[a]$-submodule of $\Fun(S_\Z,\Q[a])$ consisting of functions $f:S_\Z\to\Q[a]$  such that 
\begin{align}\label{E:GKM}
\text{$\alpha$ divides $f|_{s_\alpha w} - f|_w$ \qquad for all $w\in S_\Z$, $\alpha\in R$.}
\end{align}

\begin{example}\label{X:equiv class from poly}
	For $p \in \Q[a]$, define $L_p\in \Fun(S_\Z,\Q[a])$ by $L_p|_w=w(p)$.
	Then $L_p\in \Psi'$. If $p$ is homogeneous of degree one then $L_p$ is an equivariant line bundle class.
\end{example}

\begin{lemma} The $\Q[a]$-submodule $\Psi'\subset \Fun(S_\Z,\Q[a])$ is a $\Q[a]$-subalgebra.
\end{lemma}
\begin{proof} Let $\phi,\psi\in \Psi'$, $\alpha\in R$ and $w\in S_\Z$. Then $(\phi \psi)|_{s_\alpha w} - (\phi\psi)|_w$ is a multiple of $\alpha$ since it is the sum of two multiples of $\alpha$, namely, $\phi|_{s_\alpha w}(\psi|_{s_\alpha w}-\psi|_w) +
	(\phi|_{s_\alpha w} - \phi|_w) \psi|_w$.
\end{proof}

We call $\Psi'$ the completed GKM ring for $\Fl$.
It has a Schubert ``basis" $\{\xi^v\mid v\in S_\Z\}$ (see \eqref{E:GKMcompleted}) which is characterized as follows.

\begin{prop}\label{P:GKM Schubert basis} \cite{KK} There is a unique family of elements $\{\xi^v \mid v\in S_\Z\}\subset \Psi'$ such that
	\begin{align}
	\xi^v|_w &=0 \qquad\text{unless $v\le w$} \\
	\xi^v|_w &\in \Q[a] \text{ is homogeneous of degree $\ell(v)$} \\
	\xi^v|_v &= \prod_{\substack{\alpha\in R^+\\ s_\alpha v < v}} (-\alpha).
	\end{align}
	Moreover,
	\begin{align}\label{E:GKMcompleted}
	\Psi' = \prod_{v\in S_\Z} \Q[a] \xi^v.
	\end{align}
\end{prop}

We define $\Psi := \bigoplus_{v \in S_\Z} \Q[a] \xi^v$, which is the $\Q[a]$-submodule of $\Psi'$ with basis $\xi^v$.  It follows from Theorem \ref{thm:HTFl} below that $\Psi$ is a $\Q[a]$-subalgebra of $\Psi'$. We call $\Psi$ the GKM ring of $\Fl$. Define the GKM ring $\Psi_\Gr$ of $\Gr$ by
\begin{align*}
\Psi_{\Gr} = \{f\in \Psi \mid \text{$f|_{ws_i}=f|_w$ for all $w\in S_\Z$ and $i\ne0$}\}.
\end{align*}

Recall the bijection $\lambda \mapsto w_\la$ \eqref{E:wla}.

\begin{proposition}
	The $\Q[a]$-algebra $\Psi_\Gr$ has a $\Q[a]$-basis $\{ \xi^{w_\la}\mid \la\in\Par\}$. 
\end{proposition}

The GKM rings of $\Fl$ and $\Gr$ are explicit realizations of the equivariant cohomology rings $H_{\TZ}^*(\Fl)$ and $H_{\TZ}^*(\Gr)$ and their Schubert bases.

\begin{prop}\label{P:GKM} There are $\Q[a]$-algebra isomorphisms
	\begin{align}
\label{E:GKMFl}
	H^*_{T_\Z}(\Fl) &\cong \Psi\\
\label{E:GKMGr}
	H_{\TZ}^*(\Gr) &\cong \Psi_\Gr
	\end{align}
under which the Schubert bases correspond.
\end{prop}
\begin{proof}
	We first show that $H^*_{T_\Z}(\Fl)' \cong \Psi'$.  
	Let $\Psi(2N)$ be the $\Q[a]$-submodule of $\Fun(S_{[-N,N-1]},\Q[a])$ consisting of functions $f:S_{[-N,N-1]}\to\Q[a]$ such that 
	$$
	\text{$\alpha$ divides $f|_{s_\alpha w} - f|_w$ \qquad for all $w, s_\alpha w\in S_{[-N,-N-1]}$, $\alpha\in R$.}
	$$
	By \cite{KK}, $H^*_{T_\Z}(\Fl(2N)) \cong \Psi(2N)$.  The inclusion $\iota_{2N}:\Fl(2N) \hookrightarrow \Fl(2(N+1))$ is $T_\Z$-equivariant and maps the torus fixed point $w \in S_{[-N,N-1]} \in \Fl(2N)^{T_\Z}$ to the torus fixed point $w \in S_{[-N-1,N]} \in \Fl(2(N+1))^{T_\Z}$.  Thus the pullback map $\iota_{2N}^*:H^*_{T_\Z}(\Fl(2(N+1))) \to H^*_{T_\Z}(\Fl(2N))$ can be identified with the restriction map $r_{2N}:\Psi(2(N+1)) \to \Psi(2N)$.  We conclude that
	$$
	H^*_{T_\Z}(\Fl)' \cong \varprojlim H^*_{T_\Z}(\Fl(2N)) \cong  \varprojlim \Psi(2N) \cong \Psi'.
	$$
	By the usual characterization of Schubert classes of $\Psi(2N)$ and $\Psi(2(N+1))$ (cf. Proposition~\ref{P:GKM Schubert basis}), the restriction map $r_{2N}$ sends a Schubert class to either a Schubert class, or to 0.  This shows that the isomorphism $H^*_{T_\Z}(\Fl)' \cong \Psi'$ sends the $\xi^w \in H^*_{T_\Z}(\Fl)'$ defined in terms of the basis dual to the homology basis, to the same named element of $\Psi'$ described in Proposition \ref{P:GKM Schubert basis}. This proves \eqref{E:GKMFl}.
	The proof of \eqref{E:GKMGr} is similar.
\end{proof}

\subsection{Realization of Schubert basis of GKM ring by backstable Schubert polynomials}

We show that the GKM rings are realized by polynomial algebras such that the 
Schubert bases of the GKM rings correspond to backstable Schubert polynomials and double Schur polynomials respectively. 

\begin{thm}\label{thm:HTFl}
We have isomorphisms of $\Q[a]$-algebras:
\begin{align}
\label{E:Fl poly to GKM iso}
\bR(x,a) &\cong \Psi &\qquad \bS_v(x;a)&\longmapsto \xi^v &\qquad&\text{for $v\in S_\Z$}  \\
\label{E:Gr poly to GKM iso}
\Lambda(x||a) &\cong \Psi_\Gr &\qquad s_\la(x||a) &\longmapsto \xi^\lambda_\Gr&&\text{for $\la\in\Par$.}
\end{align} 
\end{thm}
\begin{proof}
Let $f\in \bR(x;a)$.   Then $f$ can be considered an element of $\Fun(S_\Z,\Q[a])$ by \eqref{E:localize f}. For any $\alpha=a_i-a_j\in R$ the element
\begin{align*}
f(s_\alpha^x w x; a) - f(w x; a) &= (s_\alpha^x - \id)f(wx;a)
\end{align*}
is divisible by $x_\alpha=x_i-x_j$. Applying $\epsilon$ from \textsection \ref{SS:loc back symmetric} we see that  $f\in \Psi$. It is immediate that the map $\bR(x;a)\to\Psi$ is a $\Q[a]$-algebra homomorphism. 

It is not hard to see that if $f\in\bR(x;a)$ is nonzero then it has a nonzero localization. Thus $\bR(x;a)$ embeds into $\Psi$.

One may deduce that $\bS_v(x;a)\mapsto \xi^v$ by checking 
the conditions of Proposition \ref{P:GKM Schubert basis}. In turn, these can be verified by Proposition \ref{P:double backstable finite loc} and the analogue of Proposition \ref{P:GKM Schubert basis} for $S_n$, which is satisfied by the localizations of double Schubert polynomials $\S_v(w\ap;\ap)$ for $v,w\in S_n$ \cite[Remark 1]{Bil}.

The statements for $\Gr$ are obtained by taking $S_{\ne0}$-invariants.
\end{proof}

 Theorem \ref{thm:HTFl} specializes to:
\begin{thm}\label{thm:HFl}
We have isomorphisms of $\Q$-algebras
\begin{align*}
H^*(\Fl) &\xrightarrow{\sim} \bR \\
H^*(\Gr) &\xrightarrow{\sim} \Lambda
\end{align*} 
where the images of Schubert classes are $\bS_w$ and $s_\la$ respectively.
\end{thm}

\begin{rem}
The decomposition $H^*(\Fl) = H^*(\Gr) \otimes_\Q \Q[x]$ can be explained as follows. For $[p,q]\subset\Z$ an interval of integers, let $\Fl_{[p,q]}$ be the space of flags $F_\bullet\in \Fl$ such that $F_i=E_i$ for $i\in\Z\setminus[p,q]$. Then $\Fl_{[p,q]}\cong \Fl(E_{q+1}/E_{p-1})$ is isomorphic to the variety of complete flags in a $(q-p+2)$-dimensional complex vector space. 
Let $\Fl_{>0} = \bigcup_{n \in \Z_{>0}} \Fl_{[1,n]}$ and $\Fl_{< 0}=\bigcup_{n\in \Z_{<0}} \Fl_{[n,-1]}$.

For a fixed $\Lambda \in \Gr$, the fiber $\pi_0^{-1}(\Lambda) \subset \Fl$ is isomorphic to $\Fl_{< 0} \times \Fl_{> 0}$ which has cohomology ring $\Q[x_{>0}] \otimes \Q[x_{\leq 0}] \cong \Q[x]$.  We expect the fibration $\pi_0: \Fl \to \Gr$ to be topologically trivial.
\end{rem}

\subsection{Shifting}\label{ssec:shift}
For later use, we briefly consider the other components $\Fl^{(p)}$ and $\Gr^{(p)}$ of the Sato flag variety and Sato Grassmannian.  Let $\sh: \Z \to \Z$ be the bijection sending $i$ to $i+1$ for all $i \in \Z$. Consider the group $\cS_\Z := \langle \sh \rangle \ltimes S_{\Z}$, the group of bijections of $\Z$ generated by $S_\Z$ and by $\sh$.  The $T_\Z$-fixed points of $\Fl^{(p)}$ for $p\in \Z$ are indexed by $\sh^p w$ for $w\in S_{\Z}$.  The equivariant cohomology $H_{T_{\Z}}^*(\Fl^{(p)})$ has Schubert basis $\xi^{\sh^p v}$ for $v\in S_{\Z}$. 

Let $\shift_a$ be as in \textsection \ref{SS:double Schur}.

\begin{proposition}\label{prop:shiftflag}
For every $p\in \Z$, there is an isomorphism of rings
\begin{align*}
H_{T_{\Z}}^*(\Fl^{(p)}) &\to \bR(x,a) &
\xi^{\sh^p v} & \mapsto  \bS_{\sh^p v}(x;a)
\end{align*}
satisfying $i_{\sh^p w}^*(\xi^{\sh^p v}) = \bS_{\sh^p v}(\sh^p wa;a)$. 
\end{proposition}

\begin{proof}
The Schubert class $\xi^{\sh^p v}$ is determined by $i_{\sh^p w}^*(\xi^{\sh^p v}) = \shift_a^p(i_w^*(\xi^v))$.  Since $\sh^p w(x_i) = x_{p + w(i)}$, from the definition, we have
$\bS_{\sh^p v}(\sh^p wa;a) = \shift_a^p \bS_v(wa;a)$.  The result follows.
\end{proof}

%

The equivariant cohomology $H_{T_{\Z}}^*(\Gr^{(p)})$ has Schubert basis $\xi^{\sh^p \la}$ for $\la\in \Par$. 
Extend the definition of double Schur functions by $s_{\sh^p \la}(x||a):=\shift_a^p  s_\la(x||a) \in \Lambda(x||a)$. 

\begin{proposition}\label{prop:shiftGr}
For every $p\in \Z$, there is an isomorphism of rings
\begin{align*}
H_{T_{\Z}}^*(\Gr^{(p)}) &\to \Lambda(x||a) &
\xi^{\sh^p \la} & \mapsto  s_{\sh^p \la}(x||a)
\end{align*}
satisfying $i_{\sh^p w}^*(\xi^{\sh^p \la}) = s_{\sh^p \la}(\sh^p wa||a)$.
\end{proposition}

%
%

\section{NilHecke algebra and Hopf structure}\label{sec:local}
We show that the $\Q[a]$-algebra isomorphisms of Theorem \ref{thm:HTFl} 
preserve additional structure: for $\Psi$, two commuting actions of the nilHecke algebra of Kostant and Kumar \cite{KK} and for $\Psi_\Gr$, the Hopf $\Q[a]$-algebra structure.

\subsection{NilHecke algebra}\label{ssec:nilHecke}
Let $\Q(a)$ be the fraction field of $\Q[a]$.
Let $\Q(a)[S_\Z]$ be the twisted group algebra, with product $(f u)(g v) = (f u(g)) (uv)$ for $f,g\in\Q(a)$ and $u,v\in S_\Z$. The ring $\Q(a)[S_\Z]$ acts on $\Q(a)$: $S_\Z$ acts by permuting variables and $\Q(a)$ acts by left multiplication. For $i\in \Z$, we regard
$A_i$ as being an element of $\Q(a)[S_\Z]$:
\begin{align}\label{E:ddiff in Aloc}
A_i := \alpha_i^{-1}(\id-s_i)\in \Q(a)[S_\Z].
\end{align}
The elements $A_i$ act on $\Q[a]$. 

The (infinite) nilHecke algebra $\bA$ is by definition the $\Q$-subalgebra of $\Q(a)[S_\Z]$ generated by $\Q[a]$ and the $A_i$. We have the commutation relation
\begin{align}\label{E:ddiff derivation}
A_i f = A_i(f) + s_i(f) A_i\qquad\text{for $i\in \Z$, $f\in\Q[a]$.}
\end{align}
One may show that the expansion of $A_w\in \Q(a)[S_\Z]$ (see \eqref{E:ddiff w}) into the left $\Q(a)[S_\Z]$-basis $S_\Z$, is triangular with respect to the Bruhat order.
It follows that the $A_w$ are a left $\Q[a]$-basis of $\bA$.

Viewing $S_\Z \subset \bA$ via $s_i = 1 - \alpha_i A_i$, for $v,w\in S_\Z$ define the elements $e^v_w\in \Q[a]$ by the expansion of Weyl group elements into the basis $A_v$ of $\bA$:
\begin{align}\label{E:Schubert locs}
w = \sum_{v\in S_\Z} e^v_w A_v.
\end{align}


\begin{example} Using \eqref{E:ddiff derivation} we have
	\begin{align*}
	s_2 s_1 &= (1-\alpha_2 A_2)(1-\alpha_1A_1) 
	= 1 - \alpha_1 A_1 - (\alpha_1+\alpha_2)A_2 + \alpha_2(s_2(\alpha_1))A_2A_1.
	\end{align*}
	Therefore $e_\id^{s_2s_1}=1$, $e_{s_1}^{s_2s_1}=a_2-a_1$,
	$e_{s_2}^{s_2s_1}=a_3-a_1$, $e_{s_2s_1}^{s_2s_1}=(a_3-a_2)(a_3-a_1)$, and $e_v^{s_2s_1}=0$ for other $v\in S_\Z$.
\end{example}

\begin{prop} \label{P:loc rec}
	The elements $e^v_w$ are uniquely defined by the initial conditions
	\begin{align}
	e^v_\id &= \delta_{v,\id} \qquad\text{for all $v\in S_\Z$}
	\end{align}
	and either the following:
	\begin{enumerate}
		\item[(a)] For all $w\in S_\Z$ and $i\in\Z$ such that $ws_i<w$, 
		\begin{align}
		\label{E:loc rec right}
		e^v_w &= e^v_{ws_i} + w(\alpha_i) e^{vs_i}_{ws_i} \chi(vs_i<v) \\
		\label{E:loc rec right other case}
		e^{vs_i}_w \chi(vs_i<v) &= e^{vs_i}_{ws_i} \chi(vs_i<v)
		\end{align}
		\item[(b)] For all $w\in S_\Z$ and $i\in\Z$ such that $ws_i<w$, 
		\begin{align}\label{E:loc rec left}
		e^v_w &= s_i(e^v_{s_i w}) -\alpha_i s_i(e^{s_iv}_{s_iw})\chi(s_iv<v) \\
		\label{E:loc rec left other case}
		e^{s_iv}_w \chi(s_iv<v) &= s_i(e^{s_iv}_{s_iw}) \chi(s_iv<v)
		\end{align}
	\end{enumerate}
\end{prop}

Let $w\in S_\Z$ and $\a = (i_1,i_2,\dotsc,i_\ell) \in \Red(w)$. For $1\le j\le \ell$ let
$$
\beta_j(\a) = s_{i_1} s_{i_2} \cdots s_{i_{j-1}}(-\alpha_{i_j}) = s_{i_1} s_{i_2} \cdots s_{i_{j-1}}(a_{i_j+1} - a_{i_j}).
$$
\begin{prop}[\cite{AJS,Bil}]\label{prop:Billey}
Let $w, v \in S_\Z$ and $\a = (i_1,i_2,\dotsc,i_\ell) \in \Red(w)$.  Then we have the closed formula
$$
e^v_w = \sum_{\b \subset \a} \prod_{i_j \in \b} \beta_j(\a)
$$
summed over all subwords of $\a$ that are reduced words for $v$.

\begin{example} Let $w=s_1s_2s_1$, $\a=(1,2,1)$ and $v=s_1$. We have
	$\beta_1(\a)=a_2-a_1$, $\beta_2(\a)=s_1(a_3-a_2)=a_3-a_1$, $\beta_3(\a)=s_1s_2(a_2-a_1)=a_3-a_2$. There are two subwords of $\a$ that are reduced words of $v$, namely, $(i_1)$ and $(i_3)$. Therefore $e^{s_1}_{s_1s_2s_1}=\beta_1(\a)+\beta_3(\a)=a_3-a_1$. Using $\a'=(2,1,2)\in \Red(w)$ we have $\beta_2(\a')=a_3-a_1$, a unique subword of $\a'$ that is a reduced word of $v$, namely, $(i_2)$, and $e^{s_1}_{s_1s_2s_1} = \beta_2(\a')=a_3-a_1$.
\end{example}

\end{prop}

It follows from Proposition~\ref{prop:PsiA}(3) below that the elements $e^v_w$ which were defined using $\bA$, are none other than the localization values of equivariant Schubert classes at torus-fixed points.

\begin{prop}\label{P:schub loc} For $v,w \in S_\Z$, we have $\xi^v|_w = e^v_w$.
\end{prop}

Recall the automorphism $\omega$ of \textsection \ref{SS:coproduct backstable}.

\begin{lemma} For $v,w\in S_\Z$, we have $\omega(\xi^v|_w) = \xi^{\omega(v)} |_{\omega(w)}$.
\end{lemma}

\subsection{NilHecke actions}
The GKM ring $\Psi$ affords two actions of $\bA$ which commute.  The results in this subsection are $A_\Z$-variants of the constructions of Kostant and Kumar \cite{KK}.

\begin{prop} \label{P:nilHecke on GKM} \
	\begin{enumerate}
		\item There is an action $\raction$ of $\bA$ on $\Psi$ defined by
		\begin{align}
		\label{E:right partial}
		(A_i\raction \psi)|_w &= (w(\alpha_i))^{-1} (\psi|_w - \psi|_{ws_i}) \\
		\label{E:right poly}
		(p\raction \psi) &= L_p \,\psi\qquad\text{for $p\in \Q[a]$.}
		\end{align}
		where $L_p$ is defined in Example \ref{X:equiv class from poly}.
		It satisfies
		\begin{align}\label{E:right Weyl}
		(u \raction \psi)|_w = \psi|_{wu}\qquad\text{for $u\in S_\Z$.}
		\end{align}
		\item There is an action $\laction$ of $\bA$ on $\Psi$ is given by
		\begin{align}
		\label{E:left partial}
		(A_i\laction \psi)|_w &= \alpha_i^{-1} (\psi|_w - s_i(\psi|_{s_iw})) \\
		\label{E:left poly}
		(p \laction \psi)|_w &= p\, \psi|_w \qquad\text{for $p\in \Q[a]$.}
		\end{align}
		In particular the action of $S_\Z$ on $\Psi$ by $\laction$ is by conjugation:
		\begin{align}\label{E:left Weyl}
		(u \laction \psi)|_w = u(\psi|_{u^{-1} w}).
		\end{align}
		\item The two actions commute.
		\item For $v\in S_\Z$ and $i\in\Z$,
		\begin{align}
		\label{E:ddiff right} A_i \raction \xi^v &= 
		\begin{cases}
		\xi^{vs_i} &\text{if $vs_i<v$} \\
		0 & \text{otherwise} \\
		\end{cases} \\
		\label{E:ddiff left} A_i\laction \xi^v &= 
		\begin{cases}
		- \xi^{s_i v} &\text{if $s_iv<v$} \\
		0 & \text{otherwise.}
		\end{cases}
		\end{align}
	\end{enumerate}
\end{prop}
\begin{proof} We identify any function $\psi\in \Fun(S_\Z,\Q(a))$ 
	with the left $\Q(a)$-module homomorphism $\psi\in \mathrm{Hom}_{\Q(a)}(\Q(a)[S_\Z],\Q(a))$ by extension by left $\Q(a)$-linearity. 
	
	For (1), there is an action of $\Q(a)[S_\Z]$ on $\Fun(S_\Z,\Q(a))$ defined by
	\begin{align*}
	(b\raction \psi)|_w = \psi|_{w b}
	\qquad\text{for $b\in \Q(a)[S_\Z]$.}
	\end{align*}
	For $b=u\in S_\Z$, we obtain \eqref{E:right Weyl}. For $a=A_i$ and $a=p$, we have
	\begin{align*}
	\psi|_{w A_i} &= \psi|_{w \alpha_i^{-1} (\id - s_i)} 
	= \psi|_{w(\alpha_i)^{-1} w(\id - s_i)} 
	= w(\alpha_i)^{-1} \left( \psi|_w-\psi|_{w s_i}\right)\\
	\psi|_{w p} &= \psi|_{(w(p) w} = w(p) \psi|_w
	\end{align*}
	which agrees with \eqref{E:right partial} and \eqref{E:right poly}.
	To see that $\raction$ restricts to an action of $\bA$ on $\Psi'$, let $\psi\in \Psi'$. Note that if $\alpha = w(\alpha_i)$ then $ws_i = s_\alpha w$ so that $\alpha$ divides $\psi|_w-\psi|_{s_\alpha w} = \psi|_w - \psi|_{w s_i}$ and $A_i \raction \psi\in \Psi'$. For $p\in \Q[a]$ we have $p\raction \psi = L_p \psi\in \Psi'$ since $\Psi'$ is a ring.
	
	For (2), again working over $\Q(a)$ the $\laction$-action is defined by \eqref{E:left poly} and \eqref{E:left Weyl}. To show these define an action of $\Q(a)[S_\Z]$ one must verify that the actions of $p\in \Q(a)$ and $u\in S_\Z$ have the proper commutation relation:
	\begin{align*}
	(u\laction p \laction \psi)|_w &=
	u ( p\laction \psi)|_{u^{-1} w} 
	= u( p \psi|_{u^{-1}w}) 
	= u(p) u(\psi|_{u^{-1}w}) 
	= (u(p) \laction (u \laction \psi))|_w.
	\end{align*}
	To check that $\laction$ restricts to an action of $\bA$ on $\Psi'$, let $\psi\in\Psi'$. Note that for any $p\in \Q[a]$, $s_i(p)-p = \alpha_i g$ for some $g\in \Q[a]$ (namely, $g=-A_i(p)$.  Then $\psi|_{s_iw}-\psi|_w = \alpha_i h$ for some $h\in \Q[a]$. We have
	\begin{align*}
	\psi|_w - s_i (\psi|_{s_i w}) &= \psi|_w - s_i(\psi|_w) + s_i(\psi|_w) - s_i(\psi|_{s_iw}) \\
	&= (\id-s_i)(\psi|_w) - s_i(\alpha_i h) \\
	&= -\alpha_i g + \alpha_i s_i(h) \in \alpha_i \Q[a]
	\end{align*}
	for some $g\in \Q[a]$. Therefore $A_i \laction \psi\in \Psi$.
	For $p\in \Q[a]$, it is immediate that $p\laction\psi\in \Psi'$.
	
	For (3) it is straightforward to check over $\Q(a)$ that the operators $p \raction $ and $A_i \raction$ commute with the operators $q\laction$ and $A_j \laction$.
	
	(4) follows by Propositions \ref{P:loc rec} and 
	\ref{P:schub loc}. (4) implies that the two actions preserve $\Psi\subset\Psi'$.
\end{proof}

The nilHecke algebra $\A$ has a comultiplication map $\Delta: \A \to \A \otimes_{\Q[a]} \A$ given by 
\begin{equation}\label{eq:wgrouplike}
\Delta(w) := w \otimes w \qquad \mbox{for $w \in S_\Z \subset \A$}
\end{equation}
and extending by linearity over $\Q(a)$.  One can show that \eqref{eq:wgrouplike} is equivalent to
\begin{equation}\label{eq:deltapartial}
\Delta(A_i) = A_i \otimes 1 + s_i \otimes A_i \qquad \mbox{for $i \in \Z$.}
\end{equation}
We caution that $\A$ is \emph{not} a Hopf algebra.

Define a pairing $\pair{\cdot}{\cdot}: \Fun(S_\Z,\Q(a)) \otimes_{\Q(a)} \Q(a)[S_\Z] \to \Q(a)$ by
\begin{equation}\label{eq:PsiApairing}
\pair{\psi}{\sum a_w w} =\sum a_w \psi(w)
\end{equation}
where $a_w \in \Q(a)$.  The following result follows from Proposition~\ref{P:nilHecke on GKM}.

\begin{prop}\label{prop:PsiA}\
\begin{enumerate}
\item
Under the pairing $\pair{\cdot}{\cdot}$, $\Psi'$ and $\A$ are identified as dual $\Q[a]$-modules.
\item
The multiplication of $\Psi'$ is dual to the comultiplication $\Delta$ of $\A$.
\item
For $v,w \in S_\Z$, we have $\pair{\xi^v}{A_w} = \delta_{vw}$.   
\end{enumerate}
\end{prop}


\begin{prop}\label{prop:bAbR} The $\Q[a]$-algebra isomorphism 
	\eqref{E:Fl poly to GKM iso} is a $\bA\times\bA$-module isomorphism.
\end{prop}
\begin{proof} 
	The operators on $\bR(x;a)$ given by $A_i^x$, multiplication by $x_i$, $A_i^a$, and multiplication by $a_i$, correspond to the operators on $\Psi$ given by $A_i \raction$, $a_i \raction$, $A_i \laction$, and $a_i \laction$ respectively.
\end{proof}


\subsection{Hopf structure on $\Psi_\Gr$}\label{ssec:partialproduct}
%
%

Via \eqref{E:Gr poly to GKM iso}, the ring $\Psi_\Gr$ attains the structure of a Hopf algebra over $\Q[a]$.  We now describe the comultiplication map $\Delta: \Psi_\Gr \to \Psi_\Gr \otimes_{\Q[a]} \Psi_\Gr$.

\begin{lemma}\label{lem:seqw}
	Let $\psi \in \Psi_\Gr$ and $w^{(1)},w^{(2)},\ldots \in S_\Z^0$ be a sequence satisfying $|I_{w^{(k)},+}| =|I_{w^{(k)},-}|= k$.  Then $\psi$ is uniquely determined by $\psi|_{w^{(i)}}$.
\end{lemma}
\begin{proof}
Fix the sequence $w^{(1)},w^{(2)},\ldots \in S_\Z^0$.  Let $f = \sum_{\la} a_\la p_\la(x||a) \in \Lambda(x||a)$ where $a_\la \in \Q[a]$.  It suffices to show that $f|_{w^{(k)}} \neq 0$ for some $k$.  Let $S$ be the finite set of indices $i$ such that $a_i$ appears in some $a_\la$.  For sufficiently large $k$, the set $I_{w,+} \setminus S$ has cardinality greater than $\deg(f)$.  If $\mu = (\mu_1,\ldots,\mu_\ell) $ is minimal in dominance order in the set $A=\{\la \mid a_\la \neq 0\}$, by Lemma \ref{lem:locp} the polynomial $p_\mu(x||a)|_{w^{(k)}} \in \Q[a]$ contains a term of the form $a_\mu a_{r_1}^{\mu_1} a_{r_2}^{\mu_2} \cdots a_{r_\ell}^{\mu_\ell}$ where $r_1>r_2>\cdots>r_\ell$ are the $\ell$ largest elements in $I_{w,+} \setminus S$.  This monomial does not appear in $p_\la(x||a)|_{w^{(k)}}$ for $\la \in A \setminus \{\mu\}$.  The coefficient of this monomial must thus be nonzero in $f|_{w^{(k)}} \in \Q[a]$.
\end{proof}

There is a partial multiplication map $S_\Z^0 \times S_\Z^0 \to S_\Z^0$.  The product of $x \in S_\Z^0$ and $y\in S_\Z^0$ is equal to $z \in S_\Z^0$ if (1) $I_{x,+} \cap I_{y,+} = \emptyset = I_{x,-} \cap I_{y,-}$ and (2) $I_{x,\pm} \cup I_{y,\pm} = I_{z,\pm}$.  

\begin{prop}\label{prop:locHopf}
	There is a unique Hopf structure on $\Psi_{\Gr}$ with comultiplication $\Delta: \Psi_\Gr \to \Psi_\Gr \otimes_{\Q[a]} \Psi_\Gr$ given by 
	\begin{equation}\label{eq:GrHopf}
	\Delta(\psi)|_{x \otimes y} = \psi|_{xy}
	\end{equation}
	whenever $x,y, \in S_\Z^0$ and $xy \in S_\Z^0$ is defined.  With this Hopf structure, the map \eqref{E:Gr poly to GKM iso} is a $\Q[a]$-Hopf algebra isomorphism.
\end{prop}
\begin{proof}
Suppose the product $xy$ is well-defined.  By Lemma \ref{lem:locp}, we have
	\begin{align*}
	p_k(x||a)|_{xy} &= \sum_{i\in I_{xy,+}} a_i^k - \sum_{i\in I_{xy,-}} a_i^k \\
	&= (\sum_{i\in I_{x,+}} a_i^k - \sum_{i\in I_{x,-}} a_i^k) + (\sum_{i\in I_{y,+}} a_i^k - \sum_{i\in I_{y,-}} a_i^k)\\
	&=p_k(x||a)|_x + p_k(x||a)|_y\\
	&= (p_k(x||a)\otimes 1)|_{x \otimes y} + (1 \otimes p_k(x||a))|_{x \otimes y}\\
	&=\Delta(p_k(x||a))_{x\otimes y}.
	\end{align*}
Thus \eqref{eq:GrHopf} is consistent with the comultiplication of $\Lambda(x||a)$.  By the same argument as in the proof of Lemma \ref{lem:seqw}, $\Delta(\psi)$ is completely determined by its values $x \otimes y$ for which $xy$ is defined. 
\end{proof}

\section{Homology Hopf algebra}
In this section we identify Molev's dual Schur functions
\cite[\textsection 3.1]{M} with the equivariant homology Schubert basis of $\Fl$. 
Molev gave determinantal formulas for the Schur expansion of dual Schurs and the inverse expansion \cite{M}: we give new formulas and simple proofs for these coefficients expressed in terms of usual Schubert polynomials.
We give a divided difference formula for dual Schur functions. 
While we found the operators independently, the form presented here, via
conjugation by a Cauchy kernel, is due to Naruse \cite{Na}.  We further show that a specialization of dual Schur functions represent classes (deforming the Schur functions) defined by Knutson and Lederer \cite{KL}.

Molev \cite{M} gave an explicit combinatorial formula for the general structure constants for the dual Schurs. By Theorem \ref{thm:doubleStanleypositivity} we know that these constants, which are elements of $\Q[a]$, have a certain positivity property. When one factor is a hook we 
obtain a suitably positive formula for the structure constants.

\subsection{Molev's dual Schur functions}
Let $\La(y)$ denote the $\Q$-algebra of symmetric functions in $y=(y_0,y_{-1},y_{-2},\dotsc)$ and $\hLa(y||a)$ the completion of $\Q[a]\otimes_{\Q} \La(y)$ whose elements are formal (possibly infinite) $\Q[a]$-linear combinations $\sum_{\lambda \in \Par} a_\lambda s_\lambda(y)$ of Schur functions, with $a_\lambda \in \Q[a]$.  The ring $\hLa(y||a)$ is a $\Q[a]$-Hopf algebra with coproduct
$\Delta(p_k(y)) = 1 \otimes p_k(y) + p_k(y) \otimes 1$.

Define the Cauchy kernel
\begin{align*}
\Omega = \Omega(\xm y/ \am y) = \prod_{i,j\le0} \dfrac{1-a_iy_j}{1-x_iy_j} = \exp\left(\sum_{k \geq 0} \frac{1}{k}p_k(x||a) p_k(y)\right).
\end{align*}
This induces the structure of dual $\Q[a]$-Hopf algebras on $\La(x||a)$ and $\hLa(y||a)$.
Write $\pair{\cdot}{\cdot}$ for the corresponding pairing $\La(x||a) \otimes_{\Q[a]} \hLa(y||a) \to \Q[a]$. Then by definition
\begin{align}\label{E:equivariant pairing}
\pair{s_\la(x/a)}{s_\mu(y)} = \delta_{\la,\mu}.
\end{align}
Define $\hs_\la(y||a)\in\hLa(y||a)$ by duality with the double Schur functions:
\begin{align}\label{E:dual Schur}
\pair{s_\la(x||a)}{\hs_\mu(y||a)} &= \delta_{\la\mu}.
\end{align}
The $\hs_\la(y||a)$ are the nonpositive variable analogue of Molev's dual Schur functions \cite[\textsection 3.1]{M}.  The ring $\hLa(y||a)$ consists of formal (possibly infinite) $\Q[a]$-linear combinations of the dual Schur functions $\hs_\la(y||a)$.

From Proposition~\ref{P:GKM} and Theorem~\ref{thm:HTFl}, we know that the ring $H^*_{\TZ}(\Gr)$ is a Hopf algebra.  The dual Hopf-algebra is the completion $H_*^{\TZ}(\Gr)' = \prod_{w\in S_\Z} \Q[a] [X_w]$ of the equivariant homology $H_*^{\TZ}(\Gr)$ (see \eqref{eq:XwGr}).  Since $\Lambda(x||a)$ and $\hLa(y||a)$ are Hopf-dual, the following is immediate from the definition \eqref{E:dual Schur}.

\begin{prop}\label{prop:dualSchur}
There is a Hopf $\Q[a]$-algebra isomorphism $H_*^{\TZ}(\Gr)'\to \hLa(y||a)$ sending
the equivariant Schubert class $[X_w]$ of \eqref{eq:XwGr} to the dual Schur function $\hs_\la(y||a)$. 
\end{prop}


Recall the element $w_{\la/\mu} \in S_\Z$ from \eqref{E:w skew}.
Proposition \ref{prop:doublesuper} implies the following.

\begin{prop}\label{prop:dualsuper} For $\mu \in \Par$, we have
	\begin{align*}
	\hs_\mu(y||a) &= \sum_{\substack{\la\supset\mu\\ d(\la)=d(\mu)}}
	\S_{w_{\la/\mu}}(a) s_\la(y) \\
	s_\mu(y) &= \sum_{\substack{\la\supset\mu\\ d(\la)=d(\mu)}}
	(-1)^{|\la|-|\mu|}\S_{w_{\la/\mu}^{-1}}(a) \hs_\la(y||a).
	\end{align*}
\end{prop}

\begin{example}\label{X:hook dual Schur}
	Let $b,c\in\Z_{\ge0}$ and let $\mu = (b+1,1^c)$ be a hook partition. The partitions $\la\supset\mu$ with $d_\la=d_{\mu}=1$ have the form $\la=(B+1,1^C)$ such that $B\ge b$ and $C\ge c$. Let $w=w_{\la/\mu} = s_{-C} s_{1-C}\dotsm s_{-1-c} s_B s_{B-1} \dotsm s_{b+1}$. We have 
	\begin{align*}
	\S_w(a) &=\S_{s_{-C}\dotsm s_{-1-c}}(a) \S_{s_B \dotsm s_{b+1}}(a)\\
	&= \omega(\S_{s_C \dotsm s_{c+1}}) \S_{s_B\dotsm s_{b+1}}(a) \\
	&= \omega(h_{C-c}(a_1,\dotsc,a_{c+1})) h_{B-b}(a_1,\dotsc,a_{b+1}) \\
	&= (-1)^{C-c} h_{C-c}(a_0,a_{-1},\dotsc,a_{-c})h_{B-b}(a_1,\dotsc,a_{b+1}) .
	\end{align*}
	\begin{align*}
	\hs_{(b+1,1^c)}(y||a) &= \sum_{B\ge b, C\ge c} (-1)^{C-c} h_{C-c}(a_0,a_{-1},\dotsc,a_{-c})h_{B-b}(a_1,\dotsc,a_{b+1}) s_{(B+1,C)}(y).
	\end{align*}
\end{example}

\subsection{Homology divided difference operators}\label{SS:homology divided differences}

Since $\alpha_0=a_0-a_1$, we use expressions such as
\begin{align*}
\Omega(-\alpha_0 y_{\le0}) &= \Omega(a_1 y_{\le0}/ a_0 y_{\le0}) =
\prod_{k\le0} \dfrac{1 - a_0 y_k}{1-a_1y_k}.
\end{align*}

\begin{rem} The expression $\Omega(a_1 y/a_0 y)$ should be viewed as the action of the translation element for the weight $\theta=a_1-a_0$ in a large rank limit of the affine type A root system. 
\end{rem}

For $i\in\Z$, define the operators
\begin{align}
\tilde{s}_i^a &:= \Omega(xy/ay) s_i^a \Omega(ay/xy) \\
\delta_i &:= \Omega(xy/ay)\,A_i^a \,\Omega(ay/xy)].
\end{align}
It is clear that these operators, being conjugate to the operators $s_i^a$ and $A_i^a$ respectively, satisfy the type A braid relations. Thus for any $w\in S_\Z$, $\delta_w=\delta_{i_1}\dotsm\delta_{i_\ell}$ makes sense for any reduced word $(i_1,i_2,\dotsc,i_\ell)\in\Red(w)$.

Since $\Omega(ay/xy)$ is $s_i^a$ invariant for $i\ne0$ we have
\begin{align}\label{E:delnonzero}
\tilde{s}_i^a &= s_i &\qquad&\text{for $i\ne0$}\\
\delta_i &= A_i^a &&\text{for $i\ne 0$.}
\end{align}
Since $s_0^a$ only affects the variable $a_0$ and no others in $\Omega[(a-x)y]$, we have the following operator identities, where $f\in\hLa(y||a)$ acts by multiplication by $f$:
\begin{align}
\tilde{s}_0^a &= \Omega(a_1y/a_0y) s_0^a \\
\delta_0 &= \alpha_0^{-1}(\id - \tilde{s}_0^a) \\
\label{E:sdelta}
\tilde{s}_i^a \delta_i &= \delta_i \\ 
\label{E:deltaf}
\delta_i f &= f \delta_i + A_i(f) \tilde{s}_i^a.
\end{align}
The last two follow by conjugating the relations $s_i^a A_i^a=A_i^a$ and $A_i f = f A_i + A_i(f) s_i^a$ by $\Omega:=\Omega[(xy/ay]$.

The {\it diagonal index} (or {\it content}) of a box in row $i$ and column $j$ is by definition the integer $j-i$. For $\la\in \Par$ and $d\in \Z$ let $\la+d$ denote the partition obtained by adding a corner to $\la$ in the $d$-th diagonal if such a corner exists. Define $\la-d$ similarly for removal of the corner in diagonal $d$ if it exists. By convention, if a symmetric function is indexed by $\la\pm d$ and the relevant cell in diagonal $d$ does not exist then the expression is interpreted as $0$. In particular, by Proposition \ref{P:left ddiff} we have
\begin{align}\label{E:left ddiff double Schur}
A_i^a s_\la(x||a) = -s_{\la-i}(x||a) \qquad\text{for all $\la\in\Par$ and $i\in\Z$.}
\end{align}

\begin{prop} \label{P:left ddiff homology} For all $\mu\in\Par$ and $i\in\Z$, we have
	\begin{align} \label{E:left ddiff homology}
	\delta_i(\hs_\mu(y||a)) &= \hs_{\mu+i}(y||a).
	\end{align}
\end{prop}
\begin{proof} Using \eqref{E:deltaf}, we have
	\begin{align*}
	0 &= \Omega A_i^a(1) = \delta_i(\Omega) 
	= \delta_i(\sum_\la s_\la(x||a) \hs_\la(y||a)) 
	= \sum_\la ( s_\la(x||a) \delta_i(\hs_\la(y||a)) - s_{\la-i}(x||a) \tilde{s}^a_i(\hs_\la(y||a)).
	\end{align*}
	Taking the coefficient of $s_\mu(x||a)$, we see that
	$\delta_i(\hs_\mu(y||a))=0$ unless $\la-i$ is a partition and $\la-i=\mu$, in which case $\la=\mu+i$ and $\delta_i(\hs_\mu(y||a))=\tilde{s}^a_i(\hs_{\mu+i}(y||a))$. In the latter case, applying $\tilde{s}^a_i$ to both sides and using \eqref{E:sdelta}, we obtain
	$\delta_i(\hs_\mu(y||a))=\hs_{\mu+i}(y||a)$ as required.
\end{proof}

It follows that the dual Schurs can be created by applying the homology divided difference operators to $1$.

\begin{thm} \label{T:create dual Schur} For any $\la\in\Par$, we have $	\hs_{\la}(y||a) = \delta_{w_\la}(1)$.
\end{thm}

\begin{example} \label{X:delta dual Schur} By Example \ref{X:hook dual Schur}
	$\hs_1(y||a) = \sum_{p,q\ge0} (-a_0)^q a_1^p s_{p+1,1^q}(y)$. We have
	\begin{align*}
	\Omega(a_1y/a_0y) &= \sum_{p\ge0} a_1^p h_p[y] \sum_{q\ge0} (-a_0)^q e_q[y] \\
	&= \sum_{p,q\ge0} (-a_0)^q a_1^p (s_{(p,1^q)}[y] + s_{(p+1,1^{q-1})}[y]) \\
	&= 1 + (a_1-a_0) \sum_{p,q\ge0} (-a_0)^p a_1^q s_{(p+1,1^q)}[y] \\
	\hs_1(y||a) &= (\alpha_0^a)^{-1} (1-\Omega(a_1y/a_0y)) \\
	&= \delta_0^a(1).
	\end{align*}
\end{example}
%

\begin{remark} This construction can also be adapted to compute the homology Schubert basis for the affine Grassmannian of $G=SL_{k+1}$, equivariant with respect to the maximal torus $T$ of $G$. The resulting basis is the $k$-double-Schur functions of \cite{LaSh2}. A $k$-double Schur function consists of a $k$-Schur function in its lowest degree and typically has infinitely many terms of higher degree with equivariant coefficients.
\end{remark}

\subsection{$\delta$-Schubert polynomials and $\delta$-Schur functions}
\label{SS:delta}
There is a $\Q$-algebra morphism $\eta_\delta: \Q[a] \to \Q[\delta]$ given by 
$$
\eta_\delta(a_i) = \begin{cases} \delta & \mbox{ if $ i > 0$,} \\
0 & \mbox{if $i \leq 0$.}
\end{cases}
$$
We have an induced $\Q$-algebra homomorphism $\eta_\delta: \Lambda(x||a) \to \Lambda[\delta] := \Lambda_\Q \otimes_\Q \Q[\delta]$ given by $\eta_\delta(\sum_{\la} a_\la p_\la(x||a)) =\sum_\la \eta_\delta(a_\la) p_\la$.  This extends to the $\Q[x]$-algebra homomorphism $\eta_\delta: \bR(x;a) \to \bR[\delta] := \bR \otimes_\Q \Q[\delta]$ that acts on $\Q[x]$ as the identity.

Define the \defn{$\delta$-Schubert polynomials} $\bS(x;\delta):= \eta_\delta(\bS(x;a))$ and the \defn{$\delta$-Schur functions} $s_\lambda(x||\delta):= \eta_\delta(s_\lambda(x||a))$.  They form bases over $\Q[\delta]$ for the rings $\bR[\delta]$ and $\Lambda[\delta]$ respectively.

\subsection{$\delta$-dual Schurs represent Knutson-Lederer classes}
Knutson and Lederer \cite{KL} define a ring $R^{H^S}$ that is a one-parameter deformation of symmetric functions $\Lambda$.  Namely, $R^{H^S}$ is a free $\Q[\delta]$-module with basis $[X^\lambda]$, $\lambda\in \Par$ \footnote{Knutson-Lederer work over $\Z[\delta]$, but for consistency with our current work we use $\Q[\delta]$.  Our results generalize to $\Z[\delta]$.}.  

The multiplication in $R^{H^S}$ is defined as follows.
Let $\bigoplus: \Gr(a,a+b) \times \Gr(c,c+d) \to \Gr(a+c,a+b+c+d)$ be the direct sum map $(V,W) \mapsto V\oplus W$.  Let the circle $S^1$ act on each $\C^{a+b}$ by acting with weight 1 on the first $b$ coordinates and weight $0$ on the last $a$ coordinates.  This induces an action of $S^1$ on $\Gr(a,a+b)$.
In $R^{H^S}$, we have 
$$
[X^\lambda] \cdot [X^\mu] = \sum_{\nu} d_{\lambda \mu}^\nu(\delta) [X^\nu]
$$
where the RHS is the class in $H^{S^1}_*(\Gr(a+c,a+b+c+d))$ of the direct sum $\bigoplus(X^\lambda,X^\mu)$ of two opposite Schubert varieties.  Here $a,b,c,d$ are chosen so that $\lambda \subseteq a \times b$ and $\mu \subseteq c \times d$.

Define $\hs_\lambda(y||\delta)$ by specializing $\hs_\lambda(y||a)$ via $a_i = 0$ for $i \leq 0$ and $a_i = \delta$ for $i >0$.  

\begin{proposition} \label{P:delta dual Schur} We have
	\begin{align} \label{E:delta dual Schur to Schur}
	\hs_\mu(y||\delta) &= \sum_{\substack{\la \supset \mu \\ d(\la)=d(\mu) \\ \text{$\la/\mu \subset $ first $d(\mu)$ rows}}} \Schub_{w_{\la/\mu}}(1)\,\delta^{|\la/\mu|}  s_\la(y).
	\end{align}
\end{proposition}
\begin{proof} This is an immediate consequence of Proposition
	\ref{prop:dualsuper}.
\end{proof}

\begin{example} Consider the product $\hs_{(1)}(y||\delta) \hs_{(1,1)}(y||\delta)$. To compare with \cite[Example 1.3]{KL} we restrict to the Grassmannian $\Gr(4,7)$, that is, only keep $s_\la(y)$ when $\la$ is contained in the $4\times 3$ rectangle. We have
	\begin{align*}
	\hs_{1}(y||\delta) &= s_1 + \delta s_2 + \delta^2 s_3 + \dotsm \\
	\hs_{11}(y||\delta) &= s_{11} + \delta s_{21} + \delta^2 s_{31} + \dotsm \\
	\hs_{1}(y||\delta) \hs_{11}(y||\delta) &=
	\delta^0(s_{111}+s_{21}) + \delta^1(2 s_{211}+s_{22} + 2 s_{31}) +
	\delta^2(s_{221} + 3 s_{311} + 2 s_{32}) \\
	&+ \delta^3(2 s_{321} + s_{33}) + \delta^4 s_{331} + \dotsm 
	\end{align*}
	We have
	\begin{align*}
	\hs_{1}(y||\delta) \hs_{11}(y||\delta) &= \delta^0(\hs_{111}(y||\delta)+\hs_{21}(y||\delta)) + 
	\delta^1(\hs_{211}(y||\delta)+\hs_{22}(y||\delta)) +
	\delta^2 \hs_{221}(y||\delta)
	\end{align*}
	where
	\begin{align*}
	\hs_{111}(y||\delta) &= s_{111} + \delta s_{211} + \delta^2 s_{311} + \dotsm \\
	\hs_{21}(y||\delta)  &= s_{21} + 2\delta s_{31} + \dotsm \\
	\hs_{211}(y||\delta) &= s_{211} +2\delta s_{311} +\dotsm \\
	\hs_{22}(y||\delta) &= s_{22} + 2\delta s_{32} + \delta^2 s_{33} + \dotsm \\
	\hs_{221}(y||a) &= s_{221} + 2 \delta s_{321} + \delta^2 s_{331}.
	\end{align*}
\end{example}

\begin{proposition}
	The set $\{\hs_\lambda(y||\delta) \mid \lambda \in \P\}$ is dual to the basis $\{s_\lambda(x||\delta) \mid \lambda \in \P\}$ of $\Lambda[\delta]$.
\end{proposition}

The following result answers a question implicitly posed in \cite{KL}.

\begin{thm}\label{thm:KL}
	There is an isomorphism of $\Q[\delta]$-algebras
	\begin{align*}
	R^{H^S} & \longrightarrow \Lambda[\delta] &
	[X^\lambda]  &\longmapsto \hs_\lambda(x||\delta).
	\end{align*}
\end{thm} 
\begin{proof}
	It suffices to show that the structure constants $d_{\lambda \mu}^\nu(\delta)$ of $R^{H^S}$ are obtained from the coproduct structure constants of the Hopf algebra $\Lambda(x||a)$ after specializing $a_i = 0$ for $i \leq 0$ and $a_i = \delta$ for $i >0$.  
	
	For simplicity, we assume that $a = b$ and $c = d$ in the following calculation.  Let us think of $\C^{a+c+c+a}$ as spanned by $e_{a+c}, e_{a+c-1},\ldots, e_1,e_0, e_{-1},\ldots, e_{1-a-c}$, with a natural action of $T=(\C^\times)^{a+c+c+a}$.  We identify $H^*_T(\pt) = \Q[a_{a+c},a_{a+c-1},\ldots,a_{1-a-c}]$.  We thus have actions of $T$ on $\Gr(c,c+c)$ (the $c$-dimensional subspaces of $\sp(e_c,e_{c-1},\ldots,e_{1-c})$) and on $\Gr(a,a+a)$ (the $a$-dimensional subspaces of $\sp(e_{a+c},e_{a+c-1},\ldots,e_{c+1},e_{-c},e_{-c-1},e_{1-a-c})$).  Finally, we have a $T$-action on $\Gr(a+c,a+c+c+a)$ and the direct sum map is $T$-equivariant, so we obtain a map of $H^*_T(\pt)$-modules
	\begin{equation} \label{eq:sum}
	H^T_*(\Gr(a+c,a+c+c+a)) \to H_T^*(\Gr(a,a+a)) \otimes H_T^*(\Gr(c,c+c)).
	\end{equation}
	Since a $T$-equivariant cohomology class of any of these Grassmannians is determined by its value at $T$-fixed points, the map \eqref{eq:sum} is completely determined by the direct sum map applied to $T$-fixed points.

	The $T$-fixed points of $\Gr(c,c+c)$ are then in bijection with pairs $(J_-,J_+)$ satisfying $J_- \subset [1-c,0]$, $J_+ \subset [1,c]$ and $|J_-|=|J_+|$, via the map 
	$$
	(J_-,J_+) \mapsto \sp(e_i \mid i \in ([1-c,0] \setminus J_-) \cup J_+) \in \Gr(c,c+c).
	$$
	For $\Gr(a,a+a)$ we consider $T$-fixed points as pairs $(K_-,K_+)$ with $K_- \subseteq [1-a-c,-c]$ and $K_+ \subseteq [c+1,a+c]$.  Then the direct sum map 
	induces the map $((K_-,K_+),(J_-,J_+)) \mapsto (J_- \cup K_-, J_+ \cup K_+)$.  By Proposition \ref{prop:locHopf}, this agrees with the coproduct of $\Lambda(x||a)$ in terms of localization.  (Note that in this work we do not give a geometric explanation of the coproduct of $\Lambda(x||a)$ similar to the direct sum map, which is \emph{not} equivariant with respect to the natural infinite-dimensional torus.)
	
	By Proposition \ref{prop:doubleSchur}, the double Schur functions $s_\lambda(x||a)$ can be identified with the opposite Schubert class $[X^\lambda]$ in equivariant {\it cohomology} $H^*_T(\Gr(a,a+b))$.  It follows that the structure constants of \eqref{eq:sum} with respect to the opposite Schubert classes $[X^\lambda]$ coincide with the coproduct structure constants \eqref{E:coproduct structure constants} of the double Schur functions.  Specializing $a_i = 0$ for $i \leq 0$ and $a_i = \delta$ for $i >0$ gives the desired conclusion. 
\end{proof}

\begin{remark}
	Knutson and Lederer \cite{KL} also define a $K$-theoretic analogue, and a result similar to Theorem \ref{thm:KL} holds.
\end{remark}

\subsection{Homology equivariant Monk's rule}
A \defn{vertical strip} is a skew shape that contains at most one box per row.  A \defn{horizontal strip} is a skew shape that contains at most one box per column. 
A \defn{ribbon} $R =\la/\mu$ is a (edgewise) connected skew shape not containing any $2 \times 2$ square.    A skew shape $\lambda/\mu$ is called \defn{thin} if its connected components are ribbons.  We write $c(\la/\mu)$ for the number of connected components of a thin skew shape.
\begin{lem}\label{lem:ribbon}
	Let $R =\la/\mu$ be a nonempty ribbon.  Then there exists exactly two shapes such that $\la/\rho$ is a vertical strip and $\rho/\mu$ is a horizontal strip.
\end{lem}
\begin{proof}
	The northeast most square of $R$ can belong to either $\la/\rho$ or $\rho/\mu$.  For all other boxes $b \in R$, either $R$ contains the square directly north of $b$ in which case $b \in \la/\rho$ or $R$ contains the square directly east of $b$ in which case $b \in \rho/\mu$.
\end{proof}

Suppose $\la/\mu$ is a skew shape.  A \defn{$\Lambda$-decomposition} of $\la/\mu$ is a pair $D = (\la/\rho, \rho/\mu)$ consisting of a vertical strip and a horizontal strip.  If $\la/\mu$ has a $\Lambda$-decomposition then it must be thin.  In this case, it follows from Lemma \ref{lem:ribbon} that $\la/\mu$ has exactly $2^{c(\la/\mu)}$ $\Lambda$-decompositions.  

The weight of a $\Lambda$-decomposition $D = (\la/\rho, \rho/\mu)$ is the product
\begin{equation}\label{eq:wtD}
\wt(D) := \prod_{(i,j) \in \la/\rho} (a_{j-i+1} - a_0) \prod_{(i,j) \in \rho/\mu} (a_1 - a_{j-i+1}) \in \Q[a]
\end{equation}
which can be 0.  If $D = (\la/\rho, \rho/\mu)$ is a $\Lambda$-decomposition, let $D_-$ be obtained from $D$ by removing the northeast most square of $\la/\mu$ from whichever of $\la/\rho$ or $\rho/\mu$ that contains it.  
%

\begin{thm}\label{thm:homologyMonk}
	Let $\mu \in \Par$.  We have
	\begin{align}\label{E:homologyMonk}
	\hs_1(y||a) \hs_\mu(y||a) =\sum_\lambda \sum_{D_-} \wt(D_-) \hs_\la(y||a)  
	\end{align}
	where the inner sum is over all distinct $D_-$ that can be obtained from some nonempty $\Lambda$-decomposition $D = (\lambda/\rho,\rho/\mu)$ with outer shape $\lambda$.
\end{thm}
The proof of Theorem \ref{thm:homologyMonk} will be given in \textsection \ref{ssec:proofhomologyMonk}.  In the non-equivariant case with $a_i = 0$, Theorem \ref{thm:homologyMonk} reduces to the usual one-box Pieri rule: when $\lambda/\mu$ is a single box, there are two possible choices of $D$, but $D_-$ will always be empty and $\wt(D_-) = 1$.
%

\begin{example} Let $\mu=(1)$. The support of $\hs_1 \hs_1$ is the set of partitions of size at least 2 not containing the boxes $(3,2)$ nor $(2,3)$. We have
	\begin{align*}
	\hs_1 \hs_1 &= \hs_2 + \hs_{11} + (a_1-a_2)\hs_3 + (a_1-a_0) \hs_{21} + (a_{-1}-a_0)\hs_{111} \\ &+ (a_1-a_2)(a_1-a_3)\hs_4 
	+ (a_1-a_0)(a_1-a_2) \hs_{31} + (a_1-a_0)^2 \hs_{22} \\ &+ (a_{-1}-a_0)(a_1-a_0)\hs_{211}+(a_{-1}-a_0)(a_{-2}-a_0)\hs_{1111} + \dotsm
	\end{align*}
	First consider $\la = (2,2)$.  Then the $\Lambda$-decompositions are given by taking $\rho = (1,1)$ or $(2,1)$.  In both cases, $D_- = ((2,2)/(2,1), (1,1)/(1,0))$.  Thus the coefficient of $\hs_{22}$ is $\wt(D_-) = (a_1-a_0)^2$.
	For $\hs_{211}$, the highest box $(1,2)$ in $(2,1,1)/(1)$ is ignored. The box $(2,1)$ can be in either the horizontal or vertical strip (contributing $a_1-a_0$ or $0$ respectively) while the box $(3,1)$ must be in the vertical strip, contributing $a_{-1}-a_0$, resulting in the coefficient $(a_1-a_0)(a_{-1}-a_0)$.
	
	Finally, let us consider $\la = (3,1)$.  The box $(1,3)$ is ignored.  The box $(1,2)$ must be in the horizontal strip of $D_-$ (giving weight $(a_1-a_2)$)  while the box $(2,1)$ can be in either the vertical or horizontal strip (giving weights $(a_0-a_0)$ and $(a_1-a_0)$ respectively).  The total contribution is $(a_1-a_2)(a_1-a_0)$.
\end{example}

\begin{example}\label{X:1times11} Let $\mu=(1,1)$. The support of $\hs_1 \hs_{11}$ consists of the partitions of size at least 3 which contain the partition $(1,1)$ and do not contain the boxes $(2,3)$ or $(4,2)$. We have
	\begin{align*}
	\hs_1 \hs_{11} &= \hs_{21} + \hs_{111} + (a_1-a_2)\hs_{31} + (a_1-a_0)\hs_{22} + (a_1-a_0)\hs_{211} + (a_{-2}-a_0)\hs_{1111} \\
	&+ (a_1-a_2)(a_1-a_3)\hs_{41}+(a_1-a_0)(a_1-a_2)\hs_{32} + (a_1-a_0)(a_1-a_2)\hs_{311} \\ &+ (a_1-a_0)^2 \hs_{221} + (a_{-2}-a_0)(a_1-a_0)\hs_{2111} + (a_{-2}-a_0)(a_{-3}-a_0)\hs_{11111}
	+ \dotsm
	\end{align*}
\end{example}

\subsection{Homology equivariant Pieri rule}
Let $\rho/\mu$ be a horizontal strip and $q \geq 0$ an integer.  A \defn{$q$-horizontal filling} of $\rho/\mu$ is a filling $T$ of $\rho/\mu$ with the numbers $1,2,\ldots,q+1$ so that the numbers are weakly increasing from left to right regardless of row, and every number from $2$ to $q+1$ is used.  (The number of such $T$ is equal to the number of semistandard Young tableaux for a single row of size $|\rho/\mu|-q$ using the numbers $1$ through $q+1$.)
Define the weight of a $q$-horizontal filling $T$ by
$$
\wt_q(T) := \prod_{(i,j) \in \rho/\mu}(a_{T(i,j)}-a_{j-i+1})
$$
where the product is over all boxes $(i,j)$ such that either $T(i,j)=1$ or $(i,j)$ is not the leftmost occurrence of $T(i,j)$ in $T$.  Thus $\wt(T) \in \Q[a]$ has degree equal to
$|\rho/\mu|-q$ (or is 0 if $q > |\rho/\mu|$).  Similarly, a \defn{$p$-vertical filling} of $\la/\rho$ is a filling $T$ of a vertical strip $\la/\rho$ with integers $0,-1,\ldots,-p$ so that the numbers are weakly decreasing from top to bottom regardless of column, and every number from $-1$ to $-p$ is used.  Define the weight of a $p$-vertical filling $T'$ by
$$
\wt_p(T') := \prod_{(i,j) \in \la/\rho} (a_{j-i+1}-a_{T(i,j)})
$$
where the product is over all boxes $(i,j)$ such that either $T(i,j)=0$ or $(i,j)$ is not the topmost occurrence of $T(i,j)$ in $T$.
%

A $(p,q)$-filling of a $\Lambda$-decomposition $(\la/\rho, \rho/\mu)$ is a pair $(T',T)$ consisting of a $p$-vertical filling $T'$ of $\la/\rho$ and a $q$-horizontal filling $T$ of $\rho/\mu$.  The $(p,q)$-weight of a $\Lambda$-decomposition $D=(\la/\rho, \rho/\mu)$ to be
\begin{align}\label{E:decomp weight}
\wt_{p,q}(D) := \sum_{(T',T)} \wt_p(T') \wt_q(T)
\end{align}
summed over all $(p,q)$-fillings $(T',T)$ of $D$.  We note that $\wt_{0,0}(D)$ is the weight $\wt(D)$ from \eqref{eq:wtD}.  Also note that if $p >|\la/\rho|$ or $q > |\rho/\mu|$ then $\wt_{p,q}(D) = 0$.

The following result gives a rule for multiplication by a hook-shaped dual Schur function.
\begin{thm}\label{thm:homologyhook}
	Let $\mu \in \Par$ and $p,q \geq 0$.  We have
	\begin{align}\label{E:homologyhook}
	\hs_{(q+1,1^{p})}(y||a) \hs_\mu(y||a) =  \sum_\la \sum_{D_-} \wt_{p,q}(D_-) \hs_\la(y||a)
	\end{align}
	where the inner sum is over all distinct $D_-$ that can be obtained from some nonempty $\Lambda$-decomposition $D = (\lambda/\rho,\rho/\mu)$ with outer shape $\lambda$.
	%
\end{thm}

The proof of Theorem \ref{thm:homologyhook} will be given in \textsection \ref{ssec:proofhomologyhook}. 

\begin{remark} Suppose we forget equivariance by setting $a_i=0$ for all $i$. Let $D$ be a nonempty $\La$-decomposition with outer shape $\la$ appearing in \eqref{E:homologyhook}. Then $\hat{c}^{\la}_{\mu,(q+1,1^p)}(0)$ is the Littlewood-Richardson coefficient, the coefficient of $s_\la$ in the product $s_\mu s_{(q+1,1^p)}$. The latter is the number of standard tableaux $S$ of shape $\la/\mu$ such that $1,2,\dotsc,q+1$ go strictly east and weakly north, and the numbers $q+1,q+2,\dotsc,q+p+1$ go strictly south and weakly west \cite{RW}. By Theorem \ref{thm:homologyhook}, in order to contribute to the sum, $\wt_{p,q}(D_-)$ must be degree $0$. This restricts the sum over $(p,q)$-fillings $(T',T)$ of $D_-$ such that $\wt_p(T')=1=\wt_q(T)$. For each $D_-$ there is a unique filling: in $T'$ the numbers
	$-1,-2,\dotsc,-p$ are used once each and go strictly south and weakly west, while in $T$ the numbers $2,3,\dotsc,q+1$ are used once each and go strictly east and weakly north. These $(T',T)$ biject with the above standard tableaux $S$: $q+1$ appears in the northeastmost box of $D$ in $S$, the numbers $2$ through $q+1$ in $T$ are replaced in $S$ by the numbers $1$ through $q$, and the numbers
	$-1,-2,\dotsc,-p$ in $T'$ are replaced by $q+2,q+3,\dotsc,q+p+1$. Thus the nonequivariant specialization of Theorem \ref{thm:homologyhook} agrees with the Littlewood-Richardson rule.  
\end{remark}

\begin{remark} By Proposition \ref{P:coproduct constants are Stanley coefficients}, $\hat{c}^\la_{\mu\nu}(a) = j^{w_{\la/\mu}}_\nu(a)$.
	Theorem \ref{thm:homologyhook} expresses these polynomials  positively in the sense of Theorem \ref{thm:doubleStanleypositivity} when one of $\mu$ or $\nu$ is a hook.  This should be compared with \cite[\textsection 4]{M} in which a combinatorial formula is given for all $\hat{c}^\la_{\mu\nu}(a)$. This formula does not exhibit the positivity of Theorem \ref{thm:doubleStanleypositivity}.
\end{remark}

\begin{example} Let us compute $\hs_{11} \hs_1$ with $\mu=(1)$, $p=1$ and $q=0$. The answer is given in Example \ref{X:1times11}.  
	
	First, consider $\la = (2,2)$.  Then the $\Lambda$-decompositions are given by taking $\rho = (1,1)$ or $(2,1)$.  In both cases, $D_- = ((2,2)/(2,1), (1,1)/(1,0))$.  There is a single $0$-horizontal filling of $(1,1)/(1,0)$: the box is filled with the number 1.  There is a single $1$-vertical filling of $(2,2)/(2,1)$: the box is filled with the number $-1$.  Thus $\wt_{1,0}(D_-) = (a_1 - a_0)$ which is the coefficient of $\hs_{22}$.
	
	Next, consider $\la = (2,1,1)$.  We have two possibilities for $D_-$: (a) $D_- = ((1,1,1)/(1),\emptyset)$ or (b) $D_- = ((1,1,1)/(1,1), (1,1)/(1))$.  For (a), there are two $1$-vertical fillings: both boxes are labeled $-1$ contributing $\wt_1(T') = (a_{-1}-a_{-1}) = 0$, or one box is labeled $0$ and the other $-1$ contributing $\wt_1(T') = (a_{0}-a_{0}) = 0$.  For (b), there are unique $0$-horizontal and $1$-vertical fillings, giving $\wt_{1,0}(D_-) = (a_1-a_0) $.  So the coefficient of $\hs_{211}$ is $a_1-a_0$.
	
	Finally, let us consider $\la = (3,1)$.  The box $(1,3)$ is ignored.  The box $(1,2)$ must be in the horizontal strip of $D_-$ while the box $(2,1)$ must be in the vertical strip of $D_-$.  There is a unique filling with $(1,0)$-weight $(a_1-a_2)$ which is the coefficient of $\hs_{31}$.  
\end{example}

\section{Peterson subalgebra}
In the affine setting, Peterson constructed a commutative subalgebra $\tP \subset \ttA$ (recalled in Appendix~\ref{A:affinenilHecke}) of the level-zero affine nilHecke algebra $\ttA$, and showed that the torus-equivariant homology $H_*^T(\hGr_n)$ of the affine Grassmannian $\hGr_n$ is isomorphic to $\tP$  \cite{Pet, Lam}.  The Peterson subalgebra $\tP$ is a nilHecke counterpart to the large commutative subgroup $\Z^{n-1} \subset \tilde{S_n}$ sitting inside the affine symmetric group. 

The infinite symmetric group $S_\Z$ does not contain an analogous lattice as a subgroup.  Nevertheless, in this section, we construct a subalgebra $\P' \subset \A'$ that is an analogue of Peterson's subalgebra for the (completed) infinite nilHecke algebra $\A'$.   We show that there is an isomorphism $\P' \cong \hLa(y||a)$ of $\Q[a]$-Hopf algebras and identify the element $j_\lambda \in \P'$ that is mapped to the dual Schur function $\hs_\lambda(y||a)$ under this isomorphism.

\subsection{Affine symmetric group}\label{ssec:tSn}
The affine symmetric group $\tS_n$ is the infinite Coxeter group with generators $s_0,s_1,\ldots, s_{n-1}$ and relations $s_i s_j = s_j s_i$ for $|i-j|\geq 2$ and $s_i s_{i+1}s_i = s_{i+1} s_i s_{i+1}$ for all $i$.  Here indices are taken modulo $n$.

We have an isomorphism $\tS_n \cong S_n \rtimes Q^\vee$, where $Q^\vee :=\{\lambda = (\lambda_1,\ldots,\lambda_n) \mid \sum_{i=1}^n \lambda_i =0\} \subset \Z^n$ is the coroot lattice spanned by the simple coroots $\alpha_i^\vee = e_i-e_{i+1}$ for $1\le i\le n-1$.  For $\lambda \in Q^\vee$, we write $t_\lambda \in \tS_n$ for the corresponding \defn{translation element}.  Then 
\begin{equation}\label{eq:affinetrans}
t_\lambda t_\mu = t_{\lambda+\mu} = t_\mu t_\la
\end{equation}
and $w t_\lambda w^{-1} = t_{w \cdot \lambda}$.

Let $\tilde{S}_n^0$ be the set of $0$-Grassmannian elements, i.e. those $w\in \tS_n$ such that $ws_i>w$ for all $i\ne 0$.
Each coset $w S_n$ for $S_n$ inside $\tS_n$ contains a unique translation element $t^w$, and a unique $0$-Grassmannian element.  
Suppose $w\in \tS_n^0$ and $t^w=t_\mu$ for $\mu\in Q^\vee$. Let $u_\mu\in S_n$ be the shortest element such that $u_\mu(\mu)$ is antidominant. Then $t^w=t_\mu \doteq w u_\mu^{-1}$.


\subsection{Translation elements}\label{SS:translation}
Unlike the affine symmetric group, the infinite symmetric group $S_\Z$ does not contain translation elements.  Nevertheless, it is possible to define elements $\tau^w $ in the infinite nilHecke algebra which behave like translation elements.  Recall that in \textsection \ref{ssec:nilHecke} we have defined the nilHecke algebra $\A$, which has a $\Q[a]$-basis $A_w$, $w \in S_\Z$.  Let $\A'$ denote the completion of $\A$, consisting of formal $\Q[a]$-linear combinations of the elements $A_w$.  For a given $w \in S_\Z$, there are only finitely many pairs $(u,v) \in S_\Z \times S_\Z$ such that $w \doteq uv$.  It follows that the multiplication in $\A$ induces a natural $\Q[a]$-algebra structure on $\A'$.

Recall also that we defined a comultiplication map $\Delta:\A \to \A \otimes_{\Q[a]} \A$.  Under the pairing \eqref{eq:PsiApairing}, $\A'$ is dual to $\Psi$.  It follows from Proposition~\ref{prop:bAbR} that $\Delta$ extends to a comultiplication $\A' \to \A' \otimes_{\Q[a]} \A'$.

Let $[a,b]\subset \Z$ be an interval. For $n\gg0$ there is an injective homomorphism
$S_{[a,b]} \to \tS_n$ defined by $s_i \mapsto s_{i\mod n}$ for $a\le i<b$.
Thus any $w\in S_\Z$ can be viewed as an element of $\tS_n$ for sufficiently large $n$.
 
\begin{lemma}\label{lem:stabletranslation}
Let $w \in S_\Z^0$.  There is a positive integer $m$ and a word $\a$ in the symbols
\begin{align}\label{E:stable generators}
\{s_{-m},s_{1-m},\ldots,s_{-1},s_0,s_1,\ldots,s_{m-1}\} \cup \{r, r'\}
\end{align}
such that for any sufficiently large $n \gg m$, a reduced word $\ta$ for $t^w$ (treating $w$ as an element of $\tS_n$) is obtained from $\a$ by the substitutions
\begin{equation}
\label{eq:subs}
r \mapsto s_m s_{m+1} \dotsm s_{-m-1} \qquad \text{ and } \qquad r' \mapsto s_{-m-1}  \dotsm s_{m+1} s_m.
\end{equation}
\end{lemma}

To explain how to find the above word, let $Q^\vee_{\Z} \subset \bigoplus_{i\in \Z} \Z e_i$ be the infinite coroot lattice, the sublattice spanned by $\alpha_i^\vee = e_i-e_{i+1}$ for $i\in\Z$.
Given $\la\in\Par$ let $\beta=\beta_\la \in Q^\vee_{\Z}$ be the element
\begin{align*}
	\beta_\la = \sum_{i\in I_{w_\la,+}} e_i - \sum_{i\in I_{w_\la,-}} e_i
\end{align*}
(see \eqref{E:wla}, \eqref{E:I+}, \eqref{E:I-}). There is a projection $Q^\vee_{\Z} \to Q^\vee$ denoted $\beta \mapsto \overline{\beta}$, onto the translation lattice $Q^\vee$ in $\tilde{S}_n$, given by $e_i \mapsto e_{i+n\Z}$. We have
\begin{align*}
	t^{w_\la} &= t_{\overline{\beta}_\la} \qquad\text{for $\la\in\Par$.}
\end{align*}
Let $m$ be large enough so that $\la$ is contained in the $m\times m$ square partition and $n \ge 2m$. Then $|I_{w_\la,\pm}| \le m$
and all coordinates in $\bb$ are in $\{-1,0,1\}$ with coordinates $1$ (resp. $-1$) occurring only in the first (resp. last) $m$ positions.

To prove Lemma \ref{lem:stabletranslation} it suffices to show that
the element $u_{\bb}^{-1}$ is in the image of a product of the generators \eqref{E:stable generators} under the substitution \eqref{eq:subs}. Since images of $r$ and $r'$ are inverses we may replace $u_{\bb}^{-1}$ by $u_{\bb}$. It is enough to be able to sort
$\bb$ to antidominant using the generators. This is explained by the following example.

\begin{example} Let $\la=(4,4,3,1)$. Let us take $m=4$ and $n=11$.	
We have $I_+ =\{1,3,4\}$ and $I_-=\{-3,-1,0 \}$ because the vertical (resp. horizontal) line segments tracing the edge of $\la$ above (resp. below) the main diagonal, occur at segments $1,3,4$ (resp. $-3,-1,0$)
where the main diagonal separates segment $0$ and $1$.  This is illustrated in Figure~\ref{fig:I+-}.
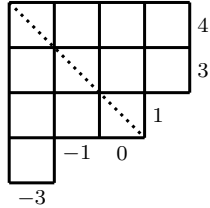
\begin{figure}
\begin{center}
\begin{tikzpicture}[scale=0.6,line width=0.8mm]
\draw[line width=0.4mm] (4,4)--(0,4); 
\draw[line width=0.4mm] (4,3)--(0,3);
\draw[line width=0.4mm] (4,2)--(0,2);
\draw[line width=0.4mm] (3,1)--(0,1);
\draw[line width=0.4mm] (1,0)--(0,0);
\draw[line width=0.4mm] (0,0)--(0,4);
\draw[line width=0.4mm] (1,0)--(1,4);
\draw[line width=0.4mm] (2,1)--(2,4);
\draw[line width=0.4mm] (3,1)--(3,4);
\draw[line width=0.4mm] (4,2)--(4,4);
\draw[line width=0.4mm, dotted] (0,4)--(3,1);

\draw (0.5,-0.35) node {$\scriptstyle{-3}$};
\draw (1.5, 0.65) node {$\scriptstyle{-1}$}; 
\draw (2.5, 0.65) node {$\scriptstyle{0}$};
\draw (3.30, 1.5) node {$\scriptstyle{1}$};
\draw (4.30, 2.5) node {$\scriptstyle{3}$};
\draw (4.30, 3.5) node {$\scriptstyle{4}$};
\end{tikzpicture}
\end{center}
\caption{For $\lambda = (4,4,3,1)$, we have $I_+ = \{1,3,4\}$ and $I_-=\{-3,-1,0\}$.}
\label{fig:I+-}
\end{figure}
We have $\bb = (1,0,1,1|0,0,0|-1,0,-1,-1)$ where the positions of the $1$s (resp. $-1$s) are given by $I_+$ (resp. $I_-$) mod $n$. The vertical bars separate the first $m=4$ positions and the last $m$ positions. Between are zeroes. Recalling that indices of reflections are identified modulo $n$, the generators are $s_7,s_8,s_9,s_{10}$, $s_0,s_1,s_2,s_3$ and $r\mapsto s_6s_5s_4$ and $r'\mapsto s_4s_5s_6$.
We must move $\bb$ to the antidominant chamber with a shortest element in $S_n$ using the given generators. Starting with $\bb$ we may apply
$r' s_7  r$ to get $(1,0,1,-1|0,0,0|1,0,-1,-1)$, then apply simple generators to reach $(-1,0,1,1|0,0,0|-1,-1,0,1)$, then apply $r' s_7 r$ to get $(-1,0,1,-1|0,0,0|1,-1,0,1)$, simple generators to reach
$(-1,-1,0,1|0,0,0|-1,0,1,1)$, $r' s_7 r$ to reach
$(-1,-1,0,-1|0,0,0|1,0,1,1)$, and simple generators to reach
$(-1,-1,-1,0|0,0,0|0,1,1,1)$ which is antidominant.
\end{example}

For $b\le 0 < a$, define
\begin{align}
\label{E:rab}
r_{a,b} &:= (\prod_{i=a}^\infty s_i)( \prod_{i=-\infty}^{b-1} s_i)  =  (s_a s_{a+1} s_{a+2} \cdots ) ( \cdots s_{b-2} s_{b-1}) \\
\label{E:rba}
r_{b,a} &:= (\prod_{i=b-1}^{-\infty} s_i)( \prod_{i=\infty}^{a} s_i)  =  ( s_{b-1} s_{b-2} \cdots ) ( \cdots s_{a+2} s_{a+1} s_a).
\end{align}
Each of these are infinite words in the alphabet $\{ s_i \mid i \in \Z \setminus \{0\}\}$, and each is a concatenation of two \defn{infinite reduced words}.  Abusing notation, we will use the same symbols $r_{a,b}$ and $r_{b,a}$ to represent  the following permutations of $\Z$ (that do not belong to $S_\Z$):
\begin{equation}\label{eq:rab}
r_{a,b}(i) = \begin{cases} i & \mbox{if $b < i < a$,} \\
i+1 &\mbox{if $i \geq a$ or $i < b$,} \\
a & \mbox{if $i = b$;}
\end{cases}\qquad
r_{b,a}(i) = \begin{cases} i & \mbox{if $b < i < a$,} \\
i-1 &\mbox{if $i > a$ or $i \leq b$,} \\
b & \mbox{if $i = a$.}
\end{cases}
\end{equation}

Let $\cS$ denote the set of infinite words in the alphabet $\{s_i \mid i \in \Z \setminus \{0\}\}$ obtained as a finite concatenation of the words $s_i$, $i \in \Z \setminus \{0\}$ and the words $r_{a,b}$, $a>0$ and $b \leq 0$.  Suppose $\a \in \cS$, and $s$ is a letter in $\a$.  Then we have a unique factorization
$
\a = \a' \, s \, \a''
$
where again $\a', \a'' \in \cS$.  We define a root $\beta(s)$ by
$$
\beta(s) := - \a' \cdot (a_i - a_{i+1})
$$
if the letter $s$ is equal to $s_i$.  Here the action of $\a'$ on $\Q[a]$ is the one induced by the action on $\Z$ given by \eqref{eq:rab}.

\begin{definition}\label{def:translation}
Let $w \in S_\Z^0$.  Define the \defn{infinite translation element} $\tau^w \in \A'$ as follows.  Take the word $\a$ of Lemma \ref{lem:stabletranslation} and replace each occurrence of $r$ or $r'$ by infinite words as follows:
$$
r \mapsto r_{m,-m} \qquad \text{ and } \qquad r' \mapsto r_{-m,m}
$$
to obtain an infinite word $\a^w_\infty \in \cS$.  Now for $v \in S_\Z$, define $\xi^v|_{\tau^w} \in \Q[a]$ (cf. Proposition \ref{prop:Billey}) by
$$
\xi^v|_{\tau^w} := \sum_{\b \subset \a^w_\infty} \prod_{s \in \b} \beta(s)
$$
summed over finite subwords $\b$ of $\a^w_\infty$ that are reduced words for $v$, and define $\tau^w \in \A'$ by
\begin{equation}\label{eq:tau}
\tau^w := \sum_{v \in S_\Z} \xi^v|_{\tau^w} A_w.
\end{equation}
\end{definition}

\begin{remark}
Suppose $w \in S_\Z^0$ and we have $I_{w,+} = \{1 \leq d_t < d_{t-1} < \cdots < d_1\}$ and $I_{w,-} = \{e_1 < e_2 < \cdots < e_t \leq 0\}$.  Then a possible choice of $\a^w_\infty$ is:
$$
\u (\prod_{j=t}^1 r_{j,e_j }) (\prod_{j=t}^1 r_{1-j-f_j,d_j}) 
$$
where $\u$ is a reduced word for $w$ and $f_j = |I_{w,-} \cap (-j,0]|$ for $j = 1,2,\ldots,t$.  Note that if $w$ is the identity element, then $\tau^w = 1$.
\end{remark}

\begin{remark}
In Definition \ref{def:translation} we have used Lemma \ref{lem:stabletranslation} which relies on the notion of translation elements in the affine symmetric group.  In future work 
we plan to study the Schubert calculus of a flag ind-variety associated to the affine infinite symmetric group $Q^\vee_\Z \rtimes S_{\Z}$, which contains translation elements $\tau^w$ as defined above.
\end{remark}

\begin{proposition}\label{prop:translations} The elements $\tau^x$ satisfy the following properties.
\begin{enumerate}
\item For $x \in S_\Z^0$, we have that $\tau^x$ is a well-defined element of $\A'$ that does not depend on the choices of $m$ and $\a$ in Lemma \ref{lem:stabletranslation}.
\item The set $\{\tau^x \mid x \in S_\Z^0\}$ is linearly independent in $\A'$.
\item If $z = xy$ under the partial product of \textsection \ref{ssec:partialproduct}, then $\tau^z = \tau^x \tau^y = \tau^y \tau^x$.
\item We have $\tau^x \tau^y = \tau^y \tau^x$ for any $x,y \in S_\Z^0$.
\item We have $\tau^x p = p \tau^x$ for any $x \in S_\Z^0$ and any $p \in \Q[a]$.
\item We have $\Delta(\tau^x)=\tau^x \otimes \tau^x$ for any $x \in S_\Z^0$.
\end{enumerate}
\end{proposition}
Proposition \ref{prop:translations} is proven in \textsection \ref{ssec:translationproof}.

\subsection{The Peterson subalgebra}

Let $\bigoplus_w \Q(a) \tau^w$ denote the $\Q(a)$-vector subspace of $\Q(a)\otimes_{\Q[a]} \A'$  spanned by the elements $\tau^w$.  Define the $\Q[a]$-submodule $\P \subset \A'$ by
$$
\P:= \A' \cap \bigoplus_w \Q(a) \tau^w.
$$
By Proposition \ref{prop:translations}(4), $\P$ lies within the centralizer subalgebra $Z_{\A'}(\Q[a])$.

Recall that $j_\lambda^w(a)$ denotes the coefficient of the double Schur function $s_\lambda(x||a)$ in the double Stanley symmetric function $F_w(x||a)$.
For $\lambda \in \Par$, define 
$$
j_\lambda = \sum_w  j_\lambda^w(a) A_w \in \A'.
$$

\begin{thm}\label{thm:jbasis}
For any $\lambda \in \Par$, we have $j_\lambda \in \P$, and it is the unique element of $\P$ satisfying
\begin{equation}\label{eq:j}
j_\lambda = A_{w_\lambda} + \sum_{u \notin S_\Z^0} a_u A_u
\end{equation}
where $a_u \in \Q[a]$ and the summation is allowed to be infinite.  The submodule $\P$ is a free $\Q[a]$-module with basis $\{j_\lambda \mid \lambda \in \Par\}$.
\end{thm}
Theorem \ref{thm:jbasis} will be proved in \textsection \ref{ssec:jbasisproof}.  Let $\P'$ be the completion of $\P$ whose elements are formal $\Q[a]$-linear combinations of the elements $\{j_\lambda \mid \lambda \in \Par\}$.  We call $\P'$ the \defn{Peterson subalgebra}.

\begin{thm}\label{thm:PetersonHopf}
The submodule $\P' \subset \A'$ is a commutative and cocommutative Hopf algebra over $\Q[a]$.
\end{thm}

\begin{conj} We have $\P' = Z_{\A'}(\Q[a])$.
\end{conj}

\begin{thm}\label{thm:PetersonSym}
There is an isomorphism $\P' \cong \hLa(y||a)$ of $\Q[a]$-Hopf algebras sending $j_\lambda$ to $\hs_\lambda(y||a)$ for all $\lambda \in \Par$.
\end{thm}
Theorems \ref{thm:PetersonHopf} and \ref{thm:PetersonSym} will be proved in \textsection \ref{ssec:PetersonHopfproof}. 

\begin{remark}
Theorems \ref{thm:jbasis}, \ref{thm:PetersonHopf}, and \ref{thm:PetersonSym} hold over $\Z$, but for consistency we work over $\Q$.
\end{remark}

%
%

\subsection{Fomin-Stanley algebra}
Let $A$ denote the (infinite) nilCoxeter algebra, which is the $\Q$-algebra with generators $A_i$, $i \in \Z$, satisfying the relations
\eqref{E:dd2}, \eqref{E:dd commute}, and \eqref{E:dd braid}.
The algebra $A$ has $\Q$-basis $A_w$, $w \in S_\Z$.  Let $A'$ denote the completion of $A$ consisting of elements $a = \sum_w a_w A_w$ that are infinite $\Q$-linear combinations of the $A_w$.  Since every $w \in S_\Z$ has finitely many factorizations of the form $w \doteq xy$, it follows that $A'$ is a $\Q$-algebra.  There is a natural map $\phi_0: \A \to A$ given by 
$$
\phi_0(\sum_w a_w A_w) = \sum_w \phi_0(a_w) A_w
$$
where $\phi_0(a_w) \in \Q$ is the constant term of the polynomial $a_w \in \Q[a]$.

Define the \emph{Fomin-Stanley subalgebra} $B \subset A$ as the image $\phi_0(\P)$.  Let $j_\lambda^0:= \phi_0(j_\lambda)$.

\begin{thm}\label{thm:j0basis}
The set $\{j^0_\lambda \mid \lambda \in \Par\}$ form a $\Q$-basis of $B$.  There is a Hopf-isomorphism $B \to \Lambda$ given by $j^0_\lambda \mapsto s_\lambda$.
\end{thm}
\begin{proof}
Since $\{j_\lambda \mid \lambda \in \Par\}$ form a basis of $\P$, it is clear that $\{j^0_\lambda \mid \lambda \in \Par\}$ spans $B$.  The equation \eqref{eq:j} shows that $j^0_\lambda = A_{w_\lambda} + \text{other terms}$ are linearly independent. The last statement follows from Theorem \ref{thm:PetersonHopf} and \ref{thm:PetersonSym}.
\end{proof}

\subsection{Stability of affine double Edelman-Greene coefficients}
\label{ssec:stability}

Let $n \geq 2$. For $v\in \tS_n$ let $\xi^v_{\hFl_n} \in H^*_{T_n}(\hFl_n)$ be the torus equivariant Schubert class class of the affine flag ind-variety $\hFl_n$ (see \textsection \ref{sec:affineflagcohom}).

Following \cite{LaSh}\footnote{Our $\xi^v_{\hFl_n}|_w$ differs from the one in \cite{LaSh} by a sign $(-1)^{\ell(v)}$.}, define $\tS^0_n \times \tS^0_n$ matrices $\tA$ and $\tB$ by
$$
\tA_{vw} =  \xi^v_{\hFl_n}|_w \qquad \text{and}\qquad \tB = \tA^{-1}.
$$
Both matrices $\tA$ and $\tB$ are lower-triangular when the rows and columns are ordered compatibly with the Bruhat order on $\tS_n^0$, and the entries belong to $\Q(a_1,a_2,\ldots,a_n)$.  For $x \in \tS_n$ and $v \in \tS_n^0$, denote by $\tj^x_v \in \Q[a_1,a_2,\ldots,a_n]$ the affine double Edelman-Greene coefficient, and let $\tj_v \in \ttA$ denote the $j$-basis element (see Appendix \ref{A:affinenilHecke}).  

\begin{proposition}[\cite{LS}] \label{prop:LS}  Let $v,w \in \tS_n^0$ and $x \in \tS_n$.
We have \begin{align}
t^w &= \sum_{v \in \tS_n^0, \;\; v \leq w} \tA_{vw} \tj_v \label{eq:LS1}\\
\tj_v &= \sum_{w \in \tS_n^0, \;\; w \leq v} \tB_{wv} t^w \label{eq:LS2}\\
\tj^x_v &= \sum_{w \in \tS_n^0, \;\; w \leq v}\tB_{wv} \xi^x_{\hFl_n}(t^w) \label{eq:LS3}.
\end{align}
\end{proposition}
%

Let $\ev_n:\Q[a] \to \Q[a_1,a_2,\ldots,a_n]$ denote the $\Q$-algebra morphism given by $a_i \mapsto a_{i \mod n}$.

\begin{lemma}\label{lem:xistable}
Let $w,v \in S_\Z$.  Then for sufficiently large $n \gg 0$, we have $\xi^v_{\hFl_n}(w) = \ev_n( \xi^v(w))$.
\end{lemma}
\begin{proof}
This follows from Proposition \ref{prop:Billey} which also holds in the affine case as well as the infinite case.
\end{proof}

\begin{lemma}\label{lem:AB}
Let $w,v \in S_\Z^0$.  Then there exists polynomials $A_{vw}(a), B_{vw}(a) \in \Q[a]$ such that for all $n \gg 0$, we have $\tA_{vw} = \ev_n(A_{vw})$ and $\tB_{vw} = \ev_n(B_{vw})$.
\end{lemma}
\begin{proof}
Follows immediately from Lemma \ref{lem:xistable}.
\end{proof}

\begin{lemma}
Let $x \in S_\Z$ and $v \in S_\Z^0$.  There exists a polynomial $q(a) \in \Q[a]$ such that for all $n \gg 0$, we have that $\tj^x_v = \ev_n(q)$. 
\end{lemma}
\begin{proof}
Using Lemma \ref{lem:stabletranslation} and Lemma \ref{lem:xistable}, we deduce that for any $u \in S_\Z$ and $w \in S_\Z^0$, there is a polynomial $p(a) \in \Q[a]$ such that for sufficiently large $n$, we have $\xi^u_{\hFl_n}(w) = \ev_n(p(a))$.  By Proposition \ref{prop:LS}, we conclude that there is a polynomial $q(a) \in \Q[a]$ such that $\tj^x_v = \ev_n(q(a))$ for sufficiently large $n$.  
\end{proof}

\subsection{Proof of Proposition \ref{prop:translations}} \label{ssec:translationproof}
Let $x \in S_\Z^0$ and $v \in S_\Z$.  Only finitely many subwords of $a^x_\infty$ are reduced words for $v$, and for $n \gg 0$ there is a bijection between such subwords and subwords of $\ta$ that are reduced words for $v$ (now thought of as an element in $\tS_n$).  It thus follows from the definitions that for $n \gg 0$ we have
\begin{equation}\label{eq:taustabilize}
\ev_n(\xi^v(\tau^x)) = \xi^v_{\hFl_n}(t^x).
\end{equation}
Claim (1) follows immediately.  Claim (2) follows from the similar claim in the affine nilHecke algebra $\ttA$.

Let $x,y \in S_Z^0$ and $v \in S_\Z$.  Only finitely many pairs of terms from the expansion \eqref{eq:tau} for $\tau^x$ and $\tau^y$ contribute to the coefficient of $A_v$ in the product $\tau^x \tau^y$.  Thus for $n \gg 0$ the coefficient of $A_v$ in $\tau^x \tau^y$ is taken to the coefficient of $A_v$ in $t^x t^y$ by $\ev_n$.  Claims (3) and (4) now follow from similar statements in the affine case (see \eqref{eq:affinetrans}).

Let $x \in S_Z^0$, $v \in S_\Z$, and $p \in \Q[a]$.  Only finitely many terms of the expansion \eqref{eq:tau} for $\tau^x$ contribute to the coefficient of $A_v$ in $\tau^x p$.  Thus for $n \gg 0$ the coefficient of $A_v$ in $\tau^x p$ is taken to the coefficient of $A_v$ in $t^x \ev_n(p)$ by $\ev_n$.  Claim (5) now follows from \eqref{eq:levelzero} in the affine case.

Let $v \in S_\Z$.  Then for $n \gg 0$, the calculation of $\Delta(A_v)$ in the affine nilHecke ring $\ttA$ is identical to that in $\A$.  Claim (6) now follows from the equality $\Delta(t^x) = t^x \otimes t^x$ in the affine case (see \eqref{eq:affinecoprod}).

\subsection{Proof of Theorem \ref{thm:jbasis}} \label{ssec:jbasisproof}
\begin{proposition}\label{prop:jstabilize}
Let $x \in S_\Z$ and $v \in S_\Z^0$.  For all $n \gg 0$, we have $\tj^x_v = \ev_n(j^x_v)$. 
\end{proposition}
\begin{proof}
By Theorem \ref{thm:small}, the image of $\bS_x(x;a)$ in $H^*_T(\hFl_n)$ represents $\xi^x_{\hFl_n}$ for sufficiently large $n$.  By Proposition \ref{prop:wrongwaycommute}, the double Stanley function $F_x(x||a)$ represents the affine double Stanley class $\varpi( \xi^x_{\hFl_n}) \in H^*_T(\hGr_n)$ for sufficiently large $n$.  By Theorem \ref{thm:small}, the image of $s_{v_\lambda}(x||a)$ in $H^*_T(\hGr_n)$ represents $\xi^v_{\hGr_n}$ for sufficiently large $n$.  We conclude that $\ev_n (j^x_v) = \tj^x_v$.
\end{proof}

By Proposition \ref{prop:jstabilize}, the element $j_\lambda \in \A'$ is the limit (taking limits of coefficients of $A_v$) of $\tj_{w_\lambda} \in \ttA$ as $ n \to \infty$.  By \eqref{eq:taustabilize}, the element $\tau^w \in \A'$ is a similar limit of the elements $t^w \in \ttA$.  Combining Lemma \ref{lem:AB} and \eqref{eq:LS2}, we thus conclude that
$$
j_\lambda = \sum_{\substack{v \in \tS_n^0 \\ v \leq w_\lambda}} B_{v w_\lambda} \tau^v.
$$
It follows that $j_\lambda \in \P$.  The expansion \eqref{eq:j} follows from Theorem \ref{thm:aPeterson}.  

By Lemma \ref{lem:AB} and \eqref{eq:LS1}, we have that both $\{j_\lambda \mid \lambda \in \Par\}$ and $\{\tau^w \mid w \in S_\Z^0\}$ form bases of $\Q(a) \otimes_{\Q[a]} \P$.  Thus an arbitrary element of $a = \sum_{a_w} A_w \in \P$ is uniquely determined by the coefficients $\{a_{w} \in \Q[a] \mid w \in S_\Z^0\}$.  Indeed, we have $a = \sum_{\lambda \in \Par} a_{w_\lambda} j_\lambda$ and the sum must be finite.  It follows that $\P$ is a free $\Q[a]$-module with basis $\{j_\lambda \mid \lambda \in \Par\}$.

\subsection{Proof of Theorems \ref{thm:PetersonHopf} and \ref{thm:PetersonSym}} \label{ssec:PetersonHopfproof}

\begin{prop}\label{prop:jprod}
For $\lambda,\nu \in \Par$, we have
\begin{equation}\label{eq:jprod}
j_\lambda j_\mu = \sum_{\nu \supset \mu}  j_\lambda^{w_{\nu/\mu}} \, j_\nu.
\end{equation}
\end{prop}
\begin{proof}
Let us calculate the coefficient of $s_\lambda(x||a) \otimes s_\mu(x||a)$ in $\Delta(F_w(x||a))$.  On the one hand, 
$$\Delta(F_w(x||a)) = \sum_{\nu \in \Par} j_\nu^w \Delta(s_\nu(x||a)) 
= \sum_{\nu \in \Par} j_\nu^w \sum_{\mu \subset \nu } F_{w_{\nu/\mu}} (x||a) \otimes s_\mu(x||a)
$$
by Corollary \ref{cor:coproddoubleSchur}.  So the coefficient is equal to $\sum_{\nu\supset\mu} j_\lambda^{w_{\nu/\mu}} j_\nu^w$, which is the coefficient of $A_w$ on the RHS of \eqref{eq:jprod}.

On the other hand, by Corollary \ref{cor:coproddoubleStanley}, we have $$\Delta(F_w(x||a)) = \sum_{w \doteq uv} F_u(x||a) \otimes F_v(x||a)
= \sum_{\lambda,\mu} \sum_{w \doteq uv} j_\lambda^u j_\mu^v (s_\lambda(x||a) \otimes s_\mu(x||a)).
$$  
So the coefficient is also equal to $\sum_{w \doteq uv} j_\lambda^u j_\mu^v$, which (using Proposition \ref{prop:translations}(5) to obtain that $j_\lambda$ commutes with $\Q[a]$) is equal to the coefficient of $A_w$ on the LHS of \eqref{eq:jprod}.
\end{proof}

It follows from Proposition \ref{prop:jprod} that $\P'$ is a commutative $\Q[a]$-algebra.  Together with Proposition \ref{prop:translations}(6), we obtain Theorem \ref{thm:PetersonHopf}.

The pairing \eqref{eq:PsiApairing} induces a pairing between $\P'$ and $\Psi_\Gr$.  By Proposition \ref{prop:PsiA}(3), we have $\pair{\xi^v}{ j_\lambda} = \delta_{vw_\lambda}$ for $v \in S_\Z^0$ and $\lambda \in \Par$.  Thus $\P'$ and $\Psi_\Gr$ are dual $\Q[a]$-modules.  By Proposition \ref{prop:PsiA}(2), the comultiplication in $\P'$ is dual to the multiplication in $\Psi_\Gr$.  By comparing Proposition \ref{prop:jprod} and Corollary \ref{cor:coproddoubleSchur}, the multiplication of $\P'$ is dual to the multiplication in $\Psi_\Gr$.  Thus $\P'$ and $\Psi_\Gr$ are dual $\Q[a]$-Hopf algebras.  By Proposition \ref{prop:locHopf} and the definition of $s_\lambda(y||a)$, we have an induced isomorphism of $\Q[a]$-Hopf algebras $\P' \cong \hLa(y||a)$ sending $j_\lambda$ to $s_\lambda(y||a)$.  This completes the proof of Theorem \ref{thm:PetersonSym}.

\subsection{Proof of Theorem \ref{thm:doubleStanleypositivity}}
\label{ssec:positivityproof}
The polynomial $j^x_v(a) \in \Q[a]$ belongs to a subring of the form $\Q[a_{1-m},a_{2-m},\ldots,a_m]$ for some $m$.  Suppose $n \gg m$.  Then $\tj^x_v(a) \in \Q[a_1,\ldots,a_n]$, and by Proposition \ref{prop:jstabilize}, it is the image of $j^x_v(a)$ under $\ev_n$.  Pick a cutoff $c$ satisfying $m \ll c \ll n-m$.  Make the substitution 
$$
a_i \mapsto \begin{cases} a_i & \mbox{if $1 \leq i \leq c$,} \\
a_{i-n} & \mbox{if $c < i \leq n$,}
\end{cases}
$$
to $\tj^x_v(a)$.  The resulting polynomial must equal $j^x_v(a)$.

By Theorem \ref{thm:aPositive}, we have that $\tj^x_v(a)$ is a positive integer polynomial expression in the linear forms 
$$
a_1 - a_2, a_2-a_3, \ldots, a_{n-1}-a_n.
$$
Applying the above substitution to this expression gives the desired expression for $j^x_v(a)$.

\section{Back stable triple Schubert polynomials}\label{sec:triple}
In this section we define triple back stable Schubert polynomials and triple Stanley symmetric functions.  This allows effective computation of some double Edelman-Greene coefficients and structure constants for dual Schur functions. Before we provide the precise definition, we present some motivation.

Corollary \ref{C:shift stanley} states that $F_{\gamma(w)}(x)=F_{w}(x)$ where $F_w$ is the Stanley symmetric function. However, the same statement is not true for double Stanley symmetric functions. 
\begin{example} \label{example:s1s0}  Recall superization notation from \eqref{E:super map}.
\begin{align*}
F_{s_1s_0}(x||a)&=\bS_{s_1s_0}(x;a)=h_2(x_{\leq 0} / a_{\leq 1}) = h_2(x_{\leq 0}/a_{\leq 0})-a_1h_1(x_{\leq 0} / a_{\leq 0}).\\
F_{s_2s_1}(x||a)&= \eta_a (\bS_{s_2s_1}(x;a))=\eta_a(h_2(x_{\leq 1} / a_{\leq 2})) \\
&= \eta_a(h_2(x_{\leq 0}/a_{\leq 0})+(x_1-a_1-a_2) h_1(x_{\leq 0} / a_{\leq 0})+(x_1-a_1)(x_1-a_2))\\
&= h_2(x_{\leq 0}/a_{\leq 0})-a_2h_1(x_{\leq 0} / a_{\leq 0}).
\end{align*}
\end{example}
Note that the only difference between $F_{s_1s_0}(x||a)$ and $F_{s_2s_1}(x||a)$ is the coefficient in front of the term $h_1(x_{\leq 0} /a_{\leq 0})$, and if we compute $F_{s_3s_2}(x||a)
$, this coefficient becomes $a_3$. In general, when we shift $w$ by $\gamma$, certain variables $a_i$ remain the same and other variables $a_j$ become $a_{j+1}$. Roughly speaking, triple Stanley symmetric functions separate stable $a_i$ and shifted $a_j$ when applying $\gamma$ to $w$, by replacing stable variables $a_i$ by $b_i$. To make the construction formal, we start by defining back stable triple Schubert polynomials.

\subsection{Tripling}
Let $\nu_{a,b}: \Lambda(a) \to \Lambda(b)$ be the map that changes symmetric functions from the $a$-variables to $b$-variables.
Let $\La(x/b) \subset \La(x) \otimes_{\Q} \La(b)$ denote the image of the superization map $p_k\mapsto p_k(x/b)$.
We use the same notation for the $\Q[a]$-algebra maps
\begin{align*}
\nu_{a,b}: \bR(a) = \Lambda(a) \otimes_\Q \Q[a] &\to \Lambda(b) \otimes_\Q \Q[a] &
p_k(a) \otimes 1 & \mapsto p_k(b) \otimes 1 \\
\nu_{a,b}: \Lambda(x||a) &\to \Lambda(x/b) \otimes_\Q \Q[a] &
p_k(x/a) & \mapsto p_k(x/b)
\end{align*}
and the $\Q[x,a]$-algebra map
\begin{align*}
\nu_{a,b}: \bR(x;a) = \Lambda(x||a) \otimes_{\Q[a]} \Q[x,a] &\to \Lambda(x/b) \otimes_\Q \Q[x,a] &
p_k(x/a) & \mapsto p_k(x/b).
\end{align*}
These maps change $a$'s to $b$'s but only ``in symmetric functions''.  All of these maps are $\Q[a]$-algebra isomorphisms: the inverse is the substitution $f \mapsto f|_{b=a}$.  Finally, note that we have an injection $\Lambda(x||a) \hookrightarrow \Lambda(x) \otimes \Lambda(a) \otimes \Q[a]$, and the action of $\nu_{a,b}$ on $\Lambda(x||a)$ is simply given by $1 \otimes \nu_{a,b} \otimes 1: \Lambda(x) \otimes \Lambda(a) \otimes \Q[a] \to \Lambda(x) \otimes \Lambda(b) \otimes \Q[a]$.

\subsection{Back stable triple Schubert polynomials}
For $w \in S_\Z$, define the \defn{back stable triple Schubert polynomials} $\bS_w(x;a;b) \in \Lambda(x/b) \otimes_\Q \Q[x,a]$ by
$$
\bS_w(x;a;b) := \nu_{a,b}(\bS_w(x;a)).
$$
The set $\{\bS_w(x;a;b)\mid w \in S_\Z\}$ form a basis of $\Lambda(x/b) \otimes_\Q \Q[x,a]$ over $\Q[a]$.  In particular, the structure constants for $\bS_w(x;a;b)$ (which are equal to the structure constants for $\bS_w(x;a)$) belong in $\Q[a]$.
\begin{example}\label{example:s1s0v2} 
Continuing Example \ref{example:s1s0}, we have
\begin{align*}
\bS_{s_1s_0}(x;a;b)&= h_2(x_{\leq 0}/b_{\leq 0})-a_1 h_1(x_{\leq 0} / b_{\leq 0}) \\
\bS_{s_2s_1}(x;a;b)&= h_2(x_{\leq 0}/b_{\leq 0})+(x_1-a_1-a_2) h_1(x_{\leq 0} / b_{\leq 0})+(x_1-a_1)(x_1-a_2)).
\end{align*}
\end{example}
\begin{prop}\label{prop:triple}
Let $w \in S_\Z$.  We have
\begin{align*}
\bS_w(x;a;b) &= \sum_{\substack{w \doteq uvz \\ u,z \in S_{\neq 0}}}(-1)^{\ell(u)} \S_{u^{-1}}(a) F_v(x/b)\S_z(x) 
= \sum_{w \doteq uv} (-1)^{\ell(u)}\nu_{a,b}(\bS_{u^{-1}}(a)) \bS_v(x).
\end{align*}
\end{prop}
\begin{proof}
The first equality follows from applying $\nu_{a,b}$ to \eqref{E:backstable triple def}.  The second equality follows from applying $\nu_{a,b}$ to Proposition \ref{P:backstable double well-defined}.
\end{proof}
%

Recall that $A_i^x$ (resp. $A_i^a$, $A_i^b$) denotes the divided difference operator in the $x$-variables (resp. $a$-variables, $b$-variables).

\begin{prop}
For $w \in S_\Z$ and $i \in \Z$, we have
$$	
A_i^x \bS_w(x;a;b) = \begin{cases} \bS_{w s_i}(x;a;b) & \text{if $ws_i < w$} \\
	0 & \text{otherwise.}
	\end{cases}
$$	
For $w \in S_\Z$ and $i \in \Z-\{0\}$, we have
$$	
A_i^a \bS_w(x;a;b) = \begin{cases} - \bS_{s_i w}(x;a;b) & \text{if $s_iw < w$} \\
	0 & \text{otherwise.}
	\end{cases}
$$	
\end{prop}
\begin{proof}
The first statement follows immediately from the last equality in Proposition \ref{prop:triple} and Theorem \ref{T:backstable}.  For $i \neq 0$, we have $A_i^a \circ \nu_{a,b} = \nu_{a,b} \circ A_i^a$, so the second statement follows by Proposition \ref{P:left ddiff}.
\end{proof}

\subsection{Triple Stanley symmetric functions}

Define the \defn{triple Stanley symmetric functions} by 
$$
F_w(x||a||b) := \nu_{a,b}(F_w(x||a)).
$$
\begin{example}\label{example:s1s0v3} 
Continuing Example \ref{example:s1s0v2}, we have
\begin{align*}
F_{s_1s_0}(x||a||b)&= h_2(x_{\leq 0}/b_{\leq 0})-a_1 h_1(x_{\leq 0} / b_{\leq 0}) \,\, \text{and} \,\,
F_{s_2s_1}(x||a||b)= h_2(x_{\leq 0}/b_{\leq 0})-a_2 h_1(x_{\leq 0} / b_{\leq 0}).
\end{align*}
\end{example}
By Proposition \ref{prop:double Stanley triple sum} and Theorem \ref{thm:doublecoprod}, we have
\begin{align*}
F_w(x||a||b) &= \sum_{\substack{w \doteq uvz \\ u,z \in S_{\neq 0}}}(-1)^{\ell(u)} \S_{u^{-1}}(a) F_v(x/b)\S_z(a) \\
\bS_w(x;a;b) &= \sum_{\substack{w \doteq uv\\ v \in S_{\neq 0}}} F_u(x||a||b) \S_v(x;a).
\end{align*}
It follows from \eqref{E:plethystic difference stanley} that $F_w(x||a||b)$ satisfies the supersymmetry (cf. \cite[(2.15)]{M})
$$
F_{w^{-1}}(x||a||b) = (-1)^{\ell(w)} F_w(b||a||x).
$$

\begin{lemma}\label{lem:tripStanley}Let $w \in S_\Z$.  Then
$$F_w(x||a||b)=\sum_{w \doteq uvz}(-1)^{\ell(u)} \bS_{u^{-1}}(a) F_v(x/b)  \bS_{z}(a).$$
\end{lemma}
\begin{proof}
We have 
\begin{align*}
&\sum_{w \doteq uvz}(-1)^{\ell(u)} \bS_{u^{-1}}(a) F_v(x/b)  \bS_{z}(a) \\
&=\sum_{\substack{w \doteq u_1u_2vz_1z_2 \\ u_1,z_2 \in S_{\neq 0}}}(-1)^{\ell(u_1)+\ell(u_2)} \S_{u_1^{-1}}(a) F_{u_2^{-1}}(a) F_v(x/b) F_{z_1}(a) \S_{z_2}(a) \\
&= \sum_{\substack{w \doteq u_1x z_2 \\ u_1,z_2 \in S_{\neq 0}}}(-1)^{\ell(u_1)} \S_{u_1^{-1}}(a) \left(\sum_{x \doteq u_2vz_2} (-1)^{\ell(u_2)} F_{u_2^{-1}}(a) F_v(x/b) F_{z_1}(a) \right)\S_{z_2}(a) \\
&= \sum_{\substack{w \doteq u_1x z_2 \\ u_1,z_2 \in S_{\neq 0}}}(-1)^{\ell(u_1)} \S_{u_1^{-1}}(a) F_x(x/b) \S_{z_2}(a) = F_w(x||a||b)
\end{align*}
using Theorem \ref{thm:coprod} and \eqref{E:plethystic difference stanley}.
\end{proof}

Define the \defn{triple Schur functions} (essentially the same as the \defn{supersymmetric Schur functions} of Molev \cite[Section 2.4]{M}) by
$s_\lambda(x||a||b) := \nu_{a,b}  ( s_\lambda(x||a))$.
Then
\begin{equation}\label{eq:3=2}
F_w(x||a||b) = \sum_\lambda j^w_\lambda(a) s_\lambda(x||a||b)
\end{equation}
where $j^w_\lambda(a)$ are the usual double Edelman-Greene coefficients.
The \defn{triple Edelman-Greene coefficients} are defined by
$$
F_w(x||a||b) = \sum_\lambda j^w_\lambda(a,b) s_\lambda(x||b)
$$
and satisfy $\deg(j^w_\lambda(a,b)) = \ell(w) - |\lambda|$.
It is clear that $j^w_\lambda(a,a) = j^w_\lambda(a)$, but by \eqref{eq:3=2}, we also have
\begin{equation}\label{eq:jj}
j^w_\lambda(a,b) = \sum_\mu j^w_\mu(a) j^{w_\mu}_\lambda(a,b).
\end{equation}

Recall the $\Q$-algebra automorphism $\shift_a$ of \textsection \ref{SS:double Schur}. This map
can be applied to the $\Q[a]$-algebra $\Lambda(x/b) \otimes_\Q \Q[x,a]$ or to the $\Q[a]$-algebra $\Q[a,b]$. Recall also $\shift:S_\Z\to S_\Z$ from \textsection \ref{SS:notation}.

\begin{prop}
For $w \in S_\Z$, we have $F_{\shift(w)}(x||a||b) = \shift_a(F_w(x||a||b))$.
\end{prop}
\begin{proof}
Follows from Lemma \ref{lem:tripStanley} and Corollaries \ref{C:backstable double shift} and \ref{C:shift stanley}.
\end{proof}

\begin{cor}\label{cor:tripleshift}
Let $\la \in \Par$ and $w \in S_\Z$.  Then $j_\lambda^{\shift(w)}(a,b) =\shift_a(j_\la^w(a,b))$.
\end{cor}
Thus triple Stanley symmetric functions allow us to distinguish between ``stable'' phenomena (the $b$-variables) and the ``shifted'' phenomena (the $a$-variables).

\begin{example}\label{example:s1s0v4} 
Continuing Example \ref{example:s1s0v3}, we have
\begin{align*}
F_{s_1s_0}(x;a;b)&= h_2(x_{\leq 0}/b_{\leq 0})-a_1 h_1(x_{\leq 0} / b_{\leq 0})\\
&=h_2(x_{\leq 0}/b_{\leq 1})+(b_1-a_1) h_1(x_{\leq 0} / b_{\leq 0})\\
&= h_{2}(x||b) - (b_1-a_1) h_1(x||b). \\
F_{s_2s_1}(x;a;b)&= h_2(x_{\leq 0}/b_{\leq 0})-a_2 h_1(x_{\leq 0} / b_{\leq 0})\\
&= h_{2}(x||b) - (b_1-a_2) h_1(x||b). \\
\end{align*}
Therefore, $j_{(1)}^{s_1s_0}(a,b)=b_1-a_1$ and $j_{(1)}^{s_2s_1}(a,b)=b_1-a_2$.
\end{example}

%

\subsection{Double to triple}
We have an explicit formula for $j_\la^w(a,b)$ in terms of double Edelman-Greene coefficients $j_\la^w(a)$.  Recall the definition of Durfee square $d(\la)$ from before Proposition \ref{prop:doublesuper}.

\begin{prop}\label{prop:doubletriple}
Let $\la,\mu \in \Par$.  Then
$$
j_\mu^{w_\la}(a,b) = \sum_{\substack{\rho: \; \mu\subseteq\rho\subseteq\la \\ d(\mu)=d(\rho)=d(\la)}} (-1)^{|\la/\rho|} \S_{w_{\la/\rho}^{-1}}(a)\S_{w_{\rho/\mu}}(b).
$$
For $\mu \in \Par$ and $w \in S_\Z$ we have
\begin{equation}\label{e:doubletriple}
j^w_\mu(a,b) = \sum_{\substack{\la,\rho: \;\; \la\supset\rho\supset\mu \\ d(\la)=d(\rho)=d(\mu)}} (-1)^{|\lambda/\rho|} j^w_\lambda(a)  \S_{w_{\la/\rho}^{-1}}(a)\S_{w_{\rho/\mu}}(b)
\end{equation}
\end{prop}
\begin{proof}
By Proposition \ref{prop:doublesuper}, we have
\begin{align*}
s_\lambda(x||a||b) &=
\sum_{\substack{\rho\subset\la \\ d(\rho)=d(\la)}} (-1)^{|\lambda/\rho|} \S_{w_{\la/\rho}^{-1}}(a) s_\rho(x/b) 
=
\sum_{\substack{\mu\subset\rho\subset\la \\ d(\mu)=d(\rho)=d(\la)}} (-1)^{|\lambda/\rho|} \S_{w_{\la/\rho}^{-1}}(a) \S_{w_{\rho/\mu}}(b) s_\mu(x||b)
\end{align*}
This gives the first formula.  The second formula follows from \eqref{eq:jj}.
\end{proof}

The following result follows from \eqref{e:doubletriple}.
\begin{prop}\label{prop:partialb}
Let $w \in \Par$, $w \in S_\Z$, and $i \in \Z-\{0\}$.  Then
$$
A_i^b j^w_\mu(a,b) = \begin{cases} j^w_{\mu+i}(a,b) & \mbox{if $\mu$ has an addable box on diagonal $i$,} \\
0 & \mbox{if $\mu$ has no addable box on diagonal $i$.}
\end{cases}
$$
\end{prop}

\subsection{Triple Edelman-Greene coefficients for a hook}
In this section we compute $j_{(q+1,1^p)}^w(a,b)$ for all $w \in S_\Z$ and $p,q\ge0$, in a way that exhibits the positivity of Theorem \ref{thm:doubleStanleypositivity}.

The support of a permutation $w \in S_\Z$ is the finite set of integers
$$
|w| :=\{i \mid s_i \text{ appears in a reduced word of } w\} \subset \Z.
$$
%
A permutation $w \in S_\Z$ is called \defn{increasing} (resp. \defn{decreasing}) if it has a reduced word $s_{i_1} s_{i_2} \cdots s_{i_\ell}$ such that $i_1 < i_2 < \cdots < i_\ell$ (resp. $i_1 > i_2 > \cdots > i_\ell$).  For $J \subset \Z$ a finite set, we denote by $u_J \in S_\Z$ (resp. $d_J \in S_\Z$) the unique increasing (resp. decreasing) permutation with support $J$.

We call $w \in S_\Z$ a $\Lambda$ if it has a factorization of the form $w \doteq u_J d_K$.  Such factorizations are called $\Lambda$-factorizations. We consider two factorizations to be distinct if their pairs $(J,K)$ are distinct.
We call a reduced word $\u$ a $\Lambda$-word if it is first increasing then decreasing.  Associated to a $\Lambda$-factorization is a unique $\Lambda$-reduced word.

%

Suppose $w$ admits a nontrivial $\La$-factorization $\id\ne w\doteq u_J d_K$. Let $m=\max |w|$, $J' = J \setminus \{m\}$ and $K' = K \setminus \{m\}$. There are exactly two pairs $(J,K)$ corresponding to a given pair $(J',K')$: $m$ occurs in exactly one of $J$ and $K$.

For a finite set $T = \{t_1,t_2,\ldots,t_r\} \subset \Z$, let $a_T$ denote the sequence of variables $(a_{t_1},a_{t_2},\ldots,a_{t_r})$.
For the above $T$ let $T+1=\{t_1+1,\dotsc,t_r+1\}$.


\begin{thm}\label{thm:j hook} Let $p,q\ge0$ and $w\in S_\Z$.  Then $j_\la^w(a,b)=0$ unless $w$ is a $\La$, in which case
\begin{align}
\label{E:j hook no s_0}
j_{(q+1,1^p)}^w(a,b) &= \sum_{\substack{ (J',K') \text{ distinct} \\ w \doteq u_J d_K} } \S_{s_{|K'|}\dotsm s_{q+1}}(b;a_{K'+1}) \S_{s_{-|J'|}\dotsm s_{-1-p}}(b; a_{J'+1})
\end{align}
where the sum runs over all distinct pairs $(J',K')$ coming from $\La$-factorizations $w = u_J d_K$ and 
$\S_{v_{\pm}}(b;a_{J'+1})$ is the image of $\S_{v_{\pm}}(b;a_{\pm})$ under the substitution $a_{\pm}\mapsto a_{J'+1}$ where $v_{\pm}\in S_{\pm}$.
\end{thm}

\begin{remark} The coefficients $j^w_\la(a,b)$ appear to satisfy the following generalization of the positivity in Theorem \ref{thm:doubleStanleypositivity}: $j^w_{\la}(a,b)$ is a sum of products of binomials $c-d$ where $c$ and $d$ are variables with $c \prec d$ where 
\begin{align*}
	b_1 \prec b_2 \prec \dotsm 
	\prec  a_1 \prec a_2 \prec 
	\dotsm \prec a_{-2} \prec a_{-1} \prec a_0 \prec
 \dotsm \prec b_{-2} \prec b_{-1} \prec b_0.
\end{align*}
The double Schubert polynomials occurring in Theorem \ref{thm:j hook} satisfy this positivity, say, by the formula for the monomial expansion of double Schubert polynomials in \cite{FK}.
\end{remark}

\begin{remark} It is possible to obtain more efficient formulas than those in Theorem \ref{thm:j hook}, especially when $p=q=0$, by grouping terms according to the set of maxima for each of the maximal subintervals of $|w|$.
\end{remark}

\begin{example} Let $w=s_1s_2$. For $j_1^w$ there is a single summand $(J',K')=(\{1\},\emptyset)$ corresponding to either of the factorizations
$(J,K)=(\{1,2\},\emptyset)$ or $(J,K)=(\{1\},\{2\})$. Then
$j_1^{s_1s_2}(a,b) = a_2-b_0$.
More generally, for $w = s_i s_{i+1} \cdots s_{k}$, we have $j_1^w(a,b) = (a_k - b_0)(a_{k-1} - b_0) \cdots (a_{i+1}-b_0)$.

\end{example}

Let $\theta = a_1 -a_0$. The proof of Theorem \ref{thm:j hook} uses localization formulas for Schubert classes in equivariant cohomology $H_{T_n}^*(\hFl_n)$ (see \textsection \ref{sec:affineflagcohom}) of the affine flag variety. In this context we set $a_i=a_{i+n}$ for all $i\in\Z$.
We shall use the following result \cite[Theorem 6]{LaSh}. 

\begin{thm}\label{thm:LS} For every $\id \neq x \in S_n$, we have $\theta^{-1} \xi^{x^{-1}}_{\hFl_n}|_{s_\theta} \in \Q[a_1,a_2,\ldots,a_n]$ and
$$\tj_{s_0}= A_{s_0} + \sum_{\id\neq x \in S_n} \left((-1)^{\ell(x)}\theta^{-1} \xi^{x^{-1}}_{\hFl_n}|_{s_\theta} A_x + (-1)^{\ell(x)} \xi^{x^{-1}}_{\hFl_n}|_{s_\theta} A_{s_0 x}\right).$$
\end{thm}

%

\begin{lem}\label{lem:theta} Let $\id\ne x\in S_n$. Then $\xi^x_{\hFl_n}|_{s_\theta}=0$ unless $x$ is a $\La$, in which case
\begin{align}\label{E:loc box theta}
	(-1)^{\ell(x)}\xi^x_{\hFl_n}|_{s_{\theta}} =(a_1-a_0)
	\sum_{\substack{(J',K') \text{ distinct}\\ x \doteq u_J d_K}} \S_{s_{|J'|}\dotsm s_1}(a;a_{J'+1})
	\S_{s_{-|K'|} \dotsm s_{-1}}(a;a_{K'+1}).
\end{align}
\end{lem}
\begin{proof}
We compute $\xi^x_{\hFl_n}|_{s_\theta}$ as an element of $\Q[a_0,a_1,\ldots,a_{n-1}]$ (setting $a_n = a_0$), using Proposition \ref{prop:Billey}, picking the reduced word $\u = s_1 s_2 \cdots s_{n-1} \cdots s_2 s_1$ of $s_\theta$.  If $x$ has no $\Lambda$-factorization, then $\u$ does not contain a reduced word for $x$.  For $i \neq n-1$, the roots $\beta(s_i)$ associated to $s_i$ are $a_{i+1}-a_1$ (left occurrence) and $a_0 - a_{i+1}$ (right occurrence), the sum of which is $a_0 -a_1$.  We also have $\beta(s_{n-1}) = a_0-a_1$. 

Summing over the $\Lambda$-factorizations, the simple generator $s_m$ where $m=\max(|w|)$ contributes a factor of $(a_1-a_0)$ to $(-1)^{\ell(x)}\xi^x|_{s_\theta}$. The remaining simple generators contribute $\prod_{j\in J'}(a_1-a_{j+1}) \prod_{k \in K'}(a_{k+1}-a_0)$. Finally, these products of binomials are double Schubert polynomials:
\begin{align*}
\prod_{j\in J'}(a_1-a_{j+1}) = \Schub_{s_{|J'|}\dotsm s_1}(a;a_{J'}+1), \qquad \prod_{k\in K'} (a_{k+1}-a_0) = \Schub_{s_{-|K'|}\dotsm s_{-1}}(a;a_{K'+1}). \qquad \qedhere
\end{align*}
\end{proof}

%

\begin{proof}[Proof of Theorem \ref{thm:j hook}]
First suppose that $w \in S_+$.
Combining Theorem \ref{thm:LS} and Lemma \ref{lem:theta} with the limiting arguments of \textsection \ref{ssec:stability}, we deduce (noting that Theorem \ref{thm:LS} has ``$x^{-1}$'') that 
\begin{align*}
	j_1^w(a) = \sum_{\substack{(J',K') \text{ distinct}\\ w \doteq u_J d_K}} \S_{s_{|K'|}\dotsm s_1}(a;a_{K'+1})
	\S_{s_{-|J'|} \dotsm s_{-1}}(a;a_{J'+1}).
\end{align*}

Recall the shift automorphism $\shift: S_\Z \to S_\Z$ from \textsection \ref{SS:notation}. 
It follows that we must have
\begin{align*}
j_1^w(a,b) = \sum_{\substack{(J',K') \text{ distinct}\\ w \doteq u_J d_K}} \S_{s_{|K'|}\dotsm s_1}(b;a_{K'+1})
\S_{s_{-|J'|} \dotsm s_{-1}}(b;a_{J'+1}).
\end{align*}
to be consistent with Corollary \ref{cor:tripleshift}, and this must hold for all $w \in S_\Z$. The formula for a general hook $(q+1,1^p)$ follows by Proposition \ref{prop:partialb}.
\end{proof}

\subsection{Proof of Theorem \ref{thm:homologyMonk}} \label{ssec:proofhomologyMonk}
By Theorem \ref{thm:PetersonSym} and Proposition \ref{prop:jprod}, the coefficient of $\hs_\la(y||a)$ in the product $\hs_1(y||a)\hs_\mu(y||a)$ is equal to $0$ if $\mu \not \subseteq \la$ and equal to $j_1^{w_{\la/\mu}}(a)$ otherwise.

\begin{lem}
Let $\mu \subseteq \la$ and $z = w_{\la/\mu}$ be a $\Lambda$.  Then $\la/\mu$ is a thin skew shape.
\end{lem}

\subsection{Proof of Theorem \ref{thm:homologyhook}}
\label{ssec:proofhomologyhook}
The product $\hs_{(q+1,1^p)}(y||a) \hs_\mu(y||a)$ is computed by evaluating $j_{(q+1,1^p)}^{z}(a)$ where $z$ is 321-avoiding.  Thus Theorem \ref{thm:homologyhook} is obtained from Theorem \ref{thm:j hook}.
Let $D = (\la/\rho,\rho/\mu)$ be a $\Lambda$-decomposition.  The key equality is
$$
A^b_{u_{[-p,-1]} d_{[1,q]}} \wt(D) = \wt_{p,q}(D).
$$
This in turn follows from the equality
$$
\wt_{p,q}(D) = \S_{s_{|\rho/\mu|}s_{|\rho/\mu|-1} \cdots s_{q+1}}(a;a_{J+1}) \S_{s_{-|\la/\rho|}\ldots s_{-p-1}}(a;a_{K+1})
$$
where $J$ is the set of diagonals in $\rho/\mu$ and $K$ is the set of diagonals in $\la/\mu$.

\section{Affine flag variety}

In this section, we recall the equivariant cohomologies of affine flag varieties and affine Grassmannians.  We preview some results in affine Schubert calculus that will be developed in subsequent work \cite{LLS:affine}.  We use notation for affine symmetric groups as in \textsection \ref{ssec:tSn}.

\subsection{Affine flag variety and affine Grassmannian}

We consider affine flag varieties $\hFl^\cdot_{n}$ and affine Grassmannians $\hGr^\cdot_n$ of $\GL_n(\C)$.  A \defn{lattice} $L$ in $F^n$ is a free $\C[[t]]$-submodule satisfying $L \otimes_{\C[[t]]} F  = F^n$.  There is a map $\zeta: F^n \to F$ sending $t^k e_i$ to $t^{kn+i}$, compatible with infinite linear combinations.  Under $\zeta$, a lattice $L \subset F^n$ is sent to an admissible subspace $\Lambda \subset F$.  We often identify a lattice $L$ with the corresponding admissible subspace $\Lambda = \zeta(L)$.

The \defn{affine Grassmannian} $\hGr^\cdot_n$ consists of all lattices in $F^n$.  It embeds inside the Sato Grassmannian $\Gr^\cdot$, and thus inherits the structure of an ind-variety over $\C$.  We have $\hGr^\cdot_n = \bigsqcup_k \hGr_n^{(k)}$, where $\hGr_n^{(k)}:= \Gr^{(k)} \cap \hGr_n$.  The neutral component $\hGr_n:=\hGr_n^{(0)}$ is the affine Grassmannian of $\SL_n(\C)$.

An \defn{affine flag} in $F^n$ is a sequence
$$
L_\bullet = \cdots \subset L_{-1} \subset L_0 \subset L_1 \subset \cdots
$$
of lattices $L_i \subset F^n$, such that $\dim L_i/L_{i-1} = 1$ for all $i$ and $L_{i-n} = t L_i$.  The \defn{affine flag variety} $\hFl^\cdot_n$ consists of all affine flags in $F^n$.  We have $\hFl^\cdot_n = \bigsqcup_k \hFl_n^{(k)}$ where $L_\bullet \in \hFl_n^{(k)}$ if $L_0 \in \hGr^{(k)}_n$.  The neutral component $\hFl_n:=\hFl_n^{(0)}$ is the affine flag variety of $\SL_n(\C)$. 

The image $\Lambda_\bullet = \zeta(L_\bullet)$ is a flag of admissible subspaces in $F$.  However, it is not an admissible flag since it is possible that $\zeta(L_i) \neq E_i$ for infinitely many $i \in \Z$.  We do not have an embedding of $\hFl^\cdot_n$ in the Sato flag variety $\Fl^\cdot$.    Nevertheless, $\hFl^\cdot_n$ is known to be an ind-variety over $\C$ \cite{Kum}.

\subsection{Equivariant cohomology of affine flag variety}
\label{sec:affineflagcohom}
Let $T_n$ be the maximal torus of $\GL_{n}(\C)$.  We have $H^*_{T_n}(\pt) \cong \Q[a_1,\ldots,a_n]$.  Write $\shift_a:\Q[a] \to \Q[a]$ for the $\Q$-algebra isomorphism given by $a_i \mapsto a_{i+1 \mod n}$.

The torus $T_n$ acts on $\hGr_n$ and $\hFl_n$.  Let $\tS_n$ be the affine Coxeter group of $\SL_n(\C)$ and $\cS_n = \Z \ltimes \tS_n = \langle \sh \rangle \times \tS_n$ the affine Weyl group of $\GL_n(\C)$.  For $w \in \cS_n$, let $\xi^w_{\hFl_n}$ denote the Schubert class of $H^*_{T_n}(\hFl^\cdot_n)$ indexed by $w$.  Similarly, the Schubert classes $\xi^w_{\hGr_n} \in H^*_{T_n}(\hGr^\cdot_n)$ of $\hGr^\cdot_n$ are indexed by $0$-affine Grassmannian elements $w \in \cS_n^0 := \Z \times \tS_n^0 \subset \cS_n$.  We have
$$
H^*_{T_n}(\hFl^\cdot_n) \cong \bigoplus_{w \in \cS_n} H^*_{T_n}(\pt)\, \xi^w \qquad \text{and} \qquad  H^*_{T_n}(\hGr^\cdot_n) \cong \bigoplus_{w \in \cS_n^0} H^*_{T_n}(\pt)\, \xi^w.
$$

There is a \defn{wrong way map} \cite{Lam, LSS}
$
\varpi: H^*_{T_n}(\hFl_n) \to H^*_{T_n}(\hGr_n)
$
induced by the homotopy equivalences $\Omega SU(n) \cong \hGr_n$ and $LSU(n)/T_n \cong \hFl_n$, and the inclusion $\Omega SU(n) \hookrightarrow LSU(n)/T_n$.  The class $\varpi(\xi)$ is completely determined by its localization at $T_n$-fixed points of $\hGr_n$:
\begin{equation}\label{eq:wrongwayfixedpoints}
\varpi(\xi)|_{t_\la S_n} = \xi|_{t_\la} \qquad \mbox{for $\la \in Q^\vee$.}
\end{equation}

\subsection{Presentations}
We have a ring map $\ev_n: H_{T_{\Z}}^*(\pt) \to H_{T_n}^*(\pt)$ which sets equal $a_i=a_{i+n}$ for all $i\in \Z$. 

The inclusion $\hGr_n \hookrightarrow \Gr$ induces a map of $H_{T_n}^*(\pt)$-algebras:
\begin{equation}\label{eq:Tnpullback}
H^*_{T_{\Z}}(\Gr) \otimes_{\ev_n} H_{T_n}^*(\pt) \to H^*_{T_n}(\hGr_n).
\end{equation}
To explain this, we would like to embed $T_n$ into $T_{\Z}$ in an $n$-periodic manner, but our definition of $T_\Z$ requires all but finitely many entries to be identity.  However, the action of $T_n$ on $\hGr_n$ is compatible with the action of $T_\Z$ on $\Gr$ as follows.  Take $N = mn$ for some positive integer $m$.  If we restrict ourselves to the finite-dimensional piece $\bigcup_k \Gr(k,2N)$ of $\Gr$, then the action of $T_\Z$ factors through $T_{[-N,N-1]}$, and this is the same as the action of $T_n$ on $\hGr_n \cap (\bigcup_k \Gr(k,2N))$ where we embed $T_n$ into $T_{[-N,N-1]}$ in a $n$-periodic manner.  Thus the embedding $\hGr_n \to \Gr$ is ``essentially'' $T_n$-equivariant, and induces \eqref{eq:Tnpullback} by pullback.

Unfortunately, no such map is available for $\hFl_n$.  Nevertheless, we have the following algebraic construction.  For $f \in \bR(x;a)$ and $w \in \tS_n$, we define $f(wa;a)$ analogues to the case $w \in S_\Z$ (see \cite{LaSh2} for details for the case $f \in \Lambda(x||a)$).  Let $\bR(x;a)_{\ev_n} := \bR(x;a) \otimes_{\ev_n} \Q[a_1,\ldots,a_n]$ and $ \Lambda(x||a)_{\ev_n}:=\Lambda(x||a) \otimes_{\ev_n} \Q[a_1,\ldots,a_n]$.

\def\tPsi{\tilde \Psi}

\begin{proposition}\label{prop:phin}
We have a $\Q[a_1,\ldots,a_n]$-algebra morphism $\phi_n: \bR(x;a)_{\ev_n} \to H^*_{T_n}(\hFl_n)$ restricting to $\phi_n:\Lambda(x||a)_{\ev_n} \to H^*_{T_n}(\hGr_n)$, forming commutative diagrams
\begin{equation}\label{eq:affineloc}
\begin{tikzcd}[column sep = large]
\bR(x;a)_{\ev_n} \arrow{d}[left]{{\phi_n}} \arrow[rd] & &\Lambda(x||a)_{\ev_n} \arrow{d}[left]{{\phi_n}} \arrow[rd] &\\
H^*_{T_n}(\hFl_n) \arrow{r}[below]{\text{localization}} & \prod_{w \in \tS_n} H^*_{T_n}(\pt) & H^*_{T_n}(\hGr_n) \arrow{r}[below]{\text{localization}} & \prod_{w \in \tS_n^0} H^*_{T_n}(\pt)
\end{tikzcd}
\end{equation}
where the diagonal arrows are given by $f(x;a) \mapsto (w \mapsto f(wa;a)) \in \Fun(\tS_n,\Q[a_1,\ldots,a_n])$.  
\end{proposition}
\begin{proof}
Let $\tPsi_n \subset \Fun(\tS_n,\Q[a_1,\ldots,a_n])$ denote the image of $H^*_{T_n}(\hFl_n)$ under localization.  It is given by GKM conditions similar to \eqref{E:GKM}.  It is straightforward to check that the generators $x_i$ and $p_k(x||a)$ of $\bR(x;a)_{\ev_n}$ are sent to $\tPsi_n$ under the diagonal map $f(x;a) \mapsto (w \mapsto f(wa;a))$.  Furthermore, this diagonal map is clearly a $\Q[a_1,\ldots,a_n]$ algebra morphism.  This uniquely determines the map $\phi_n$ with the desired properties.
\end{proof}

In fact, the map $\phi_n$ is a surjection and gives a presentation of the cohomologies $H^*_{T_n}(\hFl_n)$ and $H^*_{T_n}(\hGr_n)$.  We shall study these presentations in further detail in \cite{LLS:affine}.

\begin{remark}
The map $\phi_n$ cannot be induced by any continuous map $\hFl_n \to \Fl$ that sends $T_n$-fixed points to $T_\Z$-fixed points.  This is because for any $w \in \tS_n$ and $v \in S_\Z$, one can always find $f(x;a) \in \bR(x;a)$ such that $f(wa;a) \neq f(va;a)$.
\end{remark}


%

\begin{proposition}\label{prop:wrongwaycommute}
We have a commutative diagram
\begin{equation}\label{eq:wrongway}
\begin{tikzcd}
\bR(x;a)_{\ev_n} \arrow[d] \arrow[r, "\eta_a"] & \Lambda(x||a)_{\ev_n} \arrow[d] \\
H^*_{T_n}(\hFl_n) \arrow[r, "\varpi"] & H^*_{T_n}(\hGr_n)
\end{tikzcd}
\end{equation}
\end{proposition}
\begin{proof}
By \eqref{eq:wrongwayfixedpoints} and Proposition \ref{prop:phin}, it suffices to check that for $f(x;a) \in \bR(x;a)_{\ev_n}$ and $\la \in Q^\vee$, we have
\begin{equation} \label{eq:translationeta}
f(t_\la a; a) = \eta_a(f)(t_\la a; a).
\end{equation}
For $f \in \Lambda(x||a)$, a formula for $f(t_\la a; a)$ is given in \cite[Section 4.5]{LaSh2}.  For $p \in \Q[x,a]$, we have $t_\la x_i = x_i + \la_i \delta = x_i$ (since we are working with the finite, or level zero, torus $T_n$ rather than the affine one).  Thus $p(t_\la a; a) = \eta_a(p)$ for $p \in \Q[x,a]$ and \eqref{eq:translationeta} holds.
\end{proof}

\subsection{Small affine Schubert classes}
We shall need the following result from \cite{LLS:affine} concerning ``small" affine Schubert polynomials.

\begin{thm}\label{thm:small}
Suppose that $w \in S_\Z$ (resp. $w \in S_\Z^0$), which we also consider an element of $\tS_n$ (resp. $\tS_n^0$) for $n \gg 0$. 
For sufficiently large $n \gg 0$ the image of $\bS_w(x;a)$ in $H^*_{T_n}(\hFl_n)$ is equal to $\xi^w_{\hFl_n}$ (resp. the image of $s_\la(x||a)$ in $H^*_{T_n}(\hGr_n)$ is equal to $\xi^\la_{\hGr_n}$).
\end{thm}
\begin{proof}
We sketch the proof.  There are divided difference operators  $A_{\bar i}: H^*_{T_n}(\hFl_n) \to H^*_{T_n}(\hFl_n)$ for ${\bar i} \in \Z/n\Z$, and the Schubert classes $\xi^w_{\hFl_n}$ are determined by recurrences similar to \eqref{E:ddiff right}.  One then checks that for Schubert classes indexed by small $w \in S_\Z$, the action of $A_i$ on $\bR(x;a)$ and on $H^*_{T_n}(\hFl_n)$ are compatible: $A_{\bar i} \circ \phi_n = \phi_n \circ A_i$ acting on $\bS_w(x;a)$, when $i \in \Z$ is chosen carefully.
%
It follows that $\bS_w(x;a)$ represents $\xi^w_{\hFl_n}$ for sufficiently large $n$.
\end{proof}
 
 \section{Graph Schubert varieties}
 \subsection{Schubert varieties and double Schur functions}
Fix a positive integer $n$.
Let $\Gr(n,2n)$ denote the Grassmannian of $n$-planes in $\C^{2n} = \sp(e_{1-n},e_{2-n},\ldots,e_{n})$.  We let the torus $T_{2n}=(\C^\times)^{2n}$ act on $\C^{2n}$, and identify $H^*_{T_{2n}}(\pt) = \Q[a_{1-n},a_{2-n},\ldots,a_{n}]$, so that the weight of the basis vector $e_i \in \C^{2n}$ is equal to $a_i$.  The $T$-fixed points of $\Gr(n,2n)$ are the points $e_I \in \Gr(n,2n)$, where $I$ is an $n$-element subset $I \subset [1-n,n]$.  There is a bijection from partitions $\lambda$ fitting in a $n \times n$ box to the collection $\binom{[1-n,n]}{n}$ of subsets of size $n$ in the interval $[1-n,n]$, given by 
$\lambda \mapsto I(\lambda)= ([1,n] \setminus S_+) \cup S_-$, 
where $\lambda = \lambda(S_-,S_+)$; see \textsection \ref{SS:notation}. The Schubert variety $X^\lambda$ has codimension $|\lambda|$ and contains the $T$-fixed points $e_{I(\mu)}$ for $\mu \supseteq \lambda$. Via the forgetful map $\Q[a]\to H_{T_{2n}}^*(\pt)$ which sets $a_i$ to $0$ for $i\notin [1-n,n]$, $H_{T_{2n}}^*(\pt) \otimes_{\Q[a]} \La(x||a)$ has a $\Q[a_{1-n},\dotsc,a_n]$-algebra structure.\footnote{Let $\lambda^c$ denote the partition that is the complement of $\lambda$ in the $n \times n$ square.  Our $X^\lambda$ is equal to Knutson's $X^{\lambda^c}$ \cite{Knu}.} 

\begin{proposition}\label{prop:doubleSchur}
There is a surjection 
\begin{equation}\label{eq:Grass}
H_{T_{2n}}^*(\pt) \otimes_{\Q[a]} \Lambda(x||a) \mapsto H^*_{T_{2n}}(\Gr(n,2n))
\end{equation}
of $\Q[a_{1-n},a_{2-n},\ldots,a_{n}]$-algebras such that $s_\lambda(x||a) \mapsto [X^{\lambda}]$.
\end{proposition}
The surjection \eqref{eq:Grass} is compatible with localization, analogous to \eqref{eq:affineloc}.

\begin{remark} \label{R:Grass truncation}
Let $S_n$ act on the $x$-variables in $R=\Q[x_{1-n},\dotsc,x_{-1},x_0,a_{1-n},\dotsc,a_n]$. We may realize $H_{T_{2n}}^*(\Gr(n,2n))$ as a quotient of $R^{S_n}$. The map of Proposition \ref{prop:doubleSchur} is given by sending $s_\la(x||a)$ to the truncation $\S_{w_\la}^{[1-n,n]}(x;a)$.
\end{remark}


\subsection{The graph Schubert class}
\label{SS:graph Schub}
We describe Knutson's graph Schubert variety. 
Let $w\in S_n$. Let $M^\circ_w=B_{-}wB_+\subset M_{n\times n}$ and $M_w=\overline{M^\circ_w}\subset M_{n\times n}$ be the matrix Schubert variety. Let $V^\circ_w= (I_n |M^\circ_w) \subset M_{n\times 2n}$ where we place the $n\times n$ identity matrix side by side with $M^\circ_w$. Let $\pi: M^\circ_{n\times 2n} \rightarrow \Gr(n,2n)$ be the projection to $\Gr(n,2n)$ from the rank $n$ matrices $M^\circ_{n\times 2n} $ in $M_{n\times 2n}$.
The \defn{graph Schubert variety} $G(w)$ is given by
$$G(w)=\overline{\pi(V^\circ_w)} \subset \Gr(n,2n).$$

Define $\tilde f_w \in \cS_n$  by $\tilde f_w(i)= w(i)+n$ for $1\leq i \leq n$ and $\tilde f_w(i)=i+n$ for $n+1\leq i \leq 2n$.  Then $G(w)$ is equal to the \defn{positroid variety} $\Pi_{\tilde f_w}$ (see Section 6 in \cite{KLS}).  Let $[G(w)] \in H^*_{T_{2n}}(\Gr(n,2n))$ denote the torus-equivariant cohomology class of $G(w)$.

Define the \defn{$n$-rotated double Stanley symmetric function} $F^{(n)}_w(x||a) \in \Lambda(x||a)$ as the image of $\bS_w(x;a)$ under the map of $\Q[a]$-algebras
\begin{equation}
\label{eq:nrotate}
\Lambda(x||a)\otimes_{\bQ[a]}\bQ[x,a] \rightarrow \Lambda(x||a)
\end{equation}
which is the identity on $\Lambda(x||a)$ and sends $x_i \in \bQ[x,a]$ to $a_{i-n}$.

\begin{thm}\label{thm:graphSchubStanley}
Under \eqref{eq:Grass}, the image of $F_{w}^{(n)}(x||a)$ in $H^*_{T_{2n}}(\Gr(n,2n))$ is equal to $[G(w)]$.
\end{thm}

\subsection{Proof of Theorem \ref{thm:graphSchubStanley}}
There is an embedding $\iota: \Gr(n,2n) \rightarrow \hGr^{(n)}_{2n}$, placing the Grassmannian as a Schubert variety at the ``bottom'' of the affine Grassmannian of $\GL_{2n}$. 
 This induces a pullback back map $\iota^*: H^*_{T_{2n}}(\hGr^{(n)}_{2n}) \to H^*_{T_{2n}}(\Gr(n,2n))$.  There is also the wrongway map of rings $\varpi: H^*_{T_{2n}}(\hFl^{(n)}_{2n}) \to H^*_{T_{2n}}(\hGr^{(n)}_{2n})$.

For a \defn{bounded affine permutation} $f$, let $[\Pi_f] \in H^*_{T_{2n}}(\Gr(n,2n))$ denote its equivariant cohomology class, and let $\xi^f \in H^*_{T_{2n}}(\hFl^{(n)}_{2n})$ denote the Schubert class.  The following result is due to Knutson-Lam-Speyer \cite{KLS} (see also He and Lam \cite{HL}). 

\begin{thm}\label{thm:KLS}
For any positroid variety $\Pi_f$, we have $\iota^* \circ \varpi(\xi^f) = [\Pi_f]$.
\end{thm}

In particular, this result holds for $\Pi_f = \Pi_{\tilde f_w} = G(w)$.  The remainder of the proof is concerned with working through the interpretation of Theorem \ref{thm:KLS} in terms of double symmetric functions.

Let us first consider $\xi^{\tf_w} \in H^*_{T'_{2n}}(\hFl^{(n)}_{2n})$.  Here, we use $T'_{2n}$ to distinguish from $T_{2n}$.  We have $H^*_{T'_{2n}}(\pt) = \Q[a_1,a_2,\ldots,a_{2n}]$ but $H^*_{T_{2n}}(\pt) = \Q[a_{1-n},a_{2-n},\ldots,a_n]$.  Recall from Proposition \ref{prop:phin} the algebra map $\phi_n: \bR(x;a)_{\ev_n} \to H^*_{T_{2n}}(\hFl_{2n})$.  Combining Theorem \ref{thm:small} with Proposition \ref{prop:shiftflag} (and the analogue of Proposition \ref{prop:shiftflag} for $\hFl^{(n)}_{2n}$), we obtain $\phi_n(\bS_{\sh^n w}(x;a)) = \xi^{\tf_w}_{\hFl^{(n)}_{2n}}$. 
By Proposition \ref{prop:wrongwaycommute}, the class $\varpi(\xi^{\tf_w}) \in H^*_{T'_{2n}}(\Gr^{(n)}_{2n})$ is the image under $\phi_n$ of the element $\eta_a(\bS_{\sh^n w}(x;a)) \in \Lambda(x||a)$.  

Finally, we need to switch from $T'_{2n}$ back to the isomorphic torus $T_{2n}$.  This is simply the map $a_i \mapsto a_{i-n}$ on $\Q[a]$.  Thus
$$
\shift^{-n}_a \left( \bS_{\sh^n w}(x;a)|_{x_i \mapsto a_i}\right) = F^{(n)}_w(x||a) = \varpi(\xi^{\tf_w}) \in H^*_{T_{2n}}(\Gr^{(n)}_{2n}).
$$
Theorem \ref{thm:graphSchubStanley} follows from this equality and Theorem \ref{thm:KLS}.

\subsection{Proof of Theorem \ref{thm:decomp}}\label{ssec:decompproof}
By the main result of \cite{Knu} applied to the interval positroid variety $G(w)$, we have the expansion
$$
[G(w)] = \sum_D \wt(D) [X^{\lambda(D)}]
$$
where the sum is over all \defn{IP pipedreams} $D$ for $G(w)$ that live in the triangular region $\{(i,j) \mid 1 \leq i \leq j \leq 2n\}$.  We do not give the full definition of IP pipedream here.  Indeed, for the special case of $G(w)$, the IP pipedreams are in a canonical bijection with rectangular $w$-bumpless pipedreams.

Let $P$ be a rectangular $w$-bumpless pipedream.  We produce an IP pipedream $D$ as follows:
\begin{enumerate}
\item erase all boxes in the lower-triangular part of the left $n \times n$ square of $P$ (these boxes always contain vertical pipes);
\item add an upper-triangular part below the right $n \times n$ square of $P$, and fill with vertical pipes;
\item rename the pipes numbered $1,2,\ldots,n$ to the letters $A_1,A_2,\ldots,A_n$;
\item rename the nonpositively numbered pipes to the label $1$;
\item add 0 pipes so that every tile has two pipes (in an empty tile, we use a double elbow).
\end{enumerate}
In Figure~\ref{fig:IP}, all 0 pipes are black, all 1 pipes are red, and lettered pipes are blue.  

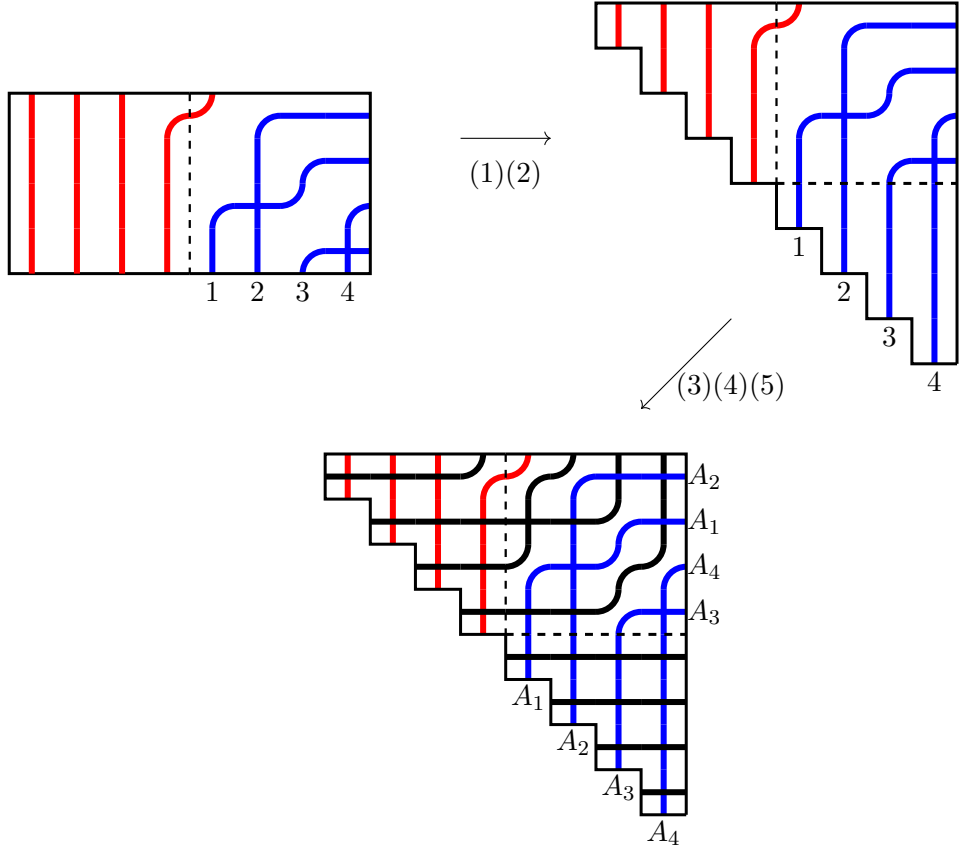
\begin{figure}
\begin{center}
\begin{tikzpicture}[scale=0.6,line width=0.8mm]
\rightelbow{1}{0}{blue}
\horline{2}{0}{blue}
\horline{3}{0}{blue}

\vertline{1}{-1}{blue}
\rightelbow{2}{-1}{blue}
\horline{3}{-1}{blue}

\rightelbow{0}{-2}{blue}
\cross{1}{-2}{blue}{blue}
\leftelbow{2}{-2}{blue}
\rightelbow{3}{-2}{blue}

\vertline{0}{-3}{blue}
\vertline{1}{-3}{blue}
\rightelbow{2}{-3}{blue}
\cross{3}{-3}{blue}{blue}
\vertline{-4}{0}{red}

\vertline{-3}{0}{red}
\vertline{-3}{-1}{red}

\vertline{-2}{0}{red}
\vertline{-2}{-1}{red}
\vertline{-2}{-2}{red}

\rightelbow{-1}{0}{red}
\vertline{-1}{-1}{red}
\vertline{-1}{-2}{red}
\vertline{-1}{-3}{red}

\leftelbow{0}{0}{red}

\draw[line width=0.4mm] (0,-3)--(4,-3)--(4,1)--(0,1);
\draw[dashed,line width=0.3mm] (0,1)--(0,-3);
\draw[line width=0.4mm] (0,1)--(-4,1)--(-4,-3)--(0,-3);
\vertline{-4}{-1}{red}
\vertline{-4}{-2}{red}
\vertline{-4}{-3}{red}
\vertline{-3}{-2}{red}
\vertline{-3}{-3}{red}
\vertline{-2}{-3}{red}

\draw (0.5,-3.4) node  {$1$};
\draw (1.5,-3.4) node  {$2$};
\draw (2.5,-3.4) node  {3};
\draw (3.5,-3.4) node  {4};
\draw[->,thin] (6,0)--(8,0);
\draw (7,-0.8) node {${(1) (2)}$};
\draw[->,thin] (12,-4)--(10,-6);
\draw (12,-5.5) node {${(3) (4)(5)}$};

\begin{scope}[shift={(13,2)}]
\rightelbow{1}{0}{blue}
\horline{2}{0}{blue}
\horline{3}{0}{blue}

\vertline{1}{-1}{blue}
\rightelbow{2}{-1}{blue}
\horline{3}{-1}{blue}

\rightelbow{0}{-2}{blue}
\cross{1}{-2}{blue}{blue}
\leftelbow{2}{-2}{blue}
\rightelbow{3}{-2}{blue}

\vertline{0}{-3}{blue}
\vertline{1}{-3}{blue}
\rightelbow{2}{-3}{blue}
\cross{3}{-3}{blue}{blue}
\vertline{-4}{0}{red}

\vertline{-3}{0}{red}
\vertline{-3}{-1}{red}

\vertline{-2}{0}{red}
\vertline{-2}{-1}{red}
\vertline{-2}{-2}{red}

\rightelbow{-1}{0}{red}
\vertline{-1}{-1}{red}
\vertline{-1}{-2}{red}
\vertline{-1}{-3}{red}

\leftelbow{0}{0}{red}

\vertline{0}{-4}{blue}
\vertline{1}{-4}{blue}
\vertline{2}{-4}{blue}
\vertline{3}{-4}{blue}

\vertline{1}{-5}{blue}
\vertline{2}{-5}{blue}
\vertline{3}{-5}{blue}

\vertline{2}{-6}{blue}
\vertline{3}{-6}{blue}

\vertline{3}{-7}{blue}

\draw[line width=0.4mm,dashed] (0,-3)--(4,-3);
\draw[line width=0.4mm] (4,-7)--(4,1)--(0,1);
\draw[dashed,line width=0.3mm] (0,1)--(0,-3);
\draw[line width=0.4mm] (0,1)--(-4,1)--(-4,0)--(-3,0)--(-3,-1)--(-2,-1)--(-2,-2)--(-1,-2)--(-1,-3)--(0,-3)--(0,-4)--(1,-4)--(1,-5)--(2,-5)--(2,-6)--(3,-6)--(3,-7)--(4,-7);

\draw (0.5,-4.4) node  {$1$};
\draw (1.5,-5.4) node  {$2$};
\draw (2.5,-6.4) node  {3};
\draw (3.5,-7.4) node  {4};
\end{scope}

\begin{scope}[shift={(7,-8)}]
\elbow{1}{0}{blue}{black}
\cross{2}{0}{black}{blue}
\cross{3}{0}{black}{blue}

\cross{0}{-1}{black}{black}
\cross{1}{-1}{blue}{black}
\elbow{2}{-1}{blue}{black}
\cross{3}{-1}{black}{blue}

\elbow{0}{-2}{blue}{black}
\cross{1}{-2}{blue}{blue}
\elbow{2}{-2}{black}{blue}
\elbow{3}{-2}{blue}{black}

\cross{0}{-3}{blue}{black}
\cross{1}{-3}{blue}{black}
\elbow{2}{-3}{blue}{black}
\cross{3}{-3}{blue}{blue}
\cross{-4}{0}{red}{black}

\cross{-3}{0}{red}{black}
\cross{-3}{-1}{red}{black}

\cross{-2}{0}{red}{black}
\cross{-2}{-1}{red}{black}
\cross{-2}{-2}{red}{black}

\elbow{-1}{0}{red}{black}
\cross{-1}{-1}{red}{black}
\cross{-1}{-2}{red}{black}
\cross{-1}{-3}{red}{black}

\elbow{0}{0}{black}{red}

\cross{0}{-4}{blue}{black}
\cross{1}{-4}{blue}{black}
\cross{2}{-4}{blue}{black}
\cross{3}{-4}{blue}{black}

\cross{1}{-5}{blue}{black}
\cross{2}{-5}{blue}{black}
\cross{3}{-5}{blue}{black}

\cross{2}{-6}{blue}{black}
\cross{3}{-6}{blue}{black}

\cross{3}{-7}{blue}{black}

\draw[line width=0.4mm,dashed] (0,-3)--(4,-3);
\draw[line width=0.4mm] (4,-7)--(4,1)--(0,1);
\draw[dashed,line width=0.3mm] (0,1)--(0,-3);
\draw[line width=0.4mm] (0,1)--(-4,1)--(-4,0)--(-3,0)--(-3,-1)--(-2,-1)--(-2,-2)--(-1,-2)--(-1,-3)--(0,-3)--(0,-4)--(1,-4)--(1,-5)--(2,-5)--(2,-6)--(3,-6)--(3,-7)--(4,-7);

\draw (0.5,-4.4) node  {$A_1$};
\draw (1.5,-5.4) node  {$A_2$};
\draw (2.5,-6.4) node  {$A_3$};
\draw (3.5,-7.4) node  {$A_4$};
\draw (4.4,0.5) node {$A_2$};
\draw (4.4,-0.5) node {$A_1$};
\draw (4.4,-1.5) node {$A_4$};
\draw (4.4,-2.5) node {$A_3$};
\end{scope}

\end{tikzpicture}
\end{center}
\caption{From a bumpless pipedream to an IP pipedream.}
\label{fig:IP}
\end{figure}

Going through the definition of IP pipedream in \cite{Knu}, we see that they are in bijection with rectangular $w$-bumpless pipedreams.  Comparing $\wt(D)$ with $\wt(P)$, it follows from Proposition \ref{prop:doubleSchur} and Theorem \ref{thm:graphSchubStanley} that in $H^*_{T_{2n}}(\Gr(n,2n))$ we have
\begin{equation}\label{eq:Fn}
F^{(n)}_w(x||a) = \sum_P \wt^{(n)}(P) s_{\lambda(P)}(x||a),
\end{equation}
where the summation is over all rectangular $w$-bumpless pipedreams, and $\wt^{(n)}(P) = \wt(P)|_{x_i \mapsto a_{i-n}}$.  But we have injections $S_n \hookrightarrow S_{n+1} \hookrightarrow \cdots$.  The rectangular $S_{n+1}$-bumpless pipedreams $P'$ for $w$ are obtained from the rectangular $S_{n}$-bumpless pipedreams $P$ for $w$ by (1) adding an elbow in the southeastern most corner, (2) filling the rest of the southmost row with vertical pipes, and (3) filling the rest of the eastmost column with horizontal pipes.  Thus, \eqref{eq:Fn} holds for all sufficiently large $n$, where the summation is over the same set of rectangular $w$-bumpless pipedreams.  The only expansion of $\bS(x;a)$ in terms of $s_\lambda(x||a)$ consistent with this is the one in Theorem \ref{thm:decomp}.

%

\subsection{Divided difference formula for graph Schubert class}
For completeness, we include the following formula due to Allen Knutson.  

\begin{thm}
Let $w \in S_n$.  Then
$$
[G(w)] = A_{w_0}\left(\left(\prod_{1-n\le i<j\le 0} (x_i-a_j)\right) \gamma_x^{-n} \S_w(\xp;\ap)\right).
$$
where the action of $A_{w_0}$ is defined by the action of $S_n$ on the variables $x_{1-n},\dotsc,x_{-1},x_0$.
\end{thm}
\begin{proof}[Sketch of proof] We use the notation of \textsection \ref{SS:graph Schub}. There is a canonical projection
$$H_{GL_n\times T_{2n}}^*(M_{n\times 2n}) \to 
H_{GL_n\times T_{2n}}^*(M_{n\times 2n}^\circ) \cong
H_{T_{2n}}^*(\Gr(n,2n)).$$ By \cite{BF} this map has a section
$\sigma: H_{T_{2n}}^*(\Gr(n,2n)) \to H_{GL_n \times T_{2n}}^*(M_{n\times 2n})$ such that for any closed subscheme $Z \subset \Gr(n,2n)$,
$\sigma([Z]) = [\overline{\pi^{-1}(Z)}]$. In particular  $\sigma([X^\la]) = [\overline{\pi^{-1}(X^\la)}]$ which is identified with the double Schur polynomial $\S_{w_\la}^{[1-n,n]}(x;a)$ in variables $x_{1-n},\dotsc,x_0$ and $a_{1-n},\dotsc,a_n$ where the row torus $T_n\subset GL_n$ acts on $M_{n\times 2n}$ by the weights $x_{1-n}$ through $x_0$ and $T_{2n}$ acts on columns by weights $a_{1-n}$ through $a_n$. Let $Z=G(w)$ and $Y=\overline{\pi^{-1}(G(w))}\subset M_{n\times 2n}$. In the notation of \textsection \ref{SS:graph Schub} we have
$\sigma([G(w)]) =  [Y]$. Let $Y'$ be the closed $B_-$-stable subvariety $(\overline{B_-}|M_w)$ of $M_{n\times 2n}$.
Since $M_w^\circ$ is $B_-$-stable we have
\begin{align*}
Y &= \overline{GL_n\cdot (I|M_w^\circ)} = \overline{B_+ B_- (I|M_w^\circ)} 
= \overline{B_+ (B_-|M_w^\circ)} 
= \overline{B_+ Y'}.
\end{align*}
Since $B_+$ acts freely on $(B_-|M_w^\circ)$ one may show that
$[Y] = A_{w_0} [Y']$ where $[Y'] \in H_{T_n\times T_{2n}}^*(M_{n\times 2n})$. But $Y'$ is a product. The equivariant class of the affine space $\overline{B_-}$ is the product of the weights of the matrix entries that are set to zero in $\overline{B_-}$ and the equivariant class of $M_w$, which is $\gamma_x^{-n}\S_w(\xp;\ap)$ by \cite{KM} (the shift in $x$ variables is due to the convention on weights). We deduce that
\begin{align*}
	[Y'] &= \left(\prod_{1-n \le i<j\le 0} (x_i-a_j)\right) \gamma_x^{-n}(\S_w(x;a)) 
\end{align*}
as required.
\end{proof}

\begin{example} Let $n=2$ and $w=s_1$. Then $\S_w(\xp;\ap)=x_1-a_1$, $\gamma_x^{-n}(\S_w(\xp;\ap)) = x_{-1}-a_1$ and
\begin{align*}
	\sigma([G(w)]) &= A_{-1}((x_{-1}-a_0)(x_{-1}-a_1)) \\
	&= x_{-1}+x_0 - a_0 - a_1 \\
	&= (x_{-1}+x_0-a_{-1}-a_0) + (a_{-1}-a_1) \\
	&= \sigma([X^{\tableau[pby]{\\ }}]) +(a_{-1}-a_1) \sigma([X^{\varnothing}]).
\end{align*}
On the other hand, we have $\bS_w=s_1(x||a) + (x_1-a_1)$.
Setting $x_1 \mapsto a_{-1}$, the formula for $F^{(2)}_{s_1}(x||a)$ agrees
with the above computation.
\end{example}

\appendix
\section{Dictionary between positive and nonpositive alphabets}
\label{S:pos to non}
The literature uses double Schur symmetric functions $s^{>0}_\la(x||a)$ (e.g. \cite[\textsection 2.1]{M}) while we use $s^{\le0}_\la(x||a)$.
The two kinds of double Schurs are compared explicitly below using localization. For more connections with various kinds of double Schur polynomials used in the literature, see \cite[\textsection 2.1]{M}.

\subsection{Positive alphabets}
Recall that $\xp=(x_1,x_2,\dotsc)$ and $\xm=(x_0,x_{-1},\dotsc)$
and similarly for $\ap$ and $\am$.

Let $\bQ[a]=\bQ[a_i\mid i\in\Z]$ and $\La^{>0}(x||a)$ the polynomial $\bQ[a]$-algebra generated by $p_k(\xp/\ap)$ for $k\ge 1$. Recall the definition of $\shift_a$ from \eqref{E:phi_a_def}. Define
\begin{align*}
  h_r^{>0}(x||a) &= \shift_a^{1-r} h_r(\xp/\ap) &
  s_\la^{>0}(x||a) &= \det \shift^{j-1}_a h^{>0}_{\la_i-i+j}(x||a).
 \end{align*}

\subsection{Nonpositive alphabets}
Let $\La^{\le0}(x||a)$ be the polynomial $\bQ[a]$-algebra with generators  $p_k(\xm/\am)$ for $k\ge 1$. Define
\begin{align}
  h_r^{\le0}(x||a) &= \shift^{r-1}_a h_r(\xm/\am) &
 \label{E:nonpos}
  s_\la^{\le0}(x||a) &= \det \shift_a^{1-j} h^{\le0}_{\la_i-i+j}(x||a).
\end{align}
Applying $\omega$ and using \eqref{E:shift and negate} we have
\begin{align}
  e_r^{\le0}(x||a) &= \shift_a^{1-r} e_r(\xm/\am)  &
 \label{E:nonpostranspose}
  s_\la^{\le0}(x||a) &= \det \shift_a^{j-1} e^{\le0}_{\la'_i-i+j}(x||a).
\end{align}

\subsection{Localization}

\begin{prop}\label{P:postonot} 
Let $\Phi:\La^{>0}(x||a)\to \La^{\le0}(x||a)$ be the $\bQ[a]$-algebra isomorphism given by
\begin{align}
\label{E:Phi}
p_k(\xp/\ap) \mapsto - p_k(\xm/\am)\qquad\text{for all $k\ge 1$.}
\end{align}
 It satisfies
\begin{align}\label{E:Phi res}
\Phi(f)|_w = f|_w \qquad\text{for all $f\in \La^{>0}(x||a)$ and $w\in S_{\Z}$.}
\end{align}
Moreover,
\begin{align}\label{E:Phis}
  \Phi(s_\la^{>0}(x||a)) = (-1)^{|\la|} s^{\le0}_{\la'}(x||a)\qquad\text{for all $\la\in\Par$.}
\end{align}
\end{prop}
\begin{proof} Checking \eqref{E:Phi res} on algebra generators, we have
\begin{align*}
 p_k(\xp/\ap)|_w + p_k(\xm/\am)|_w 
= p_k(w\ap/\ap) + p_k(w \am/\am) 
= p_k(w a_{\Z} / a_{\Z}) = p_k(a_{\Z}/a_{\Z}) = 0.
\end{align*}
Since $\Phi$ acts like the antipode (up to changing nonpositive for positive alphabets), we have the equality $\Phi(s_\la(\xp/\ap)) = (-1)^{|\la|} s_{\la'}(\xm/\am)$ for all $\la\in\Par$.
It is straightforward to verify that $\Phi$ is $\shift_a$-equivariant: $ \Phi(\shift_a(f)) = \shift_a(\Phi(f))$ for all $f\in \La^{>0}(x||a)$.
We compute
\begin{align*}
  \Phi(s_\la^{>0}(x||a)) &= \det \Phi(\shift_a^{j-1}( h^{>0}_{\la_i-i+j}(x||a))) \\
  &= \det \Phi ( \shift_a^{j-(\la_i-i+j)}h_{\la_i-i+j}(\xp/\ap) ) \\
  &= \det \shift_a^{i-\la_i} (-1)^{\la_i-i+j} e_{\la_i-i+j}(\xm/\am) \\
  &= (-1)^{|\la|} \det \shift_a^{j-1} e_{\la_i-i+j}^{\le0}(x||a) \\
  &= (-1)^{|\la|} s_{\la'}^{\le0}(x||a). \qedhere
\end{align*}
\end{proof}

\subsection{Molev's skew double Schur functions}\label{SS:Molev skew double} Molev's skew double Schur functions \cite{M2} \cite{M} \cite{ORV} are the positive variable analogues of double Stanley functions for 321-avoiding permutations.

For $\la\in\Par$ and $n \ge \ell(\la)$ the double Schur polynomial, 
$s_\la(x_1,\dotsc,x_n||a)$ may be defined by $\S_{\shift^n(w_\la)}$.
It is stable (the limit as $n\to\infty$ is well-defined),
yielding the element $s_\la^{>0}(x||a) \in \La^{>0}(x||a)$.

The same is true of Molev's skew double Schur polynomials
$s_{\nu/\mu}(x_1,\dotsc,x_n||a)$ as defined in  \cite[(2.20)]{M}, because they have a stable expansion into double Schur polynomials as $n\to\infty$. Define $F_{\nu/\mu}^{>0}(x||a)\in \La^{>0}(x||a)$ by $F_{\nu/\mu}^{>0}(x||a) := \lim_{n\to\infty} s_{\nu/\mu}(x_1,\dotsc,x_n||a)$.

Recalling $w_{\la/\mu}$ from \eqref{E:w skew}, we have
$
\omega(w_{\nu/\mu}) = w_{\nu'/\mu'}
$.
We define 
$
F^{\le0}_{\nu/\mu}(x||a) := F^{\le0}_{w_{\nu/\mu}}(x||a)
$.

\begin{prop} \label{P:skew double is 321 double Stanley}
With $\Phi$ as in Proposition \ref{P:postonot},
\begin{align*}
  \Phi(F^{>0}_{\nu/\mu}(x||a)) = (-1)^{|\nu|-|\mu|} F^{\le0}_{\nu'/\mu'}(x||a).
\end{align*}
\end{prop}

\section{Schubert Inversion} \label{A:Schub inversion}
\subsection{Proof of Lemma \ref{L:Schub cancellation}}
\begin{proof}
We expand using the Billey-Jockusch-Stanley formula \eqref{E:BJS}:
\begin{align*}
	&\sum_{w \doteq uy} (-1)^{\ell(y)} \S_{u^{-1}} (\xp)\S_{y}(\xp) \\
	&= \sum_{a_1a_2 \cdots a_\ell \in \Red(w)} \left(\sum_{k=0}^{\ell+1} (-1)^k \sum_{\substack{b_1 \geq b_2 \geq \cdots \geq b_k \geq 1 \leq b_{k+1} \leq b_{k+2} \leq \cdots \leq b_\ell\\ \text{ if } i > k \text{ then } a_i<a_{i+1} \implies b_i < b_{i+1} \\ \text{ if } i < k \text{ then } a_i > a_{i+1} \implies b_i > b_{i+1} \\ b_i \leq a_i}} x_{b_1} x_{b_2} \cdots x_{b_\ell} \right).
\end{align*}
We perform a sign-reversing involution on the inner sum on the RHS (contained inside the parentheses) as follows.  If either ($k > 0$ and $b_k < b_{k+1}$) or $k = \ell$, then we change $k$ to $k-1$.  If either ($k < \ell$ and $b_k > b_{k+1}$) or $k = 0$, then we change $k$ to $k + 1$.  If $0 < k < \ell$ and $b_k = b_{k+1}$, then we change $k$ to $k-1$ if $a_k < a_{k+1}$; we change $k$ to $k+1$ if $a_k > a_{k+1}$. 
\end{proof}

\subsection{Inverting systems with Schubert polynomials as change-of-basis matrix}

Let $W \subset S_{\ne0}$ be a subgroup generated by simple reflections $s_i$ for $i\in I$ for some $I\subset \Z-\{0\}$. For $J\subset I$ let $W_J$ be the subgroup of $W$ generated by $s_i$ for $i\in J$.
For $x,y\in W$ say $x \leleft{J} y$ if $yx^{-1}\in W_J$ and $\ell(yx^{-1})+\ell(x)=\ell(y)$. Equivalently, $x\leleft{J} y$ if and only if there is a $v\in W_J$ such that $y\doteq vx$.

\begin{lem} Let $W'$ be a fixed coset of $W_J\backslash W$. Then the $W'\times W'$-matrices
\begin{align*}
	A_{x,y} &= (-1)^{\ell(yx^{-1})} \chi(x\leleft{J} y) \Schub_{xy^{-1}}(a) \\
	B_{x,y} &= \chi(x\leleft{J} y) \Schub_{yx^{-1}}(a).
\end{align*}
are mutually inverse.
\end{lem}
\begin{proof} For $x,y\in W'$, we have
\begin{align*}
	(AB)_{xy} &= \sum_{z\in W'} A_{xz} B_{zy} \\
	&= \sum_z \chi(x \leleft{J} z) \chi(z \leleft{J} y) (-1)^{\ell(zx^{-1})} \Schub_{xz^{-1}}(a) \Schub_{yz^{-1}}(a).
\end{align*}
Thus $(AB)_{xy}=0$ unless $x\leleft{J} y$. Let us assume this.
Let $u,v\in W_J$ be such that $ux=z$ and $vz=y$. There are factorizations $y\doteq vz$ and $y \doteq vux$ with
\begin{align*}
  (AB)_{xy} &= \sum_{vu=yx^{-1}} (-1)^{\ell(u)} \Schub_{u^{-1}}(a) \Schub_v(a) = \delta_{x,y}
\end{align*}
using the obvious generalization of Lemma \ref{L:Schub cancellation} to $S_{\ne0}$.
\end{proof}

The right hand analogue also holds. For $x,y\in W$ say $x \leright{J} y$ if $x^{-1}y\in W_J$ and $\ell(x)+\ell(x^{-1}y)=\ell(y)$. Equivalently, $x\leright{J} y$ if and only if there is a $v\in W_J$ such that $y\doteq xv$.

\begin{lem} Let $W'$ be a fixed coset of $W/W_J$. The $W' \times W'$-matrices
	\begin{align*}
		A_{x,y} &= (-1)^{\ell(x^{-1}y)} \chi(x\leright{J} y) \Schub_{y^{-1}x}(a) \\
		B_{x,y} &= \chi(x \leright{J} y) \Schub_{x^{-1}y}(a)
	\end{align*}
	are inverses.
\end{lem}

\begin{cor} \label{C:SchubInversion}
Let $\{F_w \mid w\in W\}$ and $G_w\mid w\in W\}$ be families of elements. Then
\begin{enumerate}
\item[(a)] We have
\begin{align} \label{E:LFtoG}
	F_w &= \sum_{\substack{w\doteq uv \\ (u,v)\in W_J\times W }} (-1)^{\ell(u)} \Schub_{u^{-1}}(a) G_v&\qquad&\text{for all $w\in W$}
\intertext{if and only if}
\label{E:LGtoF}
	G_w &= \sum_{\substack{w\doteq uv \\ (u,v)\in W_J\times W }} \Schub_u(a) F_v&\qquad&\text{for all $w\in W$.}
\end{align}
\item[(b)] We have
\begin{align} \label{E:RFtoG}
	F_w &= \sum_{\substack{w\doteq \\ (v,z)\in W\times W_J }} (-1)^{\ell(z)} G_v \Schub_{z^{-1}}(a) &\qquad&\text{for all $w\in W$}
	\intertext{if and only if}
	\label{E:RGtoF}
	G_w &= \sum_{\substack{w\doteq vz \\ (v,z)\in W\times W_J}} F_v \Schub_z(a) &\qquad&\text{for all $w\in W$.}
\end{align}
\end{enumerate}
\end{cor}

\section{Level zero affine nilHecke ring}
\label{A:affinenilHecke}
We recall in this section standard results the affine nilHecke algebra and the Peterson subalgebra.  We use affine symmetric group notation from \textsection\ref{ssec:tSn}.

%

\subsection{Level zero affine nilHecke ring}
Let $\ttA$ denote the level zero affine nilHecke ring (see for example \cite{LaSh} for details).  It has $\Q[a_1,a_2,\ldots,a_n]$-basis $\{A_w \mid w \in \tS_n\}$.  There is an injection $\tS_n \hookrightarrow \ttA$ that is a group isomorphism onto its image.  It is given by $s_i \mapsto 1-\alpha_i A_i = 1-(a_{i+1}-a_i)A_{s_i}$.  The image of $\tS_n$ in $\ttA$ forms a basis of $\ttA$ over $\Q(a_1,a_2,\ldots,a_n)$.

The action of $\tS_n$ on $\Q[a_1,\ldots,a_n]$ is the level 0 action.  Thus in $\ttA$ we have the commutation relation
\begin{equation}\label{eq:levelzero}
(w t_\la) p = (w \cdot p) (w t_\la)
\end{equation}
for $p \in \Q[a_1,\ldots,a_n]$ and $w \in S_n$.  In particular, $t_\la \in Z_\ttA(\Q[a_1,\ldots,a_n])$.

The affine nilHecke ring $\ttA$ has a coproduct map $\Delta: \ttA \to \ttA \otimes_{\Q[a_1,\ldots,a_n]} \ttA$ which is $\Q[a_1,\ldots,a_n]$-linear and satisfies
\begin{equation}\label{eq:affinecoprod}
\Delta(w) = w \otimes w \qquad \mbox{for $w \in \tS_n$.}
\end{equation}

\subsection{Peterson algebra}

Let $\tP := Z_\ttA(\Q[a_1,\ldots,a_n])$ denote the Peterson subalgebra of $\ttA$, defined as the centralizer of $\Q[a_1,\ldots,a_n]$ inside $\ttA$.  Then $\tP$ has basis $\{t_\la \mid \la \in Q^\vee\}$ over $\Q(a_1,\ldots,a_n)$. 

\begin{thm}\label{thm:aPeterson}
The Peterson subalgebra $\tP$ is a commutative subalgebra of $\ttA$.  It is a free $\Q[a_1,\ldots,a_n]$-module with basis $\{\tj_\lambda \mid \la \in Q^\vee\}$.  The element $\tj_\la \in \tP$ is uniquely characterized by the expansion
$$
\tj_\la = A_w + \sum_{u \notin \tS_n^0} \tj_\la^u A_u
$$
for $\tj_\la^u \in \Q[a_1,\ldots,a_n]$, where $wS_n = t_\la S_n$.
\end{thm}

The following result follows from combining \cite{LaSh:QH}, which proves Peterson's isomorphism of localizations of $H^*(\hGr)$ and the equivariant quantum cohomology $H_{T_n}^*(\Fl_n)$ together with an explicit correspondence of Schubert classes, and the positivity result of 
\cite{Mi} in equivariant quantum cohomology.

\begin{thm}\label{thm:aPositive}
Let $\la \in Q^\vee$ and $u \in \tS_n$.  Then $\tj_\la^u \in \Z_{\geq 0}[a_i - a_j| 1\leq i < j \leq n]$.
\end{thm}

\end{document}